\documentclass[12pt,a4paper]{book}

\usepackage{amssymb}
\usepackage{amsmath}
\usepackage[all]{xy}
\usepackage{latexsym}
\usepackage{mathtools}
\usepackage{amsthm}
\usepackage{color}
\usepackage{makeidx}
\usepackage{fancyhdr}
\usepackage{hyperref}
\usepackage[total={145mm,215mm}]{geometry}
\input diagxy

\theoremstyle{plain}
  \newtheorem{thm}{Theorem}[section]
  \newtheorem{lem}[thm]{Lemma}
  \newtheorem{prop}[thm]{Proposition}
  \newtheorem{cor}[thm]{Corollary}
\theoremstyle{definition}
  \newtheorem{defn}[thm]{Definition}
  
  \newtheorem{exmp}[thm]{Example}
  \newtheorem{rem}[thm]{Remark}

\DeclareMathOperator*{\colim}{colim}

\DeclareMathOperator{\Lan}{Lan}
\DeclareMathOperator{\Ran}{Ran}
\DeclareMathOperator{\ub}{ub}
\DeclareMathOperator{\lb}{lb}

\makeatletter
\def\oarrowfill@#1#2#3#4#5{%
  $\m@th\thickmuskip0mu\medmuskip\thickmuskip\thinmuskip\thickmuskip
   \relax#5#1\mkern-7mu%
   \cleaders\hbox{$#5\mkern-2mu#2\mkern-2mu$}\hfill
   \mathclap{#3}\mathclap{#2}%
   \cleaders\hbox{$#5\mkern-2mu#2\mkern-2mu$}\hfill
   \mkern-7mu#4$%
}
\def\orightarrowfill@{%
  \oarrowfill@\relbar\relbar\circ{\to<100>}}
\newcommand\xorightarrow[2][]{%
  \ext@arrow 0055{\orightarrowfill@}{#1}{#2}}
\makeatother

\newcommand{\da}{\downarrow}

\newcommand{\ua}{\uparrow}
\newcommand{\ra}{\rightarrow}
\newcommand{\la}{\leftarrow}
\newcommand{\nra}{\nrightarrow}
\newcommand{\nla}{\nleftarrow}

\newcommand{\lda}{\swarrow}
\newcommand{\rda}{\searrow}
\newcommand{\rd}{\backslash}
\newcommand{\Lra}{\Longrightarrow}

\newcommand{\rat}{\!\rightarrowtail\!}
\newcommand{\xora}{\xorightarrow}
\newcommand{\oto}{\xora{\ \ \ \ }}

\newcommand{\bv}{\bigvee}
\newcommand{\bw}{\bigwedge}

\newcommand{\dv}{\dashv}
\newcommand{\rhu}{\rightharpoonup}
\newcommand{\ulc}{\ulcorner}
\newcommand{\urc}{\urcorner}
\newcommand{\nat}{\natural}
\newcommand{\dint}{\displaystyle\int}

\newcommand{\al}{\alpha}
\newcommand{\be}{\beta}

\newcommand{\ep}{\epsilon}

\newcommand{\De}{\Delta}
\newcommand{\ga}{\gamma}

\newcommand{\lam}{\lambda}

\newcommand{\si}{\sigma}

\newcommand{\CA}{\mathcal{A}}
\newcommand{\CB}{\mathcal{B}}

\newcommand{\CD}{\mathcal{D}}
\newcommand{\CE}{\mathcal{E}}
\newcommand{\CF}{\mathcal{F}}
\newcommand{\CG}{\mathcal{G}}

\newcommand{\CK}{\mathcal{K}}

\newcommand{\CM}{\mathcal{M}}

\newcommand{\CP}{\mathcal{P}}
\newcommand{\CQ}{\mathcal{Q}}

\newcommand{\CT}{\mathcal{T}}
\newcommand{\CU}{\mathcal{U}}
\newcommand{\CV}{\mathcal{V}}
\newcommand{\BD}{{\bf D}}
\newcommand{\BM}{{\bf M}}
\newcommand{\BT}{{\bf T}}
\newcommand{\BU}{{\bf U}}
\newcommand{\BY}{{\bf Y}}

\newcommand{\FD}{\mathfrak{D}}

\newcommand{\sA}{{\sf A}}
\newcommand{\sB}{{\sf B}}
\newcommand{\sC}{{\sf C}}

\newcommand{\sY}{{\sf Y}}
\newcommand{\bbA}{\mathbb{A}}
\newcommand{\bbB}{\mathbb{B}}
\newcommand{\bbC}{\mathbb{C}}

\newcommand{\bbX}{\mathbb{X}}

\newcommand{\Ab}{{\bf Ab}}
\newcommand{\Cat}{{\bf Cat}}
\newcommand{\CCat}{{\bf CCat}}
\newcommand{\Cls}{{\bf Cls}}
\newcommand{\Ctx}{{\bf Ctx}}
\newcommand{\Dist}{{\bf Dist}}

\newcommand{\Info}{{\bf Info}}
\newcommand{\Rel}{{\bf Rel}}
\newcommand{\Set}{{\bf Set}}
\newcommand{\Sp}{{\bf Sp}}
\newcommand{\Sup}{{\bf Sup}}
\newcommand{\Ch}{\textrm{Ch}}
\newcommand{\co}{{\rm co}}
\newcommand{\op}{{\rm op}}
\newcommand{\ChA}{\Ch_{\bullet}(\CA)}
\newcommand{\aR}{R^{\forall}}
\newcommand{\eR}{R_{\exists}}
\newcommand{\dR}{R^{\da}}
\newcommand{\uR}{R_{\ua}}

\newcommand{\dphi}{\phi^{\da}}
\newcommand{\uphi}{\phi_{\ua}}

\newcommand{\dpsi}{\psi^{\da}}
\newcommand{\upsi}{\psi_{\ua}}
\newcommand{\ulphi}{\underline{\phi}}
\newcommand{\olphi}{\overline{\phi}}
\newcommand{\ulpsi}{\underline{\psi}}
\newcommand{\olpsi}{\overline{\psi}}
\newcommand{\CPd}{{\mathcal{P}^{\dag}}}
\newcommand{\PA}{\CP\bbA}
\newcommand{\PB}{\CP\bbB}
\newcommand{\PC}{\CP\bbC}
\newcommand{\PX}{\CP\bbX}
\newcommand{\PdA}{\CP^{\dag}\bbA}
\newcommand{\PdB}{\CP^{\dag}\bbB}
\newcommand{\PdC}{\CP^{\dag}\bbC}
\newcommand{\sYd}{\sY^{\dag}}
\newcommand{\hF}{\ulc F \urc}
\newcommand{\hG}{\ulc G \urc}

\begin{document}

\pagenumbering{roman} \pagestyle{empty}

\title{{Adjunctions in Quantaloid-enriched Categories}
\vskip 1.5cm
{\Large by
\vskip 0.2cm
{Lili Shen}}
\vskip 1cm
{\large under the supervision of
\vskip 0.2cm
{Professor Dexue Zhang}}
\vskip 2.5cm
{\large A Dissertation Submitted to
\vskip 0.2cm
Sichuan University
\vskip 0.2cm
for the Degree of Doctor of Philosophy
\vskip 0.2cm
in Mathematics
\vskip 2cm
(English Version)}}

\author{}
\date{March 2014, Chengdu, China}
\maketitle

\setcounter{page}{1} \pagestyle{fancy}

\baselineskip=20pt

\cleardoublepage
\phantomsection
\addcontentsline{toc}{chapter}{Abstract}
\chapter*{Abstract}

Adjunction is a fundamental notion and an extremely useful tool in category theory. This dissertation is devoted to a study of adjunctions concerning categories enriched over a quantaloid $\CQ$ (or $\CQ$-categories for short), with the following types of adjunctions involved:
\begin{itemize}
\item[\rm (1)] adjoint functors between $\CQ$-categories;
\item[\rm (2)] adjoint distributors between $\CQ$-categories;
\item[\rm (3)] adjoint functors between categories consisting of $\CQ$-categories.
\end{itemize}

A quantaloid is a ``simplified'' bicategory in which all hom-sets are $\sup$-lattices. So, a theory of categories enriched over a quantaloid has been developed within the general framework, initiated by Betti and Walters, of categories enriched over a bicategory.

For a small quantaloid $\CQ$ and a distributor between $\CQ$-categories, two adjunctions between the $\CQ$-categories of contravariant and covariant presheaves are presented. These  adjunctions respectively extend the fundamental construction of Isbell adjunctions and Kan extensions in category theory, so, they will be called the Isbell adjunction and Kan adjunction, respectively. The functoriality of these constructions is the central topic of this dissertation.

In order to achieve this, infomorphisms between distributors are introduced to organize distributors (as objects) into a category. Then we proceed as follows.

First, the Isbell adjunction and Kan adjunction associated with each distributor $\phi:\bbA\oto\bbB$ between $\CQ$-categories respectively give rise to a monad $\dphi\circ\uphi$ on $\PA$ and a monad $\phi_*\circ\phi^*$ on $\PB$, where $\PA$ and $\PB$ denote respectively the $\CQ$-categories of contravariant presheaves on $\bbA$ and $\bbB$. A $\CQ$-category $\bbA$ together with a monad on $\PA$ is called a $\CQ$-closure spaces due to its resemblance to closure spaces in topology. It is proved that the correspondences
$$\phi\mapsto (\bbA,\dphi\circ\uphi)$$
and
$$\phi\mapsto(\bbB,\phi_*\circ\phi^*)$$
are respectively (covariant) functorial and contravariant functorial from the category of distributors and infomorphisms to the category of $\CQ$-closure spaces.

Second,  for each $\CQ$-closure space $(\bbA,C)$, the fixed objects of $C:\PA\to\PA$ is a complete $\CQ$-category. It is shown that the assignments of a distributor $\phi$ to the fixed points of $\dphi\circ\uphi$ and $\phi_*\circ\phi^*$  are respectively (covariant) functorial and contravariant functorial from the category of distributors and infomorphisms to that of skeletal complete $\CQ$-categories and left adjoint  functors.

As consequences of the functoriality of the above processes,  three factorizations of the free cocompletion functor of $\CQ$-categories are presented.

Finally, as applications, the theory of formal concept analysis and that of rough sets are extended to theories based on fuzzy relations between fuzzy sets.

\

{\bf Keywords:} Quantaloid, $\CQ$-category, $\CQ$-distributor, $\CQ$-closure space, Isbell adjunction, Kan adjunction.

\cleardoublepage
\phantomsection
\addcontentsline{toc}{chapter}{Contents}
\tableofcontents

\clearpage
\pagenumbering{arabic} \setcounter{page}{1}
\chapter{Introduction}

In this dissertation, we present a study of adjunctions concerning categories enriched over a quantaloid, with special attention to Isbell adjunctions and Kan adjunctions.

A quantaloid \cite{Rosenthal1996,Stubbe2005} is a category enriched over the symmetric monoidal closed category $\Sup$ consisting of $\sup$-lattices and $\sup$-preserving maps. Since a quantaloid $\CQ$ is a closed and locally complete bicategory, following the framework of R. Betti \cite{Betti1982} and R. F. C. Walters \cite{Walters1981}, one can develop a theory of categories enriched over $\CQ$ (or $\CQ$-categories for short). It should be stressed that since a quantaloid is a ``simplified'' bicategory in the sense that each hom-set is partially ordered (which is indeed a $\sup$-lattice), the coherence issues will not be a concern for $\CQ$-categories. For an overview of the theory of $\CQ$-categories, we refer to \cite{Heymans2010,Heymans2011,Stubbe2003,Stubbe2005,Stubbe2005a,Stubbe2006,Stubbe2013}.

\section{Adjoint functors} \label{Adjoint_functors}

``Adjoint functors arise everywhere'', as S. Mac Lane said in his classic textbook \emph{Categories for the Working Mathematician} \cite{MacLane1998}, the notion of adjoint functors introduced by D. M. Kan \cite{Kan1958} is now a core concept in category theory.

Let $\bbA$ and $\bbB$ be categories. A pair of functors $F:\bbA\to\bbB$ and $G:\bbB\to\bbA$ forms an adjunction, written $F\dv G:\bbA\rhu\bbB$, if there are natural transformations
$$\eta:1_{\bbA}\to G\circ F\quad\text{and}\quad\ep:F\circ G\to 1_{\bbB}$$
satisfying the triangle identities \cite{Awodey2010,MacLane1998}
$$\bfig
\Vtriangle/->`->`<-/[F`F`F\circ G\circ F;=`F\eta`\ep_F]
\Vtriangle(1500,0)/->`->`<-/[G`G`G\circ F\circ G;=`\eta_G`G\ep]
\efig,$$
where $1_{\bbA}$ and $1_{\bbB}$ are respectively the identity functors on $\bbA$ and $\bbB$. We present below two important examples of adjoint functors in category theory --- Isbell adjunctions and Kan extensions.
\begin{itemize}
\item[\rm (1)] Let $\bbA$ be a small category. The Isbell adjunction (or Isbell conjugacy) refers to the adjunction
    $$\ub\dv\lb:\Set^{\bbA^{\op}}\rhu(\Set^{\bbA})^{\op}$$
    arising from the Yoneda embedding $\sY:\bbA\to{\Set}^{\bbA^{\op}}$ and the co-Yoneda embedding $\sYd:\bbA\to({\Set}^{\bbA})^{\op}$ given by
    $$\ub(F)=\Set^{\bbA^{\op}}(F,\sY-)\quad\text{and}\quad\lb(G)=(\Set^{\bbA})^{\op}(\sYd-,G).$$
\item[\rm (2)] Let $K:\bbA\to\bbC$ be a functor, with $\bbA$ small. Composing with $K$ induces a functor
    $$-\circ K:{\Set}^{\bbC^{\op}}\to{\Set}^{\bbA^{\op}}.$$
    between the categories of presheaves on $\bbC$ and $\bbA$. The functor $-\circ K$ has a left adjoint $\Lan_K$ and a right adjoint $\Ran_K$. For each presheaf $F:\bbA^{\op}\to\Set$, $\Lan_K F$ and $\Ran_K F$ are respectively the left and right Kan extension of $F$ along $K$.
\end{itemize}

Isbell adjunctions and Kan extensions have also been considered for categories enriched over a symmetric monoidal closed category  \cite{Borceux1994b,Day2007,Dubuc1970,Kelly1982,Kelly2005,Lawvere1973,Lawvere1986,Riehl2014}, and will be outlined in Chapter \ref{preliminaries}.

\section{Adjoint morphisms in a bicategory} \label{Adjoint_morphisms}

The notion of adjoint functors is a special case of adjoint morphisms (or 1-cells) in a 2-category, or more generally, that in a bicategory.

Let $\CB$ be a bicategory with the associator $\al$, left unitor $\lam$ and right unitor $\rho$ (will be introduced in Section \ref{Categories_enriched_over_a_bicategory}). A pair of morphisms $f:X\to Y$ and $g:Y\to X$ forms an adjunction, written $f\dv g:X\rhu Y$, if there are 2-cells
$$\eta:1_X\to g\circ f\quad\text{and}\quad\ep:f\circ g\to 1_Y$$
such that the compositions of 2-cells
$$f\to^{\rho_{XY}^{-1}}f\circ 1_X\to^{f\circ\eta}f\circ(g\circ f)\to^{\al_{XYXY}}(f\circ g)\circ f\to^{\ep\circ f}1_Y\circ f\to^{\lam_{XY}}f$$
and
$$g\to^{\lam_{YX}^{-1}}1_X\circ g\to^{\eta\circ g}(g\circ f)\circ g\to^{\al_{YXYX}^{-1}}g\circ(f\circ g)\to^{g\circ\ep}g\circ 1_Y\to^{\rho_{YX}}g$$
are identities \cite{Gray1974}, where $1_X$ and $1_Y$ are respectively the image of the unit functor on $\CB(X,X)$ and $\CB(Y,Y)$.

In particular, if $\CB$ is a 2-category, then associators and unitors are identities, and $1_X$ and $1_Y$ are respectively the identity morphism on $X$ and $Y$. In this case, a pair of morphisms $f:X\to Y$ and $g:Y\to X$ forms an adjunction $f\dv g:X\rhu Y$, if there are 2-cells
$$\eta:1_X\to g\circ f\quad\text{and}\quad\ep:f\circ g\to 1_Y$$
satisfying the triangle identities
$$\bfig
\Vtriangle/->`->`<-/[f`f`f\circ g\circ f;=`f\circ\eta`\ep\circ f]
\Vtriangle(1500,0)/->`->`<-/[g`g`g\circ f\circ g;=`\eta\circ g`g\circ\ep]
\efig.$$

Adjoint functors are exactly adjoint morphisms in the 2-category $\Cat$ consisting of categories, functors and natural transformations.

\section{Adjunctions in $\CQ$-categories}

There are two basic types of adjunctions between categories enriched over a quantaloid $\CQ$, i.e., adjoint $\CQ$-functors between $\CQ$-categories (corresponding to Section \ref{Adjoint_functors}), and adjoint $\CQ$-distributors as adjoint morphisms in the 2-category of $\CQ$-categories and $\CQ$-distributors (corresponding to Section \ref{Adjoint_morphisms}).

A pair of $\CQ$-functors $F:\bbA\to\bbB$ and $G:\bbB\to\bbA$ forms an adjunction, written $F\dv G:\bbA\rhu\bbB$, if
$$1_{\bbA}\leq G\circ F\quad\text{and}\quad F\circ G\leq 1_{\bbB},$$
where $1_{\bbA}$ and $1_{\bbB}$ are respectively the identity $\CQ$-functors on $\bbA$ and $\bbB$. Adjoint $\CQ$-functors can be viewed as adjoint morphisms in the locally ordered 2-category $\CQ$-$\Cat$ of $\CQ$-categories and $\CQ$-functors.

Since a quantaloid $\CQ$ is itself a locally ordered 2-category, there are adjoint morphisms $f\dv g:X\rhu Y$ in $\CQ$ given by
$$1_X\leq g\circ f\quad\text{and}\quad f\circ g\leq 1_Y.$$
In our case, $\CQ$-categories and $\CQ$-distributors constitute a quantaloid $\CQ$-$\Dist$. Adjoint morphisms in $\CQ$-$\Dist$ give rise to the notion of adjoint $\CQ$-distributors. Specifically, a pair of $\CQ$-distributors $\phi:\bbA\oto\bbB$ and $\psi:\bbB\oto\bbA$ forms an adjunction, written $\phi\dv\psi:\bbA\rhu\bbB$, if
$$\bbA\leq\psi\circ\phi\quad\text{and}\quad\phi\circ\psi\leq\bbB,$$
where $\bbA$ and $\bbB$ are respectively the identity $\CQ$-distributors on $\bbA$ and $\bbB$.

The following constructions exhibit the abundance of adjoint $\CQ$-functors and adjoint $\CQ$-distributors.
\begin{itemize}
\item[\rm (1)] Each $\CQ$-functor $F:\bbA\to\bbB$ induces a pair of adjoint $\CQ$-distributors
    $$F_{\nat}\dv F^{\nat}:\bbA\rhu\bbB,$$
    called respectively the graph and the cograph of $F$.
\item[\rm (2)] If $\CQ$ is a small quantaloid, then each $\CQ$-distributor $\phi:\bbA\oto\bbB$ induces two pairs of adjoint $\CQ$-functors
    \begin{equation} \label{Isbell_intro}
    \uphi\dv\dphi:\PA\rhu\PdB
    \end{equation}
    and
    \begin{equation} \label{Kan_intro}
    \phi^*\dv\phi_*:\PB\rhu\PA
    \end{equation}
    between the $\CQ$-categories of contravariant presheaves and covariant presheaves.
\end{itemize}
These adjunctions will be the core subject of our discussion.

It should be pointed out that the adjunctions (\ref{Isbell_intro}) and (\ref{Kan_intro}) extend the construction of Isbell adjunctions and Kan extensions presented in Section \ref{Adjoint_functors}:
\begin{itemize}
\item[\rm (1)] The $\CQ$-category $\PA$ of contravariant presheaves and the $\CQ$-category $\PdA$ of covariant presheaves are the counterparts of $\Set^{\bbA^{\op}}$ and $(\Set^\bbA)^{\op}$, respectively.
\item[\rm (2)] If $\phi$ is the identity $\CQ$-distributor on $\bbA$, then the adjunction $\uphi\dv\dphi$ reduces to the Isbell adjunction
    $$\bbA\lda(-)\dv(-)\rda\bbA:\PA\rhu\PdA$$
    presented in \cite{Stubbe2005}.
\item[\rm (3)] Given a $\CQ$-functor $F:\bbA\to\bbB$, consider the graph $F_\natural :\bbA\oto\bbB$ and the cograph $F^\natural :\bbB\oto\bbA$. Then it holds that (Theorem \ref{why_kan})
    $$(F^{\natural})^*\dv (F^{\natural})_*= F^{\la}=(F_{\natural})^*\dv(F_{\natural})_*,$$
    where $F^\la:\PB\to\PA$ is the counterpart of the functor $-\circ F$ for $\CQ$-categories.
\end{itemize}
Therefore, adjunctions of the forms (\ref{Isbell_intro}) and (\ref{Kan_intro}) are called Isbell adjunctions and Kan adjunctions, respectively.

\section{Functoriality of the Isbell adjunction and Kan adjunction}

The primary goal of this dissertation is to discuss the functoriality of the constructions of the Isbell adjunction (\ref{Isbell_intro}) and Kan adjunction (\ref{Kan_intro}) in the premise that $\CQ$ is a small quantaloid.

For each $\CQ$-distributor $\phi:\bbA\oto\bbB$, the related Isbell adjunction (\ref{Isbell_intro}) and Kan adjunction (\ref{Kan_intro}) give rise to a monad $\dphi\circ\uphi$ on $\PA$ and a monad $\phi_*\circ\phi^*$ on $\PB$, respectively. In order to establish functoriality of the correspondences
\begin{equation} \label{phi_mapsto_dphiuphi}
(\phi:\bbA\oto\bbB)\ \mapsto\ (\bbA,\dphi\circ\uphi)
\end{equation}
and
\begin{equation} \label{phi_mapsto_phistar}
(\phi:\bbA\oto\bbB)\ \mapsto\ (\bbB,\phi_*\circ\phi^*),
\end{equation}
we introduce infomorphisms between $\CQ$-distributors to construct a category with $\CQ$-distributors as objects.

Given $\CQ$-distributors $\phi:\bbA\oto\bbB$ and $\psi:\bbA'\oto\bbB'$, an infomorphism $(F,G):\phi\to\psi$ is a pair of $\CQ$-functors $F:\bbA\to\bbA'$ and $G:\bbB'\to\bbB$ such that
$$\phi(-,G-)=\psi(F-,-).$$
$\CQ$-distributors and infomorphisms constitute a category $\CQ$-$\Info$.

A $\CQ$-category $\bbA$ together with a monad $C:\PA\to\PA$ is called a $\CQ$-closure space, and lax maps between monads on $\PA$ are called continuous $\CQ$-functors, due to their resemblance to closure spaces in topology (will be introduced in Chapter \ref{closure_space}). $\CQ$-closure space and continuous $\CQ$-functors constitute a category $\CQ$-$\Cls$.

It is shown that, the correspondence (\ref{phi_mapsto_dphiuphi}) gives rise to a right adjoint (covariant) functor from $\CQ$-$\Info$ to $\CQ$-$\Cls$, and the correspondence (\ref{phi_mapsto_phistar}) defines a contravariant functor (which is a right adjoint if $\CQ$ is a Girard quantaloid) from $\CQ$-$\Info$ to $\CQ$-$\Cls$.

Furthermore, for each $\CQ$-closure space $(\bbA,C)$, the fixed points of $C:\PA\to\PA$ (or equivalently, all the algebras if we consider $C$ as a monad) is a complete $\CQ$-category. We show that the correspondence
$$(\bbA,C)\mapsto C(\PA)$$
gives rise to a left adjoint functor from $\CQ$-$\Cls$ to the category $\CQ$-$\CCat$ of skeletal complete $\CQ$-categories and left adjoint $\CQ$-functors.

Therefore, the assignments of a $\CQ$-distributor $\phi$ to the fixed points of $\dphi\circ\uphi$ and $\phi_*\circ\phi^*$ respectively give rise to a (covariant) functor
$$\CM:\CQ\text{-}\Info\to\CQ\text{-}\CCat$$
and a contravariant functor
$$\CK:(\CQ\text{-}\Info)^{\op}\to\CQ\text{-}\CCat.$$

As consequences of the functoriality of the above processes, we present three factorizations of the free cocompletion functor $\CP$ of $\CQ$-categories:
\begin{itemize}
\item[\rm (1)] $\CP$ factors through the category $\CQ$-$\Cls$ of $\CQ$-closure spaces;
\item[\rm (2)] $\CP$ factors through the functor $\CM$ induced by Isbell adjunctions;
\item[\rm (3)] $\CP$ factors through the functor $\CK$ induced by Kan adjunctions.
\end{itemize}

\section{Applications in fuzzy set theory}

In Chapter \ref{Applications}, we present some applications of quantaloid-enriched categories in fuzzy set theory. We demonstrate that Isbell adjunctions and Kan adjunctions are closely related to the theories of formal concept analysis \cite{Carpineto2004,Davey2002,Ganter2005,Ganter1999} and rough set theory \cite{Pawlak1982,Polkowski2002,Yao2004} in computer science. The functoriality of Isbell adjunctions and Kan adjunctions generalizes the functoriality of concept lattices based on formal concept analysis and rough set theory \cite{Shen2013}. Furthermore, the theories of formal concept analysis and rough set theory on fuzzy sets are established.

Recall that a preorder $\leq$ on a (crisp) set $A$ is a reflexive and transitive relation \cite{Davey2002,Gierz2003} on $A$ in the sense that
\begin{itemize}
\item[\rm (1)] $\forall x\in A$, $x\leq x$;
\item[\rm (2)] $\forall x,y,z\in A$, $y\leq z$ and $x\leq y$ implies $x\leq z$.
\end{itemize}
A preordered set $(A,\leq)$ is exactly a category enriched over the quantaloid ${\bf 2}$, the two-element Boolean algebra.

A formal context $(A,B,R)$ consists of (crisp) sets $A,B$ and a relation $R\subseteq A\times B$ (or equivalently, a ${\bf 2}$-distributor between discrete ${\bf 2}$-categories). Formal contexts provide a common framework for formal concept analysis and rough set theory. For each formal context $(A,B,R)$, there exists a contravariant Galois connection $(\uR,\dR)$ and a covariant Galois connection $(\eR,\aR)$ between the powersets of $A$ and $B$. These two Galois connections play fundamental roles in formal concept analysis and rough set theory, respectively, and they are exactly the Isbell adjunction and Kan adjunction in ${\bf 2}$-categories.

Classical preorders are extended to many-valued preorders by replacing ${\bf 2}$ with a more complicated quantaloid:
\begin{itemize}
\item[\rm (1)] Let $Q$ be a unital quantale, i.e., a one-object quantaloid. A fuzzy relation on a (crisp) set $A$ is a map $A\times A\to Q$, and the notion of reflexivity and transitivity can be extended to fuzzy relations. Then a (crisp) set $A$ equipped with a reflexive and transitive fuzzy relation on $A$ is a $Q$-preordered set \cite{Bvelohlavek2004,Lai2006,Lai2009,Shen2013,Zadeh1971}, which is exactly a category enriched over $Q$.
\item[\rm (2)] Let $Q$ be a divisible unital quantale. A fuzzy set (or a $Q$-subset \cite{Goguen1967,Zadeh1965}) is a map $\sA:\sA_0\to Q$ from a (crisp) set $\sA_0$ to $Q$, where the value $\sA x$ is interpreted as the membership degree of $x$ in $\sA_0$. A $Q$-preorder on a fuzzy set \cite{Pu2012,Tao2012} is then characterized as a category enriched over a quantaloid $\CQ$ induced by $Q$ (see Proposition \ref{divisible_quantale_quantaloid} for the definition of $\CQ$), which can also be interpreted elementarily as a reflexive and transitive fuzzy relation on the fuzzy set $\sA$.
\end{itemize}

Thus, by translating Isbell adjunctions and Kan adjunctions in categories enriched over a certain quantaloid $\CQ$, the theories of formal concept analysis and rough set theory can be established on fuzzy relations between (crisp) sets, and further on fuzzy relation between fuzzy sets. As consequences of the results in Chapter \ref{Isbell_adjunction} and \ref{Kan_adjunction}, for formal concept analysis and rough set theory built in each one of these frameworks, the processes of generating concept lattices from formal contexts are functorial.

\chapter{Preliminaries} \label{preliminaries}

We assume that the readers are familiar with the basic notions of classic category theory \cite{Awodey2010,Barr1990,Borceux1994a,Lawvere2003,MacLane1998}, including category, functor, natural transformation, Yoneda embedding, limit, colimit and adjunction. In this chapter, we review the theory of categories enriched over a monoidal category $\CV$, and then introduce Isbell adjunctions and Kan extensions in $\CV$-categories. Finally, we present the basic concepts of categories enriched over a bicategory as a foundation for the next chapter.

\section{Categories enriched over a monoidal category}

In this section, we recall some basic concepts of categories enriched over a monoidal category. Some necessary \emph{coherence} diagrams are omitted here, and we refer to \cite{Adamek1990,Borceux1994b,Kelly1982,Lawvere1973,MacLane1998} for the detailed definitions.

\begin{defn} \label{monoidal_category}
A \emph{monoidal category} $\CV$ is a category equipped with a bifunctor
$$\otimes:\CV\times\CV\to\CV$$
as the \emph{tensor product}, which is associative in the sense that there is a natural and coherent isomorphism
$$a_{XYZ}:X\otimes(Y\otimes Z)\cong(X\otimes Y)\otimes Z$$
for each triple $X,Y,Z$ of objects in $\CV$, and has a unit object $I$ of $\CV$ in the sense that there are natural and coherent isomorphisms
$$l_X:I\otimes X\cong X\quad\text{and}\quad r_X:X\otimes I\cong X$$
for each object $X$ of $\CV$.
\end{defn}

\begin{defn} \label{monoidal_category_symmetric}
A monoidal category $\CV$ is \emph{symmetric} if there is a natural and coherent isomorphism
$$s_{XY}:X\otimes Y\cong Y\otimes X$$
for each pair $X,Y$ of objects in $\CV$.
\end{defn}

\begin{defn} \label{monoidal_category_closed}
A monoidal category $\CV$ is \emph{closed} if both the functors
$$-\otimes X:\CV\to\CV\quad\text{and}\quad X\otimes -:\CV\to\CV$$
have a right adjoint.
\end{defn}

\begin{exmp} \label{monoidal_category_example}
\begin{itemize}
\item[\rm (1)] The category $\Set$ of small sets and functions is a symmetric monoidal closed category with cartesian products of sets playing as the tensor products.
\item[\rm (2)] \cite{Rosenthal1990} A unital quantale is a complete lattice $Q$ equipped with a binary operation $\&$ such that
    \begin{itemize}
    \item[\rm (i)] $(Q,\&)$ is a monoid;
    \item[\rm (ii)] $a\&(\bv b_i)=\bv(a\& b_i)$ and $(\bv b_i)\& a=\bv(b_i\& a)$ for all $a,b_i\in Q$.
    \end{itemize}
    A unital quantale $(Q,\&)$ is a monoidal closed category with $\&$ being the tensor product operation. Furthermore, if $(Q,\&)$ is a commutative monoid, then $(Q,\&)$ is a commutative unital quantale, which is a symmetric monoidal closed category.
\item[\rm (3)] \cite{Joyal1984,Kenney2010,Pitts1988} A $\sup$-lattice is a partially ordered set $X$ in which every subset $A\subseteq X$ has a join (or supremum, least upper bound) $\bv A$. A $\sup$-lattice necessarily admits a meet (or infimum, greatest lower bound) $\bw A$ for every subset $A\subseteq X$. A $\sup$-preserving map $f:X\to Y$ between $\sup$-lattices is a map $f:X\to Y$ such that $f(\bv A)=\bv f(A)$ for every subset $A\subseteq X$. Sup-lattices and $\sup$-preserving maps constitute a symmetric monoidal closed category $\Sup$. The tensor product $X\otimes Y$ of two $\sup$-lattices $X$ and $Y$ is constructed as a quotient of the free $\sup$-lattice on $X\times Y$, i.e., the powerset $\CP(X\times Y)$, obtained from the equivalence relation generated by
    $$(\bv\limits_i x_i,y)\sim\bv\limits_i(x_i,y)\quad\text{and}\quad(x,\bv\limits_j y_j)\sim\bv\limits_j(x,y_j).$$
\item[\rm (4)] \cite{Borceux1994b,Hilton2012} The category $\Ab$ of abelian groups is a symmetric monoidal closed category equipped with the usual tensor products of abelian groups.
\item[\rm (5)] \cite{Hilton2012} The category $\ChA$ of chain complexes of $R$-modules over a commutative ring $R$ is a symmetric monoidal closed category with the usual tensor products of chain complexes.
\item[\rm (6)] The category $\Cat$ of small categories and functors is a symmetric monoidal closed category with the products of categories playing as the tensor products.
\end{itemize}
\end{exmp}

\begin{defn} \label{V_cat}
Let $\CV$ be a monoidal category. A \emph{$\CV$-category} (or category enriched over $\CV$) $\bbA$ consists of a class $\bbA_0$ as the objects, a hom-object $\bbA(x,y)$ of $\CV$ for all $x,y\in\bbA_0$, a morphism
$$c_{xyz}:\bbA(y,z)\otimes\bbA(x,y)\to\bbA(x,z)$$
in $\CV$ for all $x,y,z\in\bbA_0$ as the composition, and a morphism
$$u_x:I\to\bbA(x,x)$$
in $\CV$ for all $x\in\bbA_0$ as the unit, such that the diagrams expressing the associativity and unity laws
$$\bfig
\Atriangle/->`->`/<800,500>[\bbA(z,w)\otimes(\bbA(y,z)\otimes\bbA(x,y))`(\bbA(z,w)\otimes\bbA(y,z))\otimes\bbA(x,y)`\bbA(z,w)\otimes\bbA(x,z);
a_{\bbA(z,w),\bbA(y,z),\bbA(x,y)}`1_{\bbA(z,w)}\otimes c_{xyz}`]
\square(0,-500)/`->`->`->/<1600,500>[(\bbA(z,w)\otimes\bbA(y,z))\otimes\bbA(x,y)`\bbA(z,w)\otimes\bbA(x,z)`\bbA(y,w)\otimes\bbA(x,y)`\bbA(x,w);`c_{yzw}\otimes 1_{\bbA(x,y)}`c_{xzw}`c_{xyw}]
\efig$$
$$\bfig
\qtriangle<1200,500>[I\otimes\bbA(x,y)`\bbA(y,y)\otimes\bbA(x,y)`\bbA(x,y);u_y\otimes 1_{\bbA(x,y)}`l_{\bbA(x,y)}`c_{xyy}]
\qtriangle(0,-800)<1200,500>[\bbA(x,y)\otimes I`\bbA(x,y)\otimes\bbA(x,x)`\bbA(x,y);1_{\bbA(x,y)}\otimes u_x`r_{\bbA(x,y)}`c_{xxy}]
\efig$$
are commutative for all $x,y,z,w\in\bbA_0$.
\end{defn}

\begin{exmp} \label{V_category_example}
\begin{itemize}
\item[\rm (1)] $\Set$-categories are just the ordinary locally small categories.
\item[\rm (2)] \cite{Garcia2010,Shen2013} A category enriched over a unital quantale $Q$ is a $Q$-valued preordered set.
\item[\rm (3)] \cite{Rosenthal1996,Stubbe2005} A $\Sup$-category is a quantaloid.
\item[\rm (4)] \cite{Hilton2012} An $\Ab$-category is a ringoid. An $\Ab$-category becomes an additive category if it admits finite (co)products.
\item[\rm (5)] \cite{Bertrand2011} Let $\ChA$ be the category of chain complexes of $R$-modules over a commutative ring $R$. A category enriched over $\ChA$ is a differential graded category, or dg category for short.
\item[\rm (6)] \cite{Johnstone2002a,Leinster2004,MacLane1998} A $\Cat$-category is a (strict) 2-category.
\item[\rm (7)] \cite{Borceux1994b,Kelly1982} If $\CV$ is a symmetric monoidal closed category, then $\CV$ is itself a $\CV$-category, with $\CV(X,-)$ being the right adjoint of $-\otimes X$ for each object $X\in\CV$.
\end{itemize}
\end{exmp}

If $\CV$ is a symmetric monoidal category, then each $\CV$-category $\bbA$ has a \emph{dual} $\CV$-category $\bbA^{\op}$ given by the same objects and
$$\bbA^{\op}(x,y)=\bbA(y,x)$$
for all $x,y\in\bbA_0$.

\begin{defn} \label{V_functor}
Let $\CV$ be a monoidal category. A \emph{$\CV$-functor} $F:\bbA\to\bbB$ between $\CV$-categories consists of a map $F:\bbA_0\to\bbB_0$ and a morphism
$$F_{xx'}:\bbA(x,x')\to\bbB(Fx,Fx')$$
in $\CV$ for all $x,x'\in\bbA_0$, such that the diagrams
$$\bfig
\square<1400,500>[\bbA(x',x'')\otimes\bbA(x,x')`\bbA(x,x'')`\bbB(Fx',Fx'')\otimes\bbB(Fx,Fx')`\bbB(Fx,Fx'');(c_{\bbA})_{xx'x''}`F_{x'x''}\otimes F_{xx'}`F_{xx''}`(c_{\bbB})_{Fx,Fx',Fx''}]
\efig$$
$$\bfig
\qtriangle<1000,500>[I`\bbA(x,x)`\bbB(Fx,Fx);(u_{\bbA})_x`(u_{\bbB})_{Fx}`F_{xx}]
\efig$$
are commutative for all $x,x',x''\in\bbA_0$.

A $\CV$-functor $F:\bbA\to\bbB$ is \emph{fully faithful} if $F_{xx'}$ is an isomorphism in $\CV$ for all $x,x'\in\bbA_0$.
\end{defn}


\begin{defn} \label{V_nat_trans}
Let $\CV$ be a monoidal category. A \emph{$\CV$-natural transformation} $\al:F\to/=>/G$ between $\CV$-functors $F,G:\bbA\to\bbB$ is given by a morphism
$$\al_x:I\to\bbB(Fx,Gx)$$
in $\CV$ for each $x\in\bbA_0$, such that the diagram
$$\bfig
\Ctriangle/<-``->/[I\otimes\bbA(x,x')`\bbA(x,x')`\bbA(x,x')\otimes I;l^{-1}_{\bbA(x,x')}``r^{-1}_{\bbA(x,x')}]
\square(500,0)/->```->/<1500,1000>[I\otimes\bbA(x,x')`\bbB(Fx,Gx)\otimes\bbB(Gx,Gx')`\bbA(x,x')\otimes I`\bbB(Fx,Fx')\otimes\bbB(Fx',Gx');\al_x\otimes G_{xx'}```F_{xx'}\otimes\al_{x'}]
\Dtriangle(2000,0)/`->`<-/[\bbB(Fx,Gx)\otimes\bbB(Gx,Gx')`\bbB(Fx,Gx')`\bbB(Fx,Fx')\otimes\bbB(Fx',Gx');`c_{Fx,Gx,Gx'}`c_{Fx,Fx',Gx'}]
\efig$$
is commutative for all $x,x'\in\bbA_0$.
\end{defn}

If $\CV$ is a symmetric monoidal category, then the category $\CV$-$\Cat$ of small $\CV$-categories and $\CV$-functors is itself a symmetric monoidal category, with the tensor product $\bbA\otimes\bbB$ of $\CV$-categories given by
$$(\bbA\otimes\bbB)_0=\bbA_0\times\bbB_0$$
and
$$(\bbA\otimes\bbB)((x,y),(x',y'))=\bbA(x,x')\otimes\bbB(y,y')$$
for all $x,x'\in\bbA_0$ and $y,y'\in\bbB_0$.

\begin{defn} \label{V_end}
Let $\CV$ be a complete symmetric monoidal closed category and $\bbA$ a small $\CV$-category. The \emph{end} of a $\CV$-functor $F:\bbA^{\op}\otimes\bbA\to\CV$ consists of an object $\dint_{a\in\bbA_0}F(a,a)$ of $\CV$ and a universal $\CV$-natural family
$$\al_x:\dint_{a\in\bbA_0}F(a,a)\to F(x,x)$$
in the sense that any other $\CV$-natural family $\be_x:X\to F(x,x)$ factors uniquely through $\al_x$ via a morphism $f:X\to\dint_{a\in\bbA_0}F(a,a)$ in $\CV$.
\end{defn}


If $\CV$ is a complete symmetric monoidal closed category, then the $\CV$-functors between a small $\CV$-category $\bbA$ and a $\CV$-category $\bbB$ constitute a $\CV$-category $\bbB^{\bbA}$ with
$$\bbB^{\bbA}(F,G)=\int_{x\in\bbA_0}\bbB(Fx,Gx)$$
for all $\CV$-functors $F,G:\bbA\to\bbB$. The object $\bbB^{\bbA}(F,G)\in\CV$ is called the \emph{object of $\CV$-natural transformations} from $F$ to $G$.

\begin{defn} \label{V_limit_colimits}
Let $\CV$ be a symmetric monoidal closed category and $F:\bbA\to\bbB$ a $\CV$-functor.
\begin{itemize}
\item[\rm (1)] The \emph{limit} of $F$ \emph{weighted by} a $\CV$-functor $G:\bbA\to\CV$ is an object $\lim_G F\in\bbB_0$ together with a $\CV$-natural isomorphism
    $$ \bbB(b,{\lim}_G F)\cong\CV^{\bbA}(G,\bbB(b,F-)).$$
    A $\CV$-category $\bbB$ is \emph{complete} if $\lim_G F$ exists for each $\CV$-functor $F:\bbA\to\bbB$ and $G:\bbA\to\CV$ with the domain $\bbA$ small.
\item[\rm (2)] The \emph{colimit} of $F$ \emph{weighted by} a $\CV$-functor $G:\bbA^{\op}\to\CV$ is an object $\colim_G F\in\bbB_0$ together with a $\CV$-natural isomorphism
    $$\bbB({\colim}_G F,b)\cong\CV^{\bbA^{\op}}(G,\bbB(F-,b)).$$
    A $\CV$-category $\bbB$ is \emph{cocomplete} if $\colim_G F$ exists for each $\CV$-functor $F:\bbA\to\bbB$ and $G:\bbA^{\op}\to\CV$ with the domain $\bbA$ small.
\end{itemize}
\end{defn}

\begin{defn} \label{V_adjunction}
A pair of $\CV$-functors $F:\bbA\to\bbB$ and $G:\bbB\to\bbA$ between $\CV$-categories is called an \emph{adjunction}, written $F\dv G:\bbA\rhu\bbB$, if for every pair $x\in\bbA_0$ and $y\in\bbB_0$, there is an isomorphism in $\CV$
$$\bbB(Fx,y)\cong\bbA(x,Gy)$$
$\CV$-natural in $x$ and $y$. In this case, $F$ is called a \emph{left adjoint} of $G$ and $G$ a \emph{right adjoint} of $F$.
\end{defn}

\section{Isbell adjunctions in $\CV$-categories}

In this section, we assume that $\CV$ is a complete symmetric monoidal closed category. In this case, $\CV$ is itself a complete $\CV$-category.

For each small $\CV$-category $\bbA$, there is a $\CV$-functor $\sY:\bbA\to\CV^{\bbA^{\op}}$ given by
$$\sY x=\bbA(-,x)\quad\text{and}\quad\sY_{xx'}=\CV^{\bbA^{\op}}(\bbA(-,x),\bbA(-,x'))$$
and a $\CV$-functor $\sYd:\bbA\to(\CV^{\bbA})^{\op}$ given by
$$\sYd x=\bbA(x,-)\quad\text{and}\quad\sY_{xx'}=(\CV^{\bbA})^{\op}(\bbA(x,-),\bbA(x',-)).$$
The following lemma implies that both $\sY$ and $\sYd$ are fully faithful $\CV$-functors, and are called respectively the \emph{Yoneda embedding} and the \emph{co-Yoneda embedding}.

\begin{lem}[Yoneda] {\rm\cite{Borceux1994b,Kelly1982}}
Let $\bbA$ be a small $\CV$-category. For each $\CV$-functor $F:\bbA\to\CV$ and $x\in\bbA_0$, there is an isomorphism in $\CV$
$$\CV^{\bbA}(\bbA(x,-),F)\cong Fx.$$
\end{lem}

For each small $\CV$-category $\bbA$, the Yoneda embedding
$$\sY:\bbA\to\CV^{\bbA^{\op}}$$
and the co-Yoneda embedding
$$\sYd:\bbA\to(\CV^{\bbA})^{\op}$$
induce a pair of $\CV$-functors
$$\ub:\CV^{\bbA^{\op}}\to(\CV^{\bbA})^{\op}\quad\text{and}\quad\lb:(\CV^{\bbA})^{\op}\to\CV^{\bbA^{\op}}$$
given by
$$\ub(F)=\CV^{\bbA^{\op}}(F,\sY-)\quad\text{and}\quad\lb(G)=(\CV^{\bbA})^{\op}(\sYd-,G).$$

\begin{prop}[Isbell] \label{Isbell_classical} {\rm\cite{Day2007,Kelly2005,Lawvere1986}}
$\ub\dv\lb:\CV^{\bbA^{\op}}\rhu(\CV^{\bbA})^{\op}$.
\end{prop}

This adjunction is known as the \emph{Isbell adjunction} (or \emph{Isbell conjugacy}) in category theory.

\section{Kan extensions of $\CV$-functors} \label{Kan_extensions_V_functors}

In this section, we also assume that $\CV$ is a complete symmetric monoidal closed category.

\begin{defn}
Let $K:\bbA\to\bbC$ be a $\CV$-functor and $\bbB$ a $\CV$-category.
\begin{itemize}
\item[\rm (1)] The \emph{left Kan extension} of a $\CV$-functor $F:\bbA\to\bbB$ along $K:\bbA\to\bbC$, if it exists, is a $\CV$-functor
$$\Lan_K F:\bbC\to\bbB$$
together with a $\CV$-natural isomorphism
$$\bbB^{\bbC}(\Lan_K F,S)\cong\bbB^{\bbA}(F,S\circ K)$$
for any other $\CV$-functor $S:\bbC\to\bbB$.
\item[\rm (2)] The \emph{right Kan extension} of a $\CV$-functor $F:\bbA\to\bbB$ along $K:\bbA\to\bbC$, if it exists, is a $\CV$-functor
$$\Ran_K F:\bbC\to\bbB$$
together with a $\CV$-natural isomorphism
$$\bbB^{\bbC}(S,\Ran_K F)\cong\bbB^{\bbA}(S\circ K,F)$$
for any other $\CV$-functor $S:\bbC\to\bbB$.
\end{itemize}
\end{defn}

If the left Kan extension $\Lan_K F$ of a $\CV$-functor $F:\bbA\to\bbB$ along $K:\bbA\to\bbC$ exists, then there is a \emph{universal} $\CV$-natural transformation $\eta:F\to/=>/(\Lan_K F)\circ K$ in the sense that for any other $\CV$-functor $S:\bbC\to\bbA$ and $\CV$-natural transformation $\ga:F\to/=>/ S\circ K$, $\ga$ factors uniquely through $\eta$.
$$\bfig
\Vtriangle/->`->`<-/[\bbA`\bbB`\bbC;F`K`S]
\place(500,300)[\ga\Downarrow]
\place(1250,300)[=]
\Vtriangle(1500,0)|alm|/->`->`<--/[\bbA`\bbB`\bbC;F`K`\Lan_K F]
\place(1950,300)[\eta\Downarrow]
\morphism(2000,0)|r|/{@{>}@/_2.5em/}/<500,500>[\bbC`\bbB;S]
\place(2300,130)[\exists !\ \twoar(1,-1)]
\efig$$

Dually, the right Kan extension $\Ran_K F$ of a $\CV$-functor $F:\bbA\to\bbB$ is equipped with a \emph{universal} $\CV$-natural transformation $\ep:(\Ran_K T)\circ K\to/=>/ T$ such that for any other $\CV$-functor $S:\bbC\to\bbA$ and $\CV$-natural transformation $\si:S\circ K\to/=>/ F$, $\si$ factors uniquely through $\ep$.
$$\bfig
\Vtriangle/->`->`<-/[\bbA`\bbB`\bbC;F`K`S]
\place(500,300)[\si\Uparrow]
\place(1250,300)[=]
\Vtriangle(1500,0)|alm|/->`->`<--/[\bbA`\bbB`\bbC;F`K`\Ran_K F]
\place(1950,300)[\ep\Uparrow]
\morphism(2000,0)|r|/{@{>}@/_2.5em/}/<500,500>[\bbC`\bbB;S]
\place(2300,130)[\exists !\twoar(-1,1)]
\efig$$

Given a $\CV$-functor $K:\bbA\to\bbC$ and another $\CV$-category $\bbB$, composing with $K$ yields a $\CV$-functor
$$K^*:\bbB^{\bbC}\to\bbB^{\bbA},$$
which sends a $\CV$-functor $S:\bbC\to\bbB$ to $S\circ K:\bbA\to\bbB$.
$$\bfig
\Vtriangle/->`->`<-/[\bbA`\bbB`\bbC;S\circ K`K`S]
\efig$$
If each $\CV$-functor $F:\bbA\to\bbB$ has a left Kan extension $\Lan_K F:\bbC\to\bbB$ along $K$, then we obtain an adjunction \cite{Borceux1994b}
$$\Lan_K\dv K^*:\bbB^{\bbA}\rhu\bbB^{\bbC}.$$
Dually, if each $\CV$-functor $F:\bbA\to\bbB$ has a right Kan extension $\Ran_K F:\bbC\to\bbB$ along $K$, then we obtain an adjunction
$$K^*\dv\Ran_K:\bbB^{\bbC}\rhu\bbB^{\bbA}.$$

Given a $\CV$-functor $F:\bbA\to\bbB$, if the codomain $\bbB$ admits certain weighted colimits, then the left Kan extension of $F$ along some $\CV$-functor $K:\bbA\to\bbC$ can be constructed \emph{pointwise} for each object $c\in\bbC_0$. Dually, if the codomain $\bbB$ admits certain weighted limits, then the right Kan extension of $F$ along $K$ can be constructed pointwise.

\begin{prop} {\rm\cite{Kelly1982}}
Let $F:\bbA\to\bbB$ and $K:\bbA\to\bbC$ be $\CV$-functors.
\begin{itemize}
\item[\rm (1)] The left Kan extension of $F$ along $K$ can be computed by
$$(\Lan_K F)c={\colim}_{\bbC(K-,c)}F$$
if the weighted colimit exists for each $c\in\bbC_0$.
\item[\rm (2)] The right Kan extension of $F$ along $K$ can be computed by
$$(\Ran_K F)c={\lim}_{\bbC(c,K-)}F$$
if the weighted limit exists for each $c\in\bbC_0$.
\end{itemize}
\end{prop}

Thus, if $\bbA$ is small, $\bbB$ is complete and cocomplete, then every $\CV$-functor $F:\bbA\to\bbB$ has a (pointwise) left Kan extension and a (pointwise) right Kan extension along any $\CV$-functor $K:\bbA\to\bbC$. In particular, we have the following corollary.

\begin{cor} \label{monoidal_Kan} {\rm\cite{Lawvere1973}}
Let $\CV$ be a complete and cocomplete symmetric monoidal closed category. For each $\CV$-functor $K:\bbA\to\bbC$ with $\bbA$ small, the $\CV$-functor ``composing with $K$''
$$K^*:\CV^{\bbC}\to\CV^{\bbA}$$
has a left adjoint $\Lan_K$ and a right adjoint $\Ran_K$.
\end{cor}

\section{Categories enriched over a bicategory} \label{Categories_enriched_over_a_bicategory}

A $\Cat$-category is called a (strict) 2-category. However, we are particularly interested in a weak version of 2-categories --- categories ``weakly enriched'' over $\Cat$, in which the associativity and unity laws of enriched categories hold only commutative up to natural and coherent isomorphism. This is what we call a \emph{bicategory} \cite{Benabou1967,Carboni2008,Carboni1987,Lack2010,Leinster1998}.

\begin{defn} \cite{Benabou1967,MacLane1998} \label{bicategory}
A \emph{bicategory} $\CB$ consists of the following data:
\begin{itemize}
\item[\rm (1)] a class $\CB_0$ of objects $X,Y,Z,\dots$, also called \emph{0-cells};
\item[\rm (2)] for every pair $X,Y$ of 0-cells, a (small) category $\CB(X,Y)$, whose objects are called morphisms or \emph{1-cells} and whose morphisms are called 2-morphisms or \emph{2-cells};
\item[\rm (3)] for every triple $X,Y,Z$ of 0-cells, a \emph{horizontal composition} functor,
               $$\circ:\CB(Y,Z)\times\CB(X,Y)\to\CB(X,Z);$$
\item[\rm (4)] for every 0-cell $X$, a \emph{unit} functor
               $$u_X:1\to\CB(X,X)$$
               from the terminal category $1$, which picks out a 1-cell $1_X\in\CB(X,X)$;
\item[\rm (5)] for every quadruple $X,Y,Z,W$ of 0-cells, an \emph{associativity} natural isomorphism
               $$\bfig
               \square<1800,600>[\CB(Z,W)\times\CB(Y,Z)\times\CB(X,Y)`\CB(Y,W)\times\CB(X,Y)`\CB(Z,W)\times\CB(X,Z)`\CB(X,W);\circ_{YZW}\times 1_{\CB(X,Y)}`1_{\CB(Z,W)}\times \circ_{XYZ}`\circ_{XYW}`\circ_{XZW}]
               \morphism(400,150)/=>/<1000,300>[`;\al_{XYZW}]
               \efig$$
               between the two functors from $\CB(Z,W)\times\CB(Y,Z)\times\CB(X,Y)$ to $\CB(X,W)$ arising from the horizontal composition functor $\circ$, called the \emph{associator};
\item[\rm (6)] for every pair $X,Y$ of 0-cells, a \emph{left unit} natural isomorphism
               $$\bfig
               \qtriangle<1200,500>[1\times\CB(X,Y)`\CB(Y,Y)\times\CB(X,Y)`\CB(X,Y);u_Y\times 1_{\CB(X,Y)}`1_{\CB(X,Y)}`\circ_{XYY}]
               \morphism(1000,400)/=>/<-150,-150>[`;\lam_{XY}]
               \efig$$
               and a \emph{right unit} natural isomorphism
               $$\bfig
               \qtriangle<1200,500>[\CB(X,Y)\times 1`\CB(X,Y)\times\CB(X,X)`\CB(X,Y);1_{\CB(X,Y)}\times u_X`1_{\CB(X,Y)}`\circ_{XXY}]
               \morphism(1000,400)/=>/<-150,-150>[`;\rho_{XY}]
               \efig,$$
               called respectively the \emph{left unitor} and the \emph{right unitor};
\item[\rm (7)] the associativity coherence diagram (pentagon identity)
               $$\bfig
               \Atriangle/->`->`/<800,500>[k\circ(h\circ(g\circ f))`k\circ((h\circ g)\circ f)`(k\circ h)\circ(g\circ f);k\circ\al_{XYZW}`\al_{XZWV}`]
               \square(0,-500)/`->`->`->/<1600,500>[k\circ((h\circ g)\circ f)`(k\circ h)\circ(g\circ f)`(k\circ(h\circ g))\circ f`((k\circ h)\circ g)\circ f;`\al_{XYWV}`\al_{XYZV}`\al_{YZWV}\circ f]
               \efig$$
               is commutative for every quintuple of 0-cells $X,Y,Z,W,V\in\CB_0$ and every quadruple of 1-cells $f\in\CB(X,Y),g\in\CB(Y,Z),h\in\CB(Z,W),k\in\CB(W,V)$;
\item[\rm (8)] the unit coherence diagram (triangle identity)
               $$\bfig
               \qtriangle<1200,500>[g\circ(1_Y\circ f)`(g\circ 1_Y)\circ f`g\circ f;\al_{XYYZ}`g\circ\lam_{XY}`\rho_{YZ}\circ f]
               \efig$$
               is commutative for every triple of 0-cells $X,Y,Z\in\CB_0$ and every pair of 1-cells $f\in\CB(X,Y),g\in\CB(Y,Z)$.
\end{itemize}
\end{defn}

\begin{exmp} \label{bicategory_example}
\begin{itemize}
\item[\rm (1)] Each (strict) 2-category is a bicategory in which the associators and the unitors are identities.
\item[\rm (2)] A monoidal category is a bicategory with only one object.
\item[\rm (3)] \cite{Rosenthal1996} A quantaloid is a bicategory (a 2-category, indeed) in which there is at most one 2-cell between each pair of 1-cells.
\end{itemize}
\end{exmp}

As stated in the above example, a bicategory can be viewed as a many-object monoidal category. Thus, the theory of categories enriched over a monoidal category can be promoted further to categories enriched over a bicategory \cite{Betti1982,Betti1983,Gordon1997,Leinster2002,Walters1981}.

\begin{defn}
Let $\CB$ be a bicategory. A \emph{$\CB$-typed set} is a set $A$ together with a mapping $t:A\to\CB_0$ sending each elements $x\in A$ to its \emph{type} $tx\in\CB_0$.
\end{defn}

\begin{defn} \label{B_cat} \cite{Betti1982,Betti1983,Walters1981}
Let $\CB$ be a bicategory. A (small) \emph{$\CB$-category} (or category enriched over $\CB$) $\bbA$ consists of the following data:
\begin{itemize}
\item[\rm (1)] a $\CB$-typed set $\bbA_0$ of objects;
\item[\rm (2)] for every pair $x,y\in\bbA_0$, a 1-cell $\bbA(x,y)\in\CB(tx,ty)$ as the hom-arrow;
\item[\rm (3)] for every triple $x,y,z\in\bbA_0$, a 2-cell in $\CB(tx,tz)$
               $$c_{xyz}:\bbA(y,z)\circ\bbA(x,y)\to\bbA(x,z)$$
               as the composition;
\item[\rm (4)] for every $x\in\bbA_0$, a 2-cell in $\CB(tx,tx)$
               $$u_x:1_{tx}\to\bbA(x,x)$$
               as the unit;
\item[\rm (5)] the diagram expressing the associativity law
               $$\bfig
               \Atriangle/->`->`/<800,500>[\bbA(z,w)\circ(\bbA(y,z)\circ\bbA(x,y))`(\bbA(z,w)\circ\bbA(y,z))\circ\bbA(x,y)`\bbA(z,w)\circ\bbA(x,z);
               \al_{tx,ty,tz,tw}`\bbA(z,w)\circ c_{xyz}`]
               \square(0,-500)/`->`->`->/<1600,500>[(\bbA(z,w)\circ\bbA(y,z))\circ\bbA(x,y)`\bbA(z,w)\circ\bbA(x,z)`\bbA(y,w)\circ\bbA(x,y)`\bbA(x,w);`c_{yzw}\circ \bbA(x,y)`c_{xzw}`c_{xyw}]
               \efig$$
               is commutative for every quadruple $x,y,z,w\in\bbA_0$;
\item[\rm (6)] the diagrams expressing the unity law
               $$\bfig
               \qtriangle<1200,500>[1_{ty}\circ\bbA(x,y)`\bbA(y,y)\circ\bbA(x,y)`\bbA(x,y);u_y\circ\bbA(x,y)`\lam_{tx,ty}`c_{xyy}]
               \qtriangle(0,-800)<1200,500>[\bbA(x,y)\circ 1_{tx}`\bbA(x,y)\circ\bbA(x,x)`\bbA(x,y);\bbA(x,y)\circ u_x`\rho_{tx,ty}`c_{xxy}]
               \efig$$
               are commutative for every pair $x,y\in\bbA_0$.
\end{itemize}
\end{defn}

\begin{defn} \cite{Betti1982} \label{B-functor}
A \emph{$\CB$-functor} $F:\bbA\to\bbB$ between $\CB$-categories consists of the following data:
\begin{itemize}
\item[\rm (1)] a \emph{type-preserving map} $F:\bbA_0\to\bbB_0$ in the sense that $\forall x\in\bbA_0$, $tx=t(Fx)$;
\item[\rm (2)] for every pair $x,x'\in\bbA_0$, a 2-cell in $\CB(tx,tx')$
               $$F_{xx'}:\bbA(x,x')\to\bbB(Fx,Fx');$$
\item[\rm (3)] the diagram expressing the composition law
               $$\bfig
               \square<1400,500>[\bbA(x',x'')\circ\bbA(x,x')`\bbA(x,x'')`\bbB(Fx',Fx'')\circ\bbB(Fx,Fx')`\bbB(Fx,Fx'');(c_{\bbA})_{xx'x''}`F_{x'x''}\circ F_{xx'}`F_{xx''}`(c_{\bbB})_{Fx,Fx',Fx''}]
               \efig$$
               is commutative for every triple $x,x',x''\in\bbA_0$;
\item[\rm (4)] the diagram expressing the unity law
               $$\bfig
               \qtriangle<1000,500>[1_{tx}`\bbA(x,x)`\bbB(Fx,Fx);(u_{\bbA})_x`(u_{\bbB})_{Fx}`F_{xx}]
               \efig$$
               is commutative for every $x\in\bbA_0$.
\end{itemize}
\end{defn}

\begin{defn} \cite{Betti1982} \label{B-distributor}
A \emph{$\CB$-distributor} (also called \emph{$\CB$-bimodule} or \emph{$\CB$-profunctor}) $\phi:\bbA\oto\bbB$ between $\CB$-categories consists of the following data:
\begin{itemize}
\item[\rm (1)] for every pair $x\in\bbA_0$ and $y\in\bbB_0$, a 1-cell in $\phi(x,y)\in\CB(tx,ty)$;
\item[\rm (2)] for every triple $x\in\bbA_0$ and $y,y'\in\bbB_0$, a 2-cell in $\CB(tx,ty)$
               $$\phi_{xy'y}:\bbB(y',y)\circ\phi(x,y')\to\phi(x,y);$$
\item[\rm (3)] for every triple $x,x'\in\bbA_0$ and $y\in\bbB_0$, a 2-cell in $\CB(tx,ty)$
               $$\phi_{xx'y}:\phi(x',y)\circ\bbA(x,x')\to\phi(x,y);$$
\item[\rm (4)] the diagrams expressing the associativity law
               $$\bfig
               \square<1800,500>[\bbB(y'',y)\circ\bbB(y',y'')\circ\phi(x,y')`\bbB(y'',y)\circ\phi(x,y'')`\bbB(y',y)\circ\phi(x,y')`\phi(x,y);\bbB(y'',y)\circ\phi_{xy'y''}`
               (c_{\bbB})_{y'y''y}\circ\phi(x,y')`\phi_{xy''y}`\phi_{xy'y}]
               \square(0,-800)<1800,500>[\phi(x'',y)\circ\bbA(x',x'')\circ\bbA(x,x')`\phi(x',y)\circ\bbA(x,x')`\phi(x'',y)\circ\bbA(x,x'')`\phi(x,y);\phi_{x'x''y}\circ \bbA(x,x')`\phi(x'',y)\circ(c_{\bbA})_{xx'x''}`\phi_{xx'y}`\phi_{xx''y}] \square(0,-1600)<1800,500>[\bbB(y',y)\circ\phi(x',y')\circ\bbA(x,x')`\bbB(y',y)\circ\phi(x,y')`\phi(x',y)\circ\bbA(x,x')`\phi(x,y);\bbB(y',y)\circ\phi_{xx'y'}
               `\phi_{x'y'y}\circ\bbA(x,x')`\phi_{xy'y}`\phi_{xx'y}]
               \efig$$
               are commutative for every sextuple $x,x',x''\in\bbA_0$ and $y,y',y''\in\bbB_0$;
\item[\rm (6)] the diagrams expressing the unity law
               $$\bfig
               \qtriangle<1400,500>[1_{ty}\circ\phi(x,y)`\bbB(y,y)\circ\phi(x,y)`\phi(x,y);(u_{\bbB})_y\circ\phi(x,y)`\lam_{tx,ty}`\phi_{xyy}]
               \qtriangle(0,-800)<1400,500>[\phi(x,y)\circ 1_{tx}`\phi(x,y)\circ\bbA(x,x)`\phi(x,y);\phi(x,y)\circ (u_{\bbA})_x`\rho_{tx,ty}`\phi_{xxy}]
               \efig$$
               are commutative for every pair $x\in\bbA_0$ and $y\in\bbB_0$.
\end{itemize}
\end{defn}

It is easily seen from Definition \ref{B-functor} that $\CB$-functors can be composed in an obvious way. Thus, $\CB$-categories and $\CB$-functors constitute a category $\CB$-$\Cat$.

So, it is natural to ask that whether $\CB$-distributors have reasonable compositions, and give rise to a category $\CB$-$\Dist$ consisting of $\CB$-categories and $\CB$-distributors.

A straightforward idea comes from the composition of distributors between categories enriched over a symmetric monoidal closed category. Explicitly, the composition $\psi\circ\phi:\bbA\oto\bbC$ of $\CB$-distributors $\phi:\bbA\oto\bbB$ and $\psi:\bbB\oto\bbC$ is given by the following coequalizer diagram
$$\coprod\limits_{y,y'\in\bbB_0}\psi(y',z)\circ\bbB(y,y')\circ\phi(x,y)\two\coprod\limits_{y\in\bbB_0}\psi(y,z)\circ\phi(x,y)\to(\psi\circ\phi)(x,z)$$
for all $x\in\bbA_0$ and $z\in\bbC_0$. In particular, the small category $\CB(tx,tz)$ admits all set-indexed coproducts if each pair of $\CB$-distributors $\phi:\bbA\oto\bbB$ and $\psi:\bbB\oto\bbC$ can be composed. This implies that $\CB(tx,tz)$ is a (cocomplete) preordered set.

Therefore, the category $\CB$-$\Dist$ only makes sense when $\CB$ is (cocompletely) locally ordered.

Throughout this dissertation, the composition of distributors is essential for developing our results. Therefore, we consider categories enriched over a quantaloid $\CQ$, which is a (cocompletely) locally ordered bicategory, instead of categories enriched over a general bicategory $\CB$.

\chapter{Quantaloids and $\CQ$-categories} \label{Quantaloids_Q_category}

In this chapter, we go over the theory of quantaloids and $\CQ$-categories, and fix the notations that will be used in the sequel. We refer to \cite{Rosenthal1996} for the theory of quantaloids, and \cite{Heymans2010,Heymans2011,Stubbe2005,Stubbe2006} for the theory of categories enriched over a quantaloid.

\section{Quantaloids}

\begin{defn}
A \emph{quantaloid} $\CQ$ is a category enriched over the symmetric monoidal closed category $\Sup$.
\end{defn}

Explicitly, a quantaloid $\CQ$ is a category with a class of objects $\CQ_0$ such that $\CQ(X,Y)$ is a $\sup$-lattice for all $X,Y\in\CQ_0$, and the composition $\circ$ of morphisms preserves joins in both variables, i.e.,
\begin{equation} \label{quantaloid_distributive}
g\circ\Big(\bv_{i}f_i\Big)=\bv_{i}(g\circ f_i)\quad\text{and}\quad\Big(\bv_{j}g_j\Big)\circ f=\bv_{j}(g_j\circ f)
\end{equation}
for all $f,f_i\in\CQ(X,Y)$ and $g,g_j\in\CQ(Y,Z)$. We denote the top and the bottom element of the $\sup$-lattice $\CQ(X,Y)$ by $\top_{X,Y}$ and $\bot_{X,Y}$ respectively, and the identity arrow on $X\in\CQ_0$ by $1_X$.

\begin{rem} \label{quantale_oid}
Quantaloids are the \emph{horizontal categorification} or \emph{oidification} \cite{Baez1998} of quantales. The suffix \emph{-oid} of the word \emph{quantaloid} is short for \emph{oidified}, which means generalizing a certain type of one-object categories to such type of categories with more than one object. In other words, a quantaloid is a quantale with many objects, while a quantale is a one-object quantaloid. There are other similar examples:
\begin{itemize}
\item[\rm (1)] The oidification of groups are groupoids, i.e., categories with every morphism invertible.
\item[\rm (2)] The oidification of rings are ringoids, i.e., categories enriched over the category of abelian groups.
\item[\rm (3)] The oidification of monoids are just ordinary categories.
\item[\rm (4)] The oidification of monoidal categories are bicategories.
\end{itemize}
\end{rem}

Throughout this dissertation, $\CQ$ always denotes a quantaloid, while $\CQ_0$ and $\CQ_1$ stand for its class of objects and its class of morphisms, respectively.

\begin{defn}
For $f:X\to Y$, $g:Y\to Z$ and $h:X\to Z$ in a quantaloid $\CQ$, define the \emph{left implication} $h\lda f:Y\to Z$ and the \emph{right implication} $g\rda h:X\to Y$ by
$$h\lda f=\bv\{g':Y\to Z\mid g'\circ f\leq h\}$$
and
$$g\rda h=\bv\{f':X\to Y\mid g\circ f'\leq h\}.$$
\end{defn}

In other words, $h\lda f$ is the largest arrow in $\CQ(Y,Z)$ that satisfies $(h\lda f)\circ f\leq h$, and $g\rda h$ is the largest arrow in $\CQ(X,Y)$ that satisfies $g\circ(g\rda h)\leq h$.
$$\bfig
\btriangle[X`Y`Z;f`h`h\lda f]
\btriangle(1500,0)[X`Y`Z;g\rda h`h`g]
\place(150,150)[\twoar(1,1)]
\place(1650,150)[\twoar(1,1)]
\efig$$

It is easy to see that for each $f\in\CQ(X,Y)$ and $g\in\CQ(Y,Z)$, functions
$$-\lda f:\CQ(X,Z)\to\CQ(Y,Z):h\mapsto h\lda f,$$
$$g\rda -:\CQ(X,Z)\to\CQ(X,Y):h\mapsto g\rda h$$
are respective right adjoints of
$$-\circ f:\CQ(Y,Z)\to\CQ(X,Z):g\mapsto g\circ f,$$
$$g\circ -:\CQ(X,Y)\to\CQ(X,Z):f\mapsto g\circ f.$$

\begin{exmp} \label{quantaloid_example}
Some basic examples of quantaloids are listed below.
\begin{itemize}
\item[\rm (1)] A unital quantale is a one-object quantaloid, as stated in Remark \ref{quantale_oid}. In particular, the two-element Boolean algebra ${\bf 2}$ is a quantaloid.
\item[\rm (2)] The category $\Sup$ is itself a quantaloid, in which $\sup$-preserving maps are endowed with the pointwise order.
\item[\rm (3)] \cite{Rosenthal1996} The category $\Rel$ of sets and (binary) relations is a quantaloid, in which
               \begin{itemize}
               \item the local order between relations is given by the inclusion;
               \item for relations $R\subseteq A\times B$ and $S\subseteq B\times C$ between sets, the composition $S\circ R\subseteq A\times C$ is given by
               $$S\circ R=\{(x,z)\mid\exists y\in B,(x,y)\in R\text{ and }(y,z)\in S\};$$
               \item the identity relation on a set $A$ is the diagonal relation
               $$\De_A=\{(x,x)\mid x\in A\};$$
               \item for relations $R\subseteq A\times B$, $S\subseteq B\times C$ and $T\subseteq A\times C$, the left implication $T\lda R\subseteq B\times C$ and the right implication $S\rda T\subseteq A\times B$ are given by
               $$T\lda R=\{(y,z)\mid\forall x\in A,(x,y)\in R{}\Lra{}(x,z)\in T\}$$
               and
               $$S\rda T=\{(x,y)\mid\forall z\in C,(y,z)\in S{}\Lra{}(x,z)\in T\}.$$
               \end{itemize}
\item[\rm (4)] \cite{Pitts1988} In general, each Grothendieck topos $\CE$ \cite{Borceux1994c,Freyd1990,Goldblatt2006,Johnstone2002b} gives rise to a quantaloid $\Rel(\CE)$, in which
               \begin{itemize}
               \item the objects are the same as $\CE$;
               \item a morphism $\phi\in\Rel(\CE)(X,Y)$ is a relation from $X$ to $Y$, i.e., a subobject $\phi:R\rightarrowtail X\times Y$;
               \item the local order is given by the inclusion of subobjects;
               \item for relations $\phi:R\rightarrowtail X\times Y$ and $\psi:S\rightarrowtail Y\times Z$, the composition $\eta:T\rightarrowtail X\times Z$ is given by the pullback and image factorization in $\CE$:
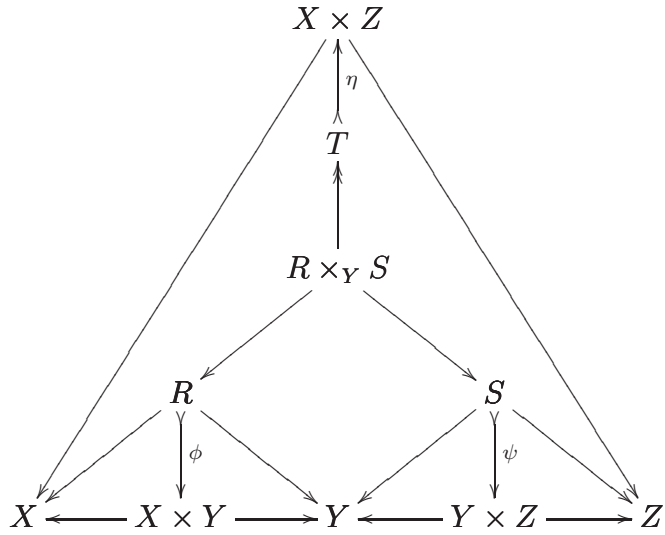
               $$\bfig
               \Atrianglepair|mrmmm|/->` >->`->`<-`->/<500,400>[R`X`X\times Y`Y;`\phi```]
               \Atrianglepair(1000,0)|mrmmm|/->` >->`->`<-`->/<500,400>[S`Y`Y\times Z`Z;`\psi```]
               \Atriangle(500,400)/->`->`/<500,400>[R\times_Y S`R`S;``]
               \morphism(1000,800)/ ->>/<0,400>[R\times_Y S`T;]
               \morphism(1000,1200)/ >->/<0,400>[T`X\times Z;\eta]
               \Atriangle/->`->`/<1000,1600>[X\times Z`X`Z;``]
               \efig$$
               \item the identity relation on an object $X$ is the diagonal subobject
               $$\De_X:X\rightarrowtail X\times X$$
               given by the universal property of the product $X\times X$:
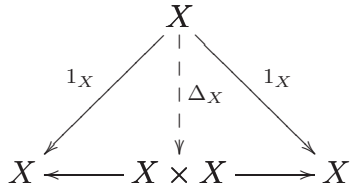
               $$\bfig
               \Atrianglepair|lrrbb|/->`-->`->`<-`->/[X`X`X\times X`X;1_X`\De_X`1_X``]
               \efig$$
               \end{itemize}
\end{itemize}
\end{exmp}

In the following two propositions, we list some useful formulas for calculating the compositions and implications of arrows in a quantaloid. They can be obtained by routine calculation and we omit the proof.

\begin{prop} \label{arrow_calculation}
In a quantaloid $\CQ$, the following properties hold for all $\CQ$-arrows $f,g,h,f_i,g_i,h_i$ whenever the compositions and implications make sense:
\begin{itemize}
\item[\rm (1)] $g\circ f\leq h\iff g\leq h\lda f\iff f\leq g\rda h$.
\item[\rm (2)] $(\displaystyle\bw\limits_i h_i)\lda f=\displaystyle\bw\limits_i(h_i\lda f)$, $g\rda(\displaystyle\bw\limits_i h_i)=\displaystyle\bw\limits_i(g\rda h_i)$.
\item[\rm (3)] $h\lda(\displaystyle\bv\limits_i f_i)=\displaystyle\bw\limits_i(h\lda f_i)$, $(\displaystyle\bv\limits_i g_i)\rda h=\displaystyle\bw\limits_i(g_i\rda h)$.
\item[\rm (4)] $(h\lda g)\circ(g\lda f)\leq h\lda f$, $(f\rda g)\circ(g\rda h)\leq f\rda h$.
\item[\rm (5)] $(h\lda f)\lda g=h\lda(g\circ f)$, $f\rda(g\rda h)=(g\circ f)\rda h$.
\item[\rm (6)] $(g\rda h)\lda f=g\rda(h\lda f)$.
\item[\rm (7)] $(h\lda f)\circ f\leq h$, $g\circ(g\rda h)\leq h$.
\item[\rm (8)] $h\circ(g\lda f)\leq(h\circ g)\lda f$, $(g\rda h)\circ f\leq g\rda(h\circ f)$.
\end{itemize}
\end{prop}

\begin{prop} \label{arrow_top_bottom}
In a quantaloid $\CQ$, the following properties hold for all $\CQ$-arrows $f:X\to Y$, $g:Y\to Z$ and $h:X\to Z$:
\begin{itemize}
\item[\rm (1)] $f\lda 1_X=1_Y\rda f=f$.
\item[\rm (2)] $\bot_{Y,Z}\circ f=g\circ \bot_{X,Y}=\bot_{X,Z}$.
\item[\rm (3)] $h\lda \bot_{X,Y}=\top_{Y,Z}$, $\bot_{Y,Z}\rda h=\top_{X,Y}$.
\item[\rm (4)] $\top_{X,Z}\lda f=\top_{Y,Z}$, $g\rda\top_{X,Z}=\top_{X,Y}$.
\end{itemize}
\end{prop}

\begin{defn} \label{quantaloid_adjunction}
An \emph{adjunction} in a quantaloid $\CQ$,  $f\dv g:X\rhu Y$ in symbols,  is a pair of $\CQ$-arrows $f:X\to Y$ and $g:Y\to X$ such that
$$1_X\leq g\circ f\quad\text{and}\quad f\circ g\leq 1_Y.$$
In this case, $f$ is a \emph{left adjoint} of $g$ and $g$ is a \emph{right adjoint} of $f$.
\end{defn}

The following two propositions can be derived by straightforward calculation.

\begin{prop} \label{fgf_f}
If $f\dv g:X\rhu Y$ in a quantaloid $\CQ$, then
$$f\circ g\circ f=f\quad\text{and}\quad g\circ f\circ g=g.$$
\end{prop}

\begin{prop} {\rm\cite{Heymans2010}} \label{adjoint_arrow_calculation}
If $f\dv g:X\rhu Y$ in a quantaloid $\CQ$, then the following identities hold for all $\CQ$-arrows $h,h'$ whenever the compositions and implications make sense:
\begin{itemize}
\item[\rm (1)] $h\circ f=h\lda g$, $g\circ h=f\rda h$.
\item[\rm (2)] $(f\circ h)\rda h'=h\rda(g\circ h')$, $(h'\circ f)\lda h=h'\lda(h\circ g)$.
\item[\rm (3)] $(h\rda h')\circ f=h\rda(h'\circ f)$, $g\circ(h'\lda h)=(g\circ h')\lda h$.
\item[\rm (4)] $g\circ(h\rda h')=(h\circ f)\rda h'$, $(h'\lda h)\circ f=h'\lda(g\circ h)$.
\end{itemize}
\end{prop}

The identities in Proposition \ref{adjoint_arrow_calculation} will be frequently applied to adjunctions of the form $F_{\natural}\dv F^{\natural}:\bbA\rhu\bbB$, the graph and cograph of a $\CQ$-functor $F:\bbA\to\bbB$ (will be defined in Section \ref{Q_distributors}).

The following corollary follows immediately from Proposition \ref{adjoint_arrow_calculation}(1).

\begin{cor} \label{adjoint_arrow_representation}
If $f\dv g:X\rhu Y$ in a quantaloid $\CQ$, then
$$g=f\rda 1_Y\quad\text{and}\quad f=1_Y\lda g.$$
In particular, the left adjoint and right adjoint of a $\CQ$-arrow are unique when they exist.
\end{cor}

\begin{defn}
Let $\FD=\{d_X:X\to X\mid X\in\CQ_0\}$ be a family of $\CQ$-arrows in a quantaloid $\CQ$.
\begin{itemize}
\item[\rm (1)] $\FD$ is called a \emph{cyclic family} if
               \begin{equation} \label{cyclic_def}
               d_X\lda f=f\rda d_Y
               \end{equation}
               for all $f:X\to Y$ in $\CQ$.
\item[\rm (2)] $\FD$ is called a \emph{dualizing family} if
               \begin{equation} \label{dualizing_def}
               (d_X\lda f)\rda d_X=f=d_Y\lda(f\rda d_Y)
               \end{equation}
               for all $f:X\to Y$ in $\CQ$.
\end{itemize}
\end{defn}

\begin{defn}
A \emph{Girard quantaloid} is a quantaloid with a cyclic dualizing family $\FD$ of $\CQ$-arrows.
\end{defn}

\begin{exmp} \label{Girard_quantaloid_example}
Some examples of Girard quantaloids are listed below.
\begin{itemize}
\item[\rm (1)] A Girard quantale \cite{Rosenthal1990} is a one-object Girard quantaloid.
\item[\rm (2)] \cite{Rosenthal1992} The quantaloid $\Rel$ of sets and relations (see Example \ref{quantaloid_example}(3)) is a Girard quantaloid. A cyclic dualizing family is given by
    $$\FD=\{\neg\De_A\subseteq A\times A\mid A\text{ is a set}\},$$
    where
    $$\neg\De_A=\{(x,x')\in A\times A\mid x\neq x'\}.$$
\item[\rm (3)] In Chapter \ref{Applications}, we will construct a quantaloid $\CQ$ from a divisible unital quantale $Q$ (see Proposition \ref{divisible_quantale_quantaloid}). If $(Q,\&)$ is a Boolean algebra (with $\&$ being $\wedge$) \cite{Birkhoff1948}, then the induced quantaloid $\CQ$ is a Girard quantaloid.
\end{itemize}
\end{exmp}

\begin{prop} \label{dualizing_bottom}
If $1_X$ is the top element of $\CQ(X,X)$ for all $X\in\CQ_0$ and $\FD=\{d_X:X\to X\mid X\in\CQ_0\}$ is a dualizing family, then $d_X=\bot_{X,X}$ for all $X\in\CQ_0$.
\end{prop}

\begin{proof}
For any $f:X\to X$,
\begin{align*}
f\lda d_X&=((d_X\lda f)\rda d_X)\lda d_X&\text{(Equation (\ref{dualizing_def}))}\\
&=(d_X\lda f)\rda(d_X\lda d_X)&\text{(Proposition \ref{arrow_calculation}(6))}\\
&=(d_X\lda f)\rda 1_X&(1_X=\top_{X,X})\\
&=1_X,&\text{(Proposition \ref{arrow_top_bottom}(4))}
\end{align*}
hence $d_X\leq f$ and consequently $d_X=\bot_{X,X}$.
\end{proof}

\begin{prop} \label{Girard_quantaloid_properties}
Suppose that $\CQ$ has a dualizing family $\FD=\{d_X:X\to X\mid X\in\CQ_0\}$. For $f,f_i:X\to Y$, $g:Y\to Z$, $h:X\to Z$ in $\CQ_1$:
\begin{itemize}
\item[\rm (1)] $\displaystyle\bv\limits_i(d_X\lda f_i)=d_X\lda(\displaystyle\bw\limits_i f_i)$, $\displaystyle\bv\limits_i(f_i\rda d_Y)=(\displaystyle\bw\limits_i f_i)\rda d_Y$.
\item[\rm (2)] $g\circ f=d_Z\lda(f\rda(g\rda d_Z))=((d_X\lda f)\lda g)\rda d_X$.
\item[\rm (3)] $h\lda f=(d_X\lda h)\rda(d_X\lda f)=d_Z\lda(f\circ(h\rda d_Z))$.
\item[\rm (4)] $g\rda h=(g\rda d_Z)\lda(h\rda d_Z)=((d_X\lda h)\circ g)\rda d_X$.
\item[\rm (5)] $(d_Y\lda g)\rda f=g\lda (f\rda d_Y)$.
\end{itemize}
\end{prop}

\begin{proof}
All the identities can be obtained by routine calculation. We just prove (5) for example. Note that $(d_Y\lda g)\rda f\leq g\lda (f\rda d_Y)$ follows from
$$((d_Y\lda g)\rda f)\circ(f\rda d_Y)\leq(d_Y\lda g)\rda d_Y=g,$$
and the reverse inequality follows from
$$(d_Y\lda g)\circ(g\lda (f\rda d_Y))\leq d_Y\lda(f\rda d_Y)=f.$$
\end{proof}

\section{$\CQ$-categories and $\CQ$-functors}

A quantaloid $\CQ$ is a locally ordered bicategory, which means that there is at most one 2-cell between each pair of $\CQ$-arrows, i.e., the order $\leq$. From now on, $\CQ$ always denotes a small quantaloid, i.e., $\CQ_0$ and $\CQ_1$ are sets.

\begin{defn} \label{Q_type}
A \emph{$\CQ$-typed set} is a set $A$ equipped with a map $t:A\to\CQ_0$ sending each elements $x\in A$ to its \emph{type} $tx\in\CQ_0$.
\end{defn}

\begin{defn} \label{Q_category}
A (small) \emph{$\CQ$-category} (or category enriched over $\CQ$) $\bbA$ consists of a $\CQ$-typed set $\bbA_0$ and hom-arrow $\bbA(x,y)\in\CQ(tx,ty)$ for all $x,y\in\bbA_0$, such that
\begin{itemize}
\item[\rm (1)] $1_{tx}\leq\bbA(x,x)$ for all $x\in\bbA_0$;
\item[\rm (2)] $\bbA(y,z)\circ\bbA(x,y)\leq\bbA(x,z)$ for all $x,y,z\in\bbA_0$.
\end{itemize}
\end{defn}

\begin{defn} \label{Q_functor}
A \emph{$\CQ$-functor} $F:\bbA\to\bbB$ between $\CQ$-categories is a map $F:\bbA_0\to\bbB_0$ such that
\begin{itemize}
\item[\rm (1)] $F$ is \emph{type-preserving} in the sense that $\forall x\in\bbA_0$, $tx=t(Fx)$;
\item[\rm (2)] $\forall x,x'\in\bbA_0$, $\bbA(x,x')\leq\bbB(Fx,Fx')$.
\end{itemize}
\end{defn}

A $\CQ$-functor $F:\bbA\to\bbB$ is \emph{fully faithful} if $\bbA(x,x')=\bbB(Fx,Fx')$  for all $x,x'\in\bbA_0$. Bijective fully faithful $\CQ$-functors are exactly the isomorphisms in the category $\CQ$-$\Cat$ of $\CQ$-categories and $\CQ$-functors.

\begin{exmp} \label{Q_category_example}
We list here some basic examples of $\CQ$-categories.
\begin{itemize}
\item[\rm (1)] For the two-element Boolean algebra ${\bf 2}$, ${\bf 2}$-categories are preordered sets, and ${\bf 2}$-functors are order-preserving maps.
\item[\rm (2)] Each $\CQ$-typed set $A$ gives rise to a \emph{discrete} $\CQ$-category $\bbA$ given by $\bbA_0=A$ and
               $$\bbA(x,y)=\begin{cases}
               1_{tx}, & x=y;\\
               \bot_{tx,ty}, & x\neq y.
               \end{cases}$$
\item[\rm (3)] For each $X\in\CQ_0$, $*_X$ is a $\CQ$-category with only one object $*$ of type $t*=X$ and hom-arrow $1_X$.
\item[\rm (4)] Let $\bbA$ be a $\CQ$-category, a $\CQ$-category $\bbB$ is a (full) $\CQ$-subcategory of $\bbA$ if $\bbB_0$ is a subset of $\bbA_0$ and $\bbB(x,y)=\bbA(x,y)$ for all $x,y\in\bbB_0$. In particular, for each $\CQ$-functor $F:\bbA\to\bbB$, the elements $\{y\in\bbB_0:\exists x\in\bbA_0,Fx=y\}$ is a subset of $\bbB_0$ and constitute a $\CQ$-subcategory $F(\bbA)$ of $\bbB$.
\end{itemize}
\end{exmp}

Given a $\CQ$-category $\bbA$, there is a natural underlying preorder $\leq$ on $\bbA_0$. For $x,y\in\bbA_0$,
$$x\leq y\iff tx=ty=X\text{ and }1_X\leq\bbA(x,y).$$
For each $X\in\CQ_0$, the objects in $\bbA$ with type $X$ constitute a subset of the underlying preordered set $\bbA_0$ and we denote it by $\bbA_X$. It is clear that the underlying preordered set $\bbA_0$ is the disjoint union of all $\bbA_X$, i.e., $x\leq y$ in $\bbA_0$ necessarily implies that $x$ and $y$ belong to the same $\bbA_X$.


\begin{prop} \label{underlying_order}
Let $\bbA$ be a $\CQ$-category and $x,y\in\bbA_0$, then the following conditions are equivalent.
\begin{itemize}
\item[\rm (1)] $x\leq y$.
\item[\rm (2)] $\bbA(y,z)\leq\bbA(x,z)$ for all $z\in\bbA_0$.
\item[\rm (3)] $\bbA(z,x)\leq\bbA(z,y)$ for all $z\in\bbA_0$.
\end{itemize}
\end{prop}

\begin{proof}
We prove the equivalence of (1) and (2) for example.

(1)${}\Lra{}$(2): If $tx=ty=X$ and $\bbA(x,y)\geq 1_X$, then for all $z\in\bbA_0$,
$$\bbA(y,z)\leq\bbA(y,z)\circ\bbA(x,y)\leq\bbA(x,z).$$

(2)${}\Lra{}$(1): In particular, $\bbA(y,y)\leq\bbA(x,y)$ implies $tx=ty=X$ and $\bbA(x,y)\geq 1_X$.
\end{proof}

Two objects $x,y$ in $\bbA$ are \emph{isomorphic} if $x\leq y$ and $y\leq x$, written $x\cong y$. $\bbA$ is \emph{skeletal} if no two different objects in $\bbA$ are isomorphic. The following proposition follows immediately from Proposition \ref{underlying_order}.

\begin{prop} \label{isomorphic_condition}
Let $\bbA$ be a $\CQ$-category and $x,y\in\bbA_0$, then the following conditions are equivalent.
\begin{itemize}
\item[\rm (1)] $x\cong y$.
\item[\rm (2)] $\bbA(x,z)=\bbA(y,z)$ for all $z\in\bbA_0$.
\item[\rm (3)] $\bbA(z,x)=\bbA(z,y)$ for all $z\in\bbA_0$.
\end{itemize}
\end{prop}

\begin{cor}
For any $\CQ$-category $\bbA$, the following conditions are equivalent.
\begin{itemize}
\item[\rm (1)] $\bbA$ is skeletal.
\item[\rm (2)] For any $x,y\in\bbA_0$, $x=y$ if and only if $\bbA(x,z)=\bbA(y,z)$ for all $z\in\bbA_0$.
\item[\rm (3)] For any $x,y\in\bbA_0$, $x=y$ if and only if $\bbA(z,x)=\bbA(z,y)$ for all $z\in\bbA_0$.
\end{itemize}
\end{cor}

The underlying preorders on $\CQ$-categories induce an preorder between $\CQ$-functors:
\begin{align*}
F\leq G:\bbA\to\bbB&\iff\forall x\in\bbA_0,Fx\leq Gx\ \text{in}\ \bbB_0\\
&\iff\forall x\in\bbA_0,1_{tx}\leq\bbB(Fx,Gx).
\end{align*}
Thus $\CQ\text{-}\Cat(\bbA,\bbB)$ becomes a preordered set, in which $F\leq G$ is an ``enriched natural transformation'' from $F$ to $G$. This makes $\CQ$-$\Cat$ a (locally ordered) 2-category. We denote $F\cong G:\bbA\to\bbB$ if $F\leq G$ and $G\leq F$. Furthermore, $\CQ\text{-}\Cat(\bbA,\bbB)$ becomes a poset (partially ordered set) if $\bbB$ is a skeletal $\CQ$-category.

\begin{prop} \label{fully_faithful_essentially_injective}
Let $F:\bbA\to\bbB$ be a fully faithful $\CQ$-functor between $\CQ$-categories. Then $F:\bbA_0\to\bbB_0$ is essentially injective in the sense that $x\cong x'$ whenever $Fx=Fx'$.
\end{prop}

\begin{proof}
Suppose $x,x'\in\bbA_0$ and $Fx=Fx'$, then $\bbA(x,x')=\bbB(Fx,Fx)\geq 1_{tx}$. Similarly $\bbA(x',x)\geq 1_{tx}$, hence $x\cong x'$.
\end{proof}

\begin{defn}
A pair of $\CQ$-functors $F:\bbA\to\bbB$ and $G:\bbB\to\bbA$ forms an \emph{adjunction}, written $F\dv G:\bbA\rhu\bbB$, if
$$1_{\bbA}\leq G\circ F\quad\text{and}\quad F\circ G\leq 1_{\bbB},$$
where $1_\bbA$ and $1_\bbB$ are respectively the identity $\CQ$-functors on $\bbA$ and $\bbB$. In this case, $F$ is called a \emph{left adjoint} of $G$ and $G$ a \emph{right adjoint} of $F$.
\end{defn}

Adjoint $\CQ$-functors can be viewed as adjoint morphisms in the 2-category $\CQ$-$\Cat$. The following proposition is easily derived from the definition.

\begin{prop} \label{FGF_F}
If $F\dv G:\bbA\rhu\bbB$ in $\CQ$-$\Cat$, then
$$F\circ G\circ F\cong F\quad\text{and}\quad G\circ F\circ G\cong G.$$
\end{prop}

We present below another useful characterization of adjoint $\CQ$-functors. It also implies that left adjoints and right adjoints of a $\CQ$-functor, when they exist, are essentially unique, i.e., unique up to isomorphism.

\begin{prop} \label{adjoint_condition}
Let $F:\bbA\to\bbB$ and $G:\bbB\to\bbA$ be type-preserving functions between $\CQ$-categories (not necessarily $\CQ$-functors). Then $F\dv G:\bbA\rhu\bbB$ if and only if $$\bbB(Fx,y)=\bbA(x,Gy)$$
for all $x\in\bbA$ and $y\in\bbB$.
\end{prop}

\begin{proof}
Suppose $F\dv G:\bbA\rhu\bbB$, then $1_{\bbA}\leq G\circ F$ and $F\circ G\leq 1_{\bbB}$. For all $x\in\bbA_0$ and $y\in\bbB_0$,
$$\bbB(Fx,y)\leq\bbA(GFx,Gy)\leq\bbA(GFx,Gy)\circ\bbA(x,GFx)\leq\bbA(x,Gy).$$
Similarly $\bbA(x,Gy)\leq\bbB(Fx,y)$. Hence $\bbB(Fx,y)=\bbA(x,Gy)$.

Conversely, first we show that $F$ and $G$ are $\CQ$-functors. For all $x,x'\in\bbA_0$,
$$\bbA(x,x')\leq\bbA(x',GFx')\circ\bbA(x,x')\leq\bbA(x,GFx')=\bbB(Fx,Fx').$$
Thus $F$ is a $\CQ$-functor. Similarly $G$ is a $\CQ$-functor.

Second, for all $x\in\bbA_0$ and $y\in\bbB_0$,
$$\bbA(x,GFx)=\bbB(Fx,Fx)\geq 1_{tx}~~\text{and}~~\bbB(FGy,y)=\bbA(Gy,Gy)\geq 1_{ty}.$$
Thus $1_{\bbA}\leq G\circ F$ and $F\circ G\leq 1_{\bbB}$.
\end{proof}

\section{$\CQ$-distributors} \label{Q_distributors}

$\CQ$-distributors between $\CQ$-categories generalize relations between sets.

\begin{defn}
A \emph{$\CQ$-distributor} (or \emph{$\CQ$-profunctor}, \emph{$\CQ$-bimodule}) $\phi:\bbA\oto\bbB$ between $\CQ$-categories is a map $\bbA_0\times\bbB_0\to\CQ_1$ that assigns to each pair  $(x,y)\in\bbA_0\times\bbB_0$ a $\CQ$-arrow $\phi(x,y)\in\CQ(tx,ty)$, such that
\begin{itemize}
\item[\rm (1)] $\forall x\in\bbA_0$, $\forall y,y'\in\bbB_0$, $\bbB(y',y)\circ\phi(x,y')\leq\phi(x,y)$;
\item[\rm (2)] $\forall x,x'\in\bbA_0$, $\forall y\in\bbB_0$, $\phi(x',y)\circ\bbA(x,x')\leq\phi(x,y)$.
\end{itemize}
\end{defn}

The following proposition is an analogue of Example \ref{quantaloid_example}(3)(4) for $\CQ$-categories.

\begin{prop} \label{Dist_quantaloid} {\rm\cite{Stubbe2005}}
$\CQ$-categories and $\CQ$-distributors constitute a quantaloid $\CQ$-$\Dist$ in which
\begin{itemize}
\item[\rm (1)] the local order is pointwise, i.e., for $\CQ$-distributors $\phi,\psi:\bbA\oto\bbB$,
    $$\phi\leq\psi\iff\forall x\in\bbA_0,\forall y\in\bbB_0,\phi(x,y)\leq\psi(x,y);$$
\item[\rm (2)] the composition $\psi\circ\phi:\bbA\oto\bbC$ of $\CQ$-distributors $\phi:\bbA\oto\bbB$ and $\psi:\bbB\oto\bbC$ is given by
    $$\forall x\in\bbA_0,\forall z\in\bbC_0,(\psi\circ\phi)(x,z)=\bv_{y\in\bbB_0}\psi(y,z)\circ\phi(x,y);$$
\item[\rm (3)] the identity $\CQ$-distributor on a $\CQ$-category $\bbA$ is the hom-arrow of $\bbA$ and will be denoted by $\bbA:\bbA\oto\bbA$;
\item[\rm (4)] for $\CQ$-distributors $\phi:\bbA\oto\bbB$, $\psi:\bbB\oto\bbC$ and $\eta:\bbA\oto\bbC$, the left implication $\eta\lda\phi:\bbB\oto\bbC$ and the right implication $\psi\rda\eta:\bbA\oto\bbB$ are given by
    $$\forall y\in\bbB_0,\forall z\in\bbC_0,(\eta\lda\phi)(y,z)=\bw_{x\in\bbA_0}\eta(x,z)\lda\phi(x,y)$$
    and
    $$\forall x\in\bbA_0,\forall y\in\bbB_0,(\psi\rda\eta)(x,y)=\bw_{z\in\bbC_0}\psi(y,z)\rda\eta(x,z).$$
\end{itemize}
\end{prop}

As a special case of Definition \ref{quantaloid_adjunction}, a pair of $\CQ$-distributors $\phi:\bbA\oto\bbB$ and $\psi:\bbB\oto\bbA$ forms an adjunction $\phi\dv\psi:\bbA\rhu\bbB$ in the quantaloid $\CQ$-$\Dist$ if
$$\bbA\leq\psi\circ\phi\quad\text{and}\quad\phi\circ\psi\leq\bbB.$$

Every $\CQ$-functor $F:\bbA\to\bbB$ induces an adjunction $F_{\nat}\dv F^{\nat}:\bbA\rhu\bbB$ in $\CQ$-$\Dist$ with
$$F_{\nat}(x,y)=\bbB(Fx,y)\quad\text{and}\quad F^{\nat}(y,x)=\bbB(y,Fx)$$
for all $x\in\bbA_0$ and $y\in\bbB_0$. The $\CQ$-distributors $F_{\nat}:\bbA\oto\bbB$ and $F^{\nat}:\bbB\oto\bbA$ are called the \emph{graph} and \emph{cograph} of $F$, respectively.

It follows immediately from Proposition \ref{underlying_order} that for each pair of $\CQ$-functors $F,G:\bbA\to\bbB$,
\begin{equation} \label{functor_graph_order}
F\leq G\iff G_{\nat}\leq F_{\nat}\iff F^{\nat}\leq G^{\nat}.
\end{equation}

\begin{prop}  {\rm\cite{Stubbe2005}} \label{graph_cograph_functor}
Let $F:\bbA\to\bbB$ and $G:\bbB\to\bbC$ be $\CQ$-functors between $\CQ$-categories.
\begin{itemize}
\item[\rm (1)] $(1_{\bbA})_{\nat}=(1_{\bbA})^{\nat}=\bbA$.
\item[\rm (2)] $G_{\nat}\circ F_{\nat}=(G\circ F)_{\nat}$, $F^{\nat}\circ G^{\nat}=(G\circ F)^{\nat}$.
\end{itemize}
\end{prop}

This proposition gives rise to a functor
$$(-)_{\nat}:\CQ\text{-}\Cat\to\CQ\text{-}\Dist$$
and a contravariant functor
$$(-)^{\nat}:(\CQ\text{-}\Cat)^{\op}\to\CQ\text{-}\Dist.$$

\begin{prop} \label{adjoint_graph} {\rm\cite{Stubbe2005}}
Let $F:\bbA\to\bbB$ and $G:\bbB\to\bbA$ be a pair of $\CQ$-functors. The following conditions are equivalent:
\begin{itemize}
\item[\rm (1)] $F\dv G:\bbA\rhu\bbB$.
\item[\rm (2)] $F_{\nat}=G^{\nat}$.
\item[\rm (3)] $G_{\nat}\dv F_{\nat}:\bbB\rhu\bbA$ in $\CQ$-$\Dist$.
\item[\rm (4)] $G^{\nat}\dv F^{\nat}:\bbA\rhu\bbB$ in $\CQ$-$\Dist$.
\end{itemize}
\end{prop}

\begin{proof}
The equivalence of (1) and (2) is a reformulation of Proposition \ref{adjoint_condition}. The equivalence of (2), (3) and (4) follows immediately from Corollary \ref{adjoint_arrow_representation}.
\end{proof}

\begin{prop} \label{fully_faithful_graph_cograph}
Let $F:\bbA\to\bbB$ be a $\CQ$-functor.
\begin{itemize}
\item[\rm (1)] If $F$ is fully faithful, then $F^{\nat}\circ F_{\nat}=\bbA$.
\item[\rm (2)] If $F$ is essentially surjective in the sense that there is some $x\in\bbA_0$ such that $Fx\cong y$ in $\bbB$ for all $y\in\bbB_0$, then $F_{\nat}\circ F^{\nat}=\bbB$.
\end{itemize}
\end{prop}

\begin{proof}
(1) If $F$ is fully faithful, then for all $x,x'\in\bbA_0$,
$$(F^{\nat}\circ F_{\nat})(x,x')=\bv_{y\in\bbB_0}\bbB(y,Fx')\circ\bbB(Fx,y)=\bbB(Fx,Fx')=\bbA(x,x').$$

(2) If $F$ is essentially surjective, then for all $y,y'\in\bbB_0$, there is some $x\in\bbA_0$ such that $Fx\cong y$. Thus
\begin{align*}
(F_{\nat}\circ F^{\nat})(y,y')&=\bv_{a\in\bbA_0}\bbB(Fa,y')\circ\bbB(y,Fa)\\
&\geq\bbB(Fx,y')\circ\bbB(y,Fx)\\
&=\bbB(y,y')\circ\bbB(y,y)\\
&\geq\bbB(y,y').
\end{align*}
Since $F_{\nat}\circ F^{\nat}\leq\bbB$ holds trivially, it follows that $F_{\nat}\circ F^{\nat}=\bbB$.
\end{proof}

\begin{defn} \label{presheaf_def}
\begin{itemize}
\item[\rm (1)] A \emph{contravariant presheaf} on  a $\CQ$-category $\bbA$ is a $\CQ$-distributor $\mu:\bbA\oto *_X$ with $X\in\CQ_0$.
\item[\rm (2)] A \emph{covariant presheaf} on  a $\CQ$-category $\bbA$ is a $\CQ$-distributor $\mu:*_X\oto\bbA$ with $X\in\CQ_0$.
\end{itemize}
\end{defn}

Contravariant presheaves on a $\CQ$-category $\bbA$ constitute a $\CQ$-category $\PA$ in which
\begin{equation} \label{PA_arrow}
t\mu=X\quad\text{and}\quad\PA(\mu,\lam)=\lam\lda\mu
\end{equation}
for all $\mu:\bbA\oto *_X$ and $\lam:\bbA\oto *_{Y}$ in $(\PA)_0$.

Dually, covariant presheaves on $\bbA$ constitute a $\CQ$-category $\PdA$ in which
\begin{equation} \label{PdA_arrow}
t\mu=X\quad\text{and}\quad\PdA(\mu,\lam)=\lam\rda\mu
\end{equation}
for all $\mu:*_X\oto\bbA$ and $\lam:*_{Y}\oto\bbA$ in $(\PdA)_0$.

$$\bfig
\morphism<600,300>[\bbA`*_{t\mu};\mu]
\morphism|b|<600,-300>[\bbA`*_{t\lam};\lam]
\morphism(600,300)|r|<0,-600>[*_{t\mu}`*_{t\lam};\PA(\mu,\lam)]
\place(300,150)[\circ] \place(300,-150)[\circ] \place(600,0)[\circ]
\morphism(1800,300)<600,-300>[*_{t\mu}`\bbA;\mu]
\morphism(1800,-300)|b|<600,300>[*_{t\lam}`\bbA;\lam]
\morphism(1800,300)|l|<0,-600>[*_{t\mu}`*_{t\lam};\PdA(\mu,\lam)]
\place(2100,150)[\circ] \place(2100,-150)[\circ] \place(1800,0)[\circ]
\efig$$

In particular, we denote $\CP(*_X)=\CP X$ and $\CPd (*_X)=\CPd X$ for each $X\in\CQ_0$.

\begin{prop} \label{PA_skeletal}
For each $\CQ$-category $\bbA$, $\PA$ and $\PdA$ are both skeletal $\CQ$-categories. In particular, $\CP X$ and $\CPd X$ are both skeletal $\CQ$-categories for each $X\in\CQ_0$.
\end{prop}

\begin{rem} \label{PdA_QDist_order}
For each $\CQ$-category $\bbA$, it follows from the definition that the underlying preorder in $\PA$ coincides with the local order in $\CQ$-$\Dist$, while the underlying preorder in $\PdA$ is  the \emph{reverse} local order in $\CQ$-$\Dist$. That is to say, for all $\mu,\lam\in\PdA$, we have
$$\mu\leq\lam\ \text{in}\ (\PdA)_0\iff\lam\leq\mu\ \text{in}\ \CQ\text{-}\Dist.$$
In order to get rid of the confusion about the symbol $\leq$, from now on we make the convention that the symbol $\leq$ between $\CQ$-distributors always denotes the local order in $\CQ$-$\Dist$ if not otherwise specified.
\end{rem}

Given a $\CQ$-category $\bbA$ and $a\in\bbA_0$, write $\sY a$ for the $\CQ$-distributor
\begin{equation} \label{Yoneda_def}
\bbA\oto *_{ta},\quad x\mapsto\bbA(x,a);
\end{equation}
write $\sYd a$ for the $\CQ$-distributor
\begin{equation} \label{coYoneda_def}
*_{ta}\oto\bbA,\quad x\mapsto\bbA(a,x).
\end{equation}

The following  lemma implies that both $\sY:\bbA\to\PA, a\mapsto\sY a$ and $\sYd:\bbA\to\PdA, a\mapsto\sYd a$ are fully faithful $\CQ$-functors (hence embeddings if $\bbA$ is skeletal). Thus, $\sY$ and $\sYd$ are called respectively the \emph{Yoneda embedding} and the \emph{co-Yoneda embedding}.

\begin{lem}[Yoneda] \label{Yoneda_lemma} {\rm\cite{Stubbe2005}}
For all $a\in\bbA_0$, $\mu\in\PA$ and $\lam\in\PdA$,
$$\PA(\sY a,\mu)=\mu(a)\quad\text{and}\quad\PdA(\lam,\sYd a)=\lam(a).$$
\end{lem}

For each $\CQ$-distributor $\phi:\bbA\oto\bbB$ and $x\in\bbA_0,y\in\bbB_0$, write $\phi(x,-)$ for the $\CQ$-distributor  $\phi\circ\sYd_{\bbA}x: *_{tx}\oto\bbA\oto\bbB$; and write $\phi(-,y)$ for the $\CQ$-distributor ${\sf Y}_{\bbB}y\circ\phi:\bbA\oto\bbB\oto*_{ty}$. Then the Yoneda lemma can be phrased as the commutativity of the following diagrams:
$$\bfig
\ptriangle/->`<-`<-/[\PA`*_{t\mu}`\bbA;\PA(-,\mu)`\sY_{\nat}`\mu] \place(250,500)[\circ]\place(0,250)[\circ]\place(250,250)[\circ]
\ptriangle(1200,0)/<-`->`->/[\PdA`*_{t\lam}`\bbA;\PdA(\lam,-)`(\sYd)^{\nat}`\lam] \place(1450,500)[\circ]\place(1200,250)[\circ]\place(1450,250)[\circ]
\efig$$
That is,
$$\mu=\PA(\sY-,\mu)=\sY_{\nat}(-,\mu)$$
and
$$\lam=\PdA(\lam,\sYd-)=(\sYd)^{\nat}(\lam,-).$$

\begin{rem} \label{distributor_notion}
Given $\CQ$-distributors $\phi:\bbA\oto\bbB$, $\psi:\bbB\oto\bbC$ and $\eta:\bbA\oto\bbC$, one can form $\CQ$-distributors such as $\phi(x,-)$, $\eta\lda\phi$, $\eta\rda\psi(y,-)$,  etc. We list here some basic formulas related to these $\CQ$-distributors that will be used in the sequel.
\begin{itemize}
\item[\rm (1)] $\forall x\in\bbA_0$, $\forall z\in\bbC_0$, $(\psi\circ\phi)(x,z)=\psi(-,z)\circ\phi(x,-)$;
\item[\rm (2)] $\forall y\in\bbA_0$, $\forall z\in\bbC_0$, $(\eta\lda\phi)(y,z)=\eta(-,z)\lda\phi(-,y)$;
\item[\rm (3)] $\forall x\in\bbA_0$, $\forall y\in\bbB_0$, $(\psi\rda\eta)(x,y)=\psi(y,-)\rda\eta(x,-)$;
\item[\rm (4)] $\forall x\in\bbA_0$, $(\psi\circ\phi)(x,-)=\psi\circ\phi(x,-)$;
\item[\rm (5)] $\forall z\in\bbC_0$, $(\psi\circ\phi)(-,z)=\psi(-,z)\circ\phi$;
\item[\rm (6)] $\forall y\in\bbB_0$, $(\eta\lda\phi)(y,-)=\eta\lda\phi(-,y)$;
\item[\rm (7)] $\forall z\in\bbC_0$, $(\eta\lda\phi)(-,z)=\eta(-,z)\lda\phi$;
\item[\rm (8)] $\forall y\in\bbB_0$, $(\psi\rda\eta)(-,y)=\psi(y,-)\rda\eta$;
\item[\rm (9)] $\forall x\in\bbA_0$, $(\psi\rda\eta)(x,-)=\psi\rda\eta(x,-)$.
\end{itemize}
\end{rem}

As an application of Remark \ref{distributor_notion}, we derive some useful formulas for graphs and cographs of $\CQ$-functors.

\begin{prop} \label{graph_cograph_distributor}
Let $F:\bbA\to\bbB$ be a $\CQ$-functor and $\phi:\bbB\oto\bbC$, $\psi:\bbC\oto\bbB$ be $\CQ$-distributors. Then
\begin{itemize}
\item[\rm (1)] $\phi\circ F_{\nat}=\phi(F-,-)=\phi\lda F^{\nat}$;
\item[\rm (2)] $F^{\nat}\circ\psi=\psi(-,F-)=F_{\nat}\rda\psi$.
\end{itemize}
\end{prop}

\begin{proof}
We prove (1) for example. Although $\phi\circ F_{\nat}=\phi\lda F^{\nat}$ follows immediately from Proposition \ref{adjoint_arrow_calculation}(1), we present here a proof that they are respectively equal to $\phi(F-,-)$ as an illustration for Remark \ref{distributor_notion}. For each $x\in\bbA_0$ and $z\in\bbC_0$,
\begin{align*}
(\phi\circ F_{\nat})(x,z)&=\phi(-,z)\circ F_{\nat}(x,-)&\text{(Remark \ref{distributor_notion}(1))}\\
&=\phi(-,z)\circ\bbB(Fx,-)\\
&=(\phi(-,z)\circ\bbB)(Fx,-)&\text{(Remark \ref{distributor_notion}(4))}\\
&=\phi(-,z)(Fx,-)\\
&=\phi(Fx,z)\\
&=(\phi\lda\bbB)(Fx,z)\\
&=\phi(-,z)\lda\bbB(-,Fx)&\text{(Remark \ref{distributor_notion}(2))}\\
&=\phi(-,z)\lda F^{\nat}(-,x)\\
&=(\phi\lda F^{\nat})(x,z).&\text{(Remark \ref{distributor_notion}(2))}
\end{align*}
Thus $\phi\circ F_{\nat}=\phi(F-,-)=\phi\lda F^{\nat}$.
\end{proof}

\begin{defn} \label{direct_inverse_image_def}
Let $F:\bbA\to\bbB$ be a $\CQ$-functor.
\begin{itemize}
\item[\rm (1)] The \emph{contravariant direct image $\CQ$-functor} of $F$ is a $\CQ$-functor $F^{\ra}:\PA\to\PB$ between the $\CQ$-categories of contravariant presheaves given by
    $$F^{\ra}(\mu)=\mu\circ F^{\nat}.$$
\item[\rm (2)] The \emph{contravariant inverse image $\CQ$-functor} of $F$ is a $\CQ$-functor $F^{\la}:\PB\to\PA$ between the $\CQ$-categories of contravariant presheaves given by
    $$F^{\la}(\lam)=\lam\circ F_{\nat}.$$
\item[\rm (3)] The \emph{covariant direct image $\CQ$-functor} of $F$ is a $\CQ$-functor $F^{\nra}:\PdA\to\PdB$ between the $\CQ$-categories of covariant presheaves given by
    $$F^{\nra}(\mu)=F_{\nat}\circ\mu.$$
\item[\rm (4)] The \emph{covariant inverse image $\CQ$-functor} of $F$ is a $\CQ$-functor $F^{\nla}:\PdB\to\PdA$ between the $\CQ$-categories of covariant presheaves given by
    $$F^{\nla}(\lam)=F^{\nat}\circ\lam.$$
\end{itemize}
\end{defn}

We illustrate Definition \ref{direct_inverse_image_def} through the following commutative diagrams:
$$\bfig
\btriangle[\bbB`\bbA`*_{t\mu};F^{\nat}`F^{\ra}(\mu)`\mu]
\place(0,250)[\circ] \place(250,0)[\circ] \place(250,250)[\circ]
\btriangle(1500,0)[\bbA`\bbB`*_{t\lam};F_{\nat}`F^{\la}(\lam)`\lam]
\place(1500,250)[\circ] \place(1750,0)[\circ] \place(1750,250)[\circ]
\efig$$
$$\bfig
\btriangle[*_{t\mu}`\bbA`\bbB;\mu`F^{\nra}(\mu)`F_{\nat}]
\place(0,250)[\circ] \place(250,0)[\circ] \place(250,250)[\circ]
\btriangle(1500,0)[*_{t\lam}`\bbB`\bbA;\lam`F^{\nla}(\lam)`F^{\nat}]
\place(1500,250)[\circ] \place(1750,0)[\circ] \place(1750,250)[\circ]
\efig$$

\begin{prop} \label{Fra_injective_surjective}
Let $F:\bbA\to\bbB$ be a $\CQ$-functor.
\begin{itemize}
\item[\rm (1)] If $F$ is fully faithful, then
               \begin{itemize}
               \item[\rm (i)] $F^{\ra}:\PA\to\PB$ and $F^{\nra}:\PdA\to\PdB$ are fully faithful and injective.
               \item[\rm (ii)] $F^{\la}:\PB\to\PA$ and $F^{\nla}:\PdB\to\PdA$ are surjective.
               \end{itemize}
\item[\rm (2)] If $F$ is essentially surjective, then
               \begin{itemize}
               \item[\rm (i)] $F^{\ra}:\PA\to\PB$ and $F^{\nra}:\PdA\to\PdB$ are surjective.
               \item[\rm (ii)] $F^{\la}:\PB\to\PA$ and $F^{\nla}:\PdB\to\PdA$ are fully faithful and injective.
               \end{itemize}
\end{itemize}
\end{prop}

\begin{proof}
We prove (1) for example. If $F$ is fully faithful, then it follows from Proposition \ref{fully_faithful_graph_cograph}(1) that $F^{\nat}\circ F_{\nat}=\bbA$.

(i) For all $\mu,\mu'\in\PA$,
\begin{align*}
\PB(F^{\ra}(\mu),F^{\ra}(\mu'))&=F^{\ra}(\mu')\lda F^{\ra}(\mu)&(\text{Equation (\ref{PA_arrow})})\\
&=(\mu'\circ F^{\nat})\lda(\mu\circ F^{\nat})&(\text{Definition \ref{direct_inverse_image_def}})\\
&=(\mu'\circ F^{\nat}\circ F_{\nat})\lda\mu&(\text{Proposition \ref{adjoint_arrow_calculation}(2)})\\
&=\mu'\lda\mu&(\text{Proposition \ref{fully_faithful_graph_cograph}(1)})\\
&=\PA(\mu,\mu').&(\text{Equation (\ref{PA_arrow})})
\end{align*}
Thus $F^{\ra}:\PA\to\PB$ is fully faithful. Since $\PA$ is skeletal, by Proposition \ref{fully_faithful_essentially_injective} we obtain that $F^{\ra}$ is injective. Similarly $F^{\nra}:\PdA\to\PdB$ is fully faithful and injective.

(ii) For all $\mu\in\PA$,
$$F^{\la}\circ F^{\ra}(\mu)=\mu\circ F^{\nat}\circ F_{\nat}=\mu\circ\bbA=\mu.$$
Thus $F^{\la}:\PB\to\PA$ is surjective. Similarly $F^{\nla}:\PdB\to\PdA$ is surjective.
\end{proof}

\begin{prop} \label{direct_inverse_image_adjunction}
For each $\CQ$-functor $F:\bbA\to\bbB$,
$$F^{\ra}\dv F^{\la}:\PA\rhu\PB$$
and
$$F^{\nla}\dv F^{\nra}:\PdB\rhu\PdA.$$
\end{prop}

\begin{proof}
For all $\mu\in\PA$, $\lam\in\PB$,
\begin{align*}
\PB(F^{\ra}(\mu),\lam)&=\lam\lda(\mu\circ F^{\nat})&\text{(Formula (\ref{PA_arrow}))}\\
&=(\lam\circ F_{\nat})\lda\mu&\text{(Proposition \ref{adjoint_arrow_calculation}(2))}\\
&=\PA(\mu,F^{\la}(\lam)).&\text{(Formula (\ref{PA_arrow}))}
\end{align*}
Thus $F^{\ra}\dv F^{\la}:\PA\rhu\PB$. That $F^{\nla}\dv F^{\nra}:\PdB\rhu\PdA$ can be deduced similarly.
\end{proof}

The following proposition is an immediate consequence of Proposition \ref{graph_cograph_distributor}.

\begin{prop} \label{F_la_lam_Fx}
Let $F:\bbA\to\bbB$ be a $\CQ$-functor.
\begin{itemize}
\item[\rm (1)] For all $\lam\in\PB$ and $x\in\bbA_0$,
$$F^{\la}(\lam)(x)=\lam(Fx)\in \CQ(tx,t\lam).$$
\item[\rm (2)] For all $\lam\in\PdB$ and $x\in\bbA_0$,
$$F^{\nla}(\lam)(x)=\lam(Fx)\in\CQ(t\lam, tx).$$
\end{itemize}
\end{prop}

We would like to stress that
\begin{equation} \label{Fla_Fra_adjuntion}
\mu\leq F^{\la}\circ F^{\ra}(\mu)\quad\text{and}\quad F^{\ra}\circ F^{\la}(\lam)\leq\lam
\end{equation}
for all $\mu\in\PA$ and $\lam\in\PB$; whereas
\begin{equation} \label{FLa_FRa_adjuntion}
\nu\leq F^{\nla}\circ F^{\nra}(\nu)\quad\text{and}\quad F^{\nra}\circ F^{\nla}(\ga)\leq\ga
\end{equation}
for all $\nu\in\PdA$ and $\gamma\in\PdB$ by Remark \ref{PdA_QDist_order}.

\begin{prop} \label{Fra_composition}
Let $F:\bbA\to\bbB$ and $G:\bbB\to\bbC$ be $\CQ$-functors between $\CQ$-categories.
\begin{itemize}
\item[\rm (1)] $G^{\ra}\circ F^{\ra}=(G\circ F)^{\ra}:\PA\to\PC$.
\item[\rm (2)] $F^{\la}\circ G^{\la}=(G\circ F)^{\la}:\PC\to\PA$.
\item[\rm (3)] $G^{\nra}\circ F^{\nra}=(G\circ F)^{\nra}:\PdA\to\PdC$.
\item[\rm (4)] $F^{\nla}\circ G^{\nla}=(G\circ F)^{\nla}:\PdC\to\PdA$.
\end{itemize}
\end{prop}

\begin{proof}
Straightforward calculation by help of Proposition \ref{graph_cograph_functor}.
\end{proof}

This proposition gives rise to two functors and two contravariant functors:
\begin{itemize}
\item[\rm (1)] $\CP:\CQ\text{-}\Cat\to\CQ\text{-}\Cat$ that sends a $\CQ$-functor $F:\bbA\to\bbB$ to its contravariant direct image $\CQ$-functor $F^{\ra}:\PA\to\PB$;
\item[\rm (2)] $\CPd:\CQ\text{-}\Cat\to\CQ\text{-}\Cat$ that sends a $\CQ$-functor $F:\bbA\to\bbB$ to its covariant direct image $\CQ$-functor $F^{\nra}:\PdA\to\PdB$;
\item[\rm (3)] $\CP^{\op}:(\CQ\text{-}\Cat)^{\op}\to\CQ\text{-}\Cat$ that sends a $\CQ$-functor $F:\bbA\to\bbB$ to its contravariant inverse image $\CQ$-functor $F^{\la}:\PB\to\PA$; and
\item[\rm (4)] $(\CPd)^{\op}:(\CQ\text{-}\Cat)^{\op}\to\CQ\text{-}\Cat$ that sends a $\CQ$-functor $F:\bbA\to\bbB$ to its covariant inverse image $\CQ$-functor $F^{\nla}:\PdB\to\PdA$.
\end{itemize}

\begin{prop} \label{Yoneda_natural}
Let $\sY$ and $\sYd$ be functions that assign to each $\CQ$-category $\bbA$, respectively, the Yoneda embedding $\sY_{\bbA}:\bbA\to\PA$ and the co-Yoneda embedding $\sYd_{\bbA}:\bbA\to\PdA$.
\begin{itemize}
\item[\rm (1)] $\sY=\{\sY_{\bbA}\}$ is a natural transformation from the identity functor on $\CQ$-$\Cat$ to $\CP$.
\item[\rm (2)] $\sYd=\{\sYd_{\bbA}\}$ is a natural transformation from the identity functor on $\CQ$-$\Cat$ to $\CPd$.
\end{itemize}
\end{prop}

\begin{proof}
(1) We show that the diagram
$$\bfig
\square[\bbA`\bbB`\PA`\PB;F`\sY_{\bbA}`\sY_{\bbB}`F^{\ra}]
\efig$$
is commutative for each $\CQ$-functor $F:\bbA\to\bbB$.
Indeed, for all $x\in\bbA_0$,
\begin{align*}
\sY_{\bbB}\circ Fx&=\bbB(-,Fx)&\text{(Formula (\ref{Yoneda_def}))}\\
&=F^{\nat}(-,x)&(\text{Definition of}\ F^{\nat})\\
&=(\bbA\circ F^{\nat})(-,x)\\
&=\bbA(-,x)\circ F^{\nat}&(\text{Remark \ref{distributor_notion}(5)})\\
&=F^{\ra}\circ\sY_{\bbA}x.&(\text{Definition \ref{direct_inverse_image_def}})
\end{align*}

(2) Similar to (1), one can prove that the diagram
$$\bfig
\square[\bbA`\bbB`\PdA`\PdB;F`\sYd_{\bbA}`\sYd_{\bbB}`F^{\nra}]
\efig$$
is commutative for each $\CQ$-functor $F:\bbA\to\bbB$.
\end{proof}

\section{Weighted colimits and limits}

In order to describe the cocompleteness and completeness of $\CQ$-categories, we first define weighted colimits and limits in $\CQ$-categories.

\begin{defn} \label{limit_colimit_def}
Let $F:\bbA\to\bbB$ be a $\CQ$-functor.
\begin{itemize}
\item[\rm (1)] The \emph{colimit} of $F$ \emph{weighted by} a contravariant presheaf $\mu\in\PA$ is an object $\colim_{\mu}F\in\bbB_0$  (necessarily of type $t\mu$) such that
    $$\bbB({\colim}_{\mu}F,-)=F_{\nat}\lda\mu.$$
    A $\CQ$-category $\bbB$ is \emph{cocomplete} if $\colim_{\mu}F$ exists for each $\CQ$-functor $F:\bbA\to\bbB$ and $\mu\in\PA$.
\item[\rm (2)] The \emph{limit} of $F$ \emph{weighted by} a covariant presheaf $\lam\in\PdA$ is an object $\lim_{\lam}F\in\bbB_0$ (necessarily of type $t\lam$) such that
    $$ \bbB(-,{\lim}_{\lam}F)=\lam\rda F^{\nat}.$$
    A $\CQ$-category $\bbB$ is \emph{complete} if $\lim_{\lam}F$ exists for each $\CQ$-functor $F:\bbA\to\bbB$ and $\lam\in\PdA$.
\end{itemize}
\end{defn}

\begin{rem}
Let $F:\bbA\to\bbB$ be a $\CQ$-functor. The definition of $\colim_{\mu}F$ for some $\mu\in\PA$ can be extended to the colimit of $F$ weighted by a $\CQ$-distributor $\phi:\bbA\oto\bbC$ as in \cite{Stubbe2005}. Explicitly, a $\CQ$-functor $G:\bbC\to\bbB$ is the colimit of $F$ weighted by $\phi:\bbA\oto\bbC$, denoted by $\colim_{\phi}F$, if
\begin{equation} \label{colimit_dist}
G_{\nat}=F_{\nat}\lda\phi.
\end{equation}
If $\bbC$ is a one-object $\CQ$-category, then $\phi$ is a contravariant presheaf and $G$ picks out an object $\colim_{\phi}F\in\bbB_0$, as formulated in Definition \ref{limit_colimit_def}.

Conversely, if $G=\colim_{\phi}F$, then for each $z\in\bbC_0$, $\phi(-,z):\bbA\oto *_{tz}$ is a contravariant presheaf, and
\begin{align*}
\bbB(Gz,-)&=G_{\nat}(z,-)&(\text{Definition of }G_{\nat})\\
&=F_{\nat}\lda\phi(-,z).&\text{(Remark \ref{distributor_notion}(6))}
\end{align*}
Thus $Gz={\colim}_{\phi(-,z)}F$. This means that the colimit of a $\CQ$-functor weighted by a $\CQ$-distributor can be obtained from Definition \ref{limit_colimit_def} via pointwise calculation.

Dually, given a $\CQ$-functor $F:\bbA\to\bbB$ and a $\CQ$-distributor $\psi:\bbC\oto\bbA$, the limit of $F$ weighted by $\psi$, denoted by $\lim_{\psi}F$, is a $\CQ$-functor $G:\bbC\to\bbB$ such that
\begin{equation} \label{limit_dist}
G^{\nat}=\psi\rda F^{\nat}.
\end{equation}
Similarly, for each $z\in\bbC_0$, $Gz={\lim}_{\psi(z,-)}F$.

As illustrated in the following diagrams, $(\colim_{\phi}F)_{\nat}$ is the largest $\CQ$-distributor that satisfies
$$({\colim}_{\phi}F)_{\nat}\circ\phi\leq F_{\nat},$$
and $(\lim_{\psi}F)^{\nat}$ is the largest $\CQ$-distributor that satisfies
$$\psi\circ({\lim}_{\psi}F)^{\nat}\leq F^{\nat}.
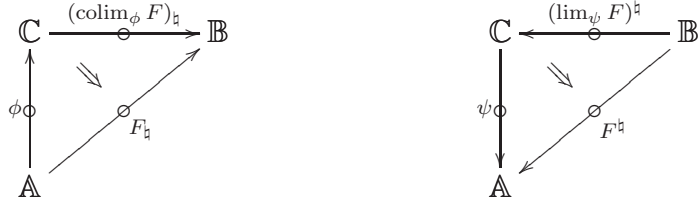
$$
$$\bfig
\ptriangle/->`<-`<-/<600,500>[\bbC`\bbB`\bbA;(\colim_{\phi}F)_{\nat}`\phi`F_{\nat}]
\place(0,250)[\circ] \place(300,500)[\circ] \place(300,250)[\circ] \place(200,400)[\twoar(1,-1)]
\ptriangle(1500,0)/<-`->`->/<600,500>[\bbC`\bbB`\bbA;(\lim_{\psi}F)^{\nat}`\psi`F^{\nat}]
\place(1500,250)[\circ] \place(1800,500)[\circ] \place(1800,250)[\circ] \place(1700,400)[\twoar(1,-1)]
\efig$$
This description applies to Definition \ref{limit_colimit_def} if $\bbA$ is replaced by a one-object $\CQ$-category.
\end{rem}

\begin{exmp} \label{ub_lb_def}
Let $\bbA$ be a $\CQ$-category and $\mu\in\PA$, the $\CQ$-distributor
$$\ub\mu=\bbA\lda\mu$$
is in $\PdA$, called the \emph{upper bounds} of $\mu$. $\ub:\PA\to\PdA$ is a $\CQ$-functor, since for all $\mu,\lam\in\PdA$,
$$\PA(\mu,\lam)=\lam\lda\mu\leq(\bbA\lda\lam)\rda(\bbA\lda\mu)=\PdA(\ub\mu,\ub\lam).$$
Note that for all $\lam\in\PdA$,
\begin{align*}
\PdA(\ub\mu,\lam)&=\lam\rda\ub\mu\\
&=\lam\rda(\bbA\lda\mu)\\
&=(\lam\rda\bbA)\lda\mu\\
&=(\sYd)_{\nat}(-,\lam)\lda\mu,
\end{align*}
thus $\ub\mu=\colim_{\mu}\sYd$.

Dually, let $\lam\in\PdA$, the $\CQ$-distributor
$$\lb\lam=\lam\rda\bbA$$
is in $\PA$, called the \emph{lower bounds} of $\lam$, and gives a $\CQ$-functor $\lb:\PdA\to\PA$. Similar calculation leads to $\lb\mu=\lim_{\mu}\sY$.
\end{exmp}

\begin{prop}[Isbell] {\rm\cite{Stubbe2005}} \label{ub_lb_adjunction}
$\ub\dv\lb:\PA\rhu\PdA$ in $\CQ$-$\Cat$.
\end{prop}

\begin{proof}
For all $\mu\in\PA$ and $\lam\in\PdA$,
\begin{align*}
\PdA(\ub\mu,\lam)&=\lam\rda\ub\mu&\text{(Equation (\ref{PdA_arrow}))}\\
&=\lam\rda(\bbA\lda\mu)\\
&=(\lam\rda\bbA)\lda\mu&\text{(Proposition \ref{arrow_calculation}(6))}\\
&=\lb\lam\lda\mu\\
&=\PA(\mu,\lb\lam).&\text{(Equation (\ref{PA_arrow}))}
\end{align*}
Hence the conclusion holds.
\end{proof}

Proposition \ref{ub_lb_adjunction} presents the Isbell adjunction
$$\bbA\lda(-)\dv(-)\rda\bbA:\PA\rhu\PdA$$
in $\CQ$-categories. We will extend the adjunction to a generalized version in Chapter \ref{Isbell_adjunction} by replacing the identity $\CQ$-distributor $\bbA$ with a general $\CQ$-distributor $\phi$.

\begin{exmp} \label{sup_def}
Let $\bbA$ be a $\CQ$-category and $\mu\in\PA$. If $\ub\mu$ is represented by some $a\in\bbA_0$, i.e.,
$$\ub\mu=\bbA(a,-),$$
we say that $a$ is a \emph{supremum} of $\mu$, and denote it by $\sup\mu$. Note that
$$\bbA(\sup\mu,-)=\bbA\lda\mu=(1_{\bbA})_{\nat}\lda\mu,$$
thus $\sup\mu=\colim_{\mu}1_{\bbA}$.

Dually, let $\lam\in\PdA$, if $\lb\lam$ is represented by some $b\in\bbA_0$, i.e.,
$$\lb\lam=\bbA(-,b),$$
we say that $b$ is an \emph{infimum} of $\lam$, and denote it by $\inf\lam$. Note that
$$\bbA(-,\inf\lam)=\lam\rda\bbA=\lam\rda (1_{\bbA})^{\nat},$$
thus $\inf\lam=\lim_{\lam}1_{\bbA}$.
\end{exmp}

\begin{rem}
It follows immediately from Definition \ref{limit_colimit_def} and Proposition \ref{isomorphic_condition} that weighted colimits and weighted limits, when they exist, are essentially unique. It should be noted that there is an abuse of notation. Strictly speaking, when we refer to essentially unique objects, such as suprema, we should use the isomorphism $a\cong\sup\mu$, since it is not necessarily unique if the related $\CQ$-category is not skeletal. But we still write $a=\sup\mu$, where we mean $a$ is one of the suprema of $\mu$. Proposition \ref{isomorphic_condition} ensures that which object we ``pick out'' to represent the supremum makes no difference in most cases. However, we must be aware that $a=\sup\mu$ and $a'=\sup\mu$ does not imply $a=a'$.
\end{rem}

\begin{prop} \label{colim_supremum}
Let $F:\bbA\to\bbB$ be a $\CQ$-functor.
\begin{itemize}
\item[\rm (1)]  For each $\mu\in\PA$,
$${\colim}_{\mu}F={\sup}_{\bbB}F^{\ra}(\mu)$$
whenever the colimit or the supremum exists.
\item[\rm (2)] For each $\lam\in\PdA$,
$${\lim}_{\lam}F={\inf}_{\bbB}F^{\nra}(\lam)$$
whenever the limit or the infimum exists.
\end{itemize}
\end{prop}

\begin{proof}
We prove (1) for example.
\begin{align*}
\bbB({\sup}_{\bbB}F^{\ra}(\mu),-)&=\bbB\lda F^{\ra}(\mu)&\text{(Example \ref{sup_def})}\\
&=\bbB\lda(\mu\circ F^{\nat})&(\text{Definition \ref{direct_inverse_image_def}(1)})\\
&=(\bbB\circ F_{\nat})\lda\mu&\text{(Proposition \ref{adjoint_arrow_calculation}(2))}\\
&=F_{\nat}\lda\mu\\
&=\bbB({\colim}_{\mu}F,-).&(\text{Definition \ref{limit_colimit_def}})
\end{align*}
Thus ${\colim}_{\mu}F={\sup}_{\bbB}F^{\ra}(\mu)$.
\end{proof}

Let $\bbA$ be a cocomplete $\CQ$-category. Since the supremum of each $\mu\in\PA$ is essentially unique, we can pick out a supremum for each $\mu\in\PA$. This gives rise to a $\CQ$-functor $\sup:\PA\to\bbA$, as shown in the following proposition. Similarly, $\inf:\PdA\to\bbA$ is a $\CQ$-functor if $\bbA$ is a complete $\CQ$-category.

\begin{prop} \label{sup_functor}
Let $\bbA$ be a $\CQ$-category.
\begin{itemize}
\item[\rm (1)] If $\bbA$ is cocomplete, then $\sup:\PA\to\bbA$ is a $\CQ$-functor and
               $$\sup\circ\sY\cong 1_{\bbA}.$$
\item[\rm (2)] If $\bbA$ is complete, then $\inf:\PdA\to\bbA$ is a $\CQ$-functor and
               $$\inf\circ\sYd\cong 1_{\bbA}.$$
\end{itemize}
\end{prop}

\begin{proof}
We prove (1) for example. For each $\mu,\lam\in\PA$,
\begin{align*}
\PA(\mu,\lam)&=\lam\lda\mu&\text{(Formula (\ref{PA_arrow}))}\\
&\leq(\bbA\lda\lam)\rda(\bbA\lda\mu)&\text{(Proposition \ref{arrow_calculation}(1)(4))}\\
&=\bbA(\sup\lam,-)\rda\bbA(\sup\mu,-)&\text{(Example \ref{sup_def})}\\
&=\bbA(\sup\mu,\sup\lam).&\text{(Remark \ref{distributor_notion}(3))}
\end{align*}
Thus $\sup$ is a $\CQ$-functor. Furthermore, for each $x\in\bbA_0$,
\begin{align*}
\bbA(\sup\circ\sY x,-)&=\bbA\lda\sY x&\text{(Example \ref{sup_def})}\\
&=\bbA\lda\bbA(-,x)\\
&=(\bbA\lda\bbA)(x,-)&\text{(Remark \ref{distributor_notion}(6))}\\
&=\bbA(x,-).
\end{align*}
Thus $\sup\circ\sY x\cong x$, and consequently $\sup\circ\sY\cong 1_{\bbA}$.
\end{proof}

\section{Tensors and cotensors}

We introduce the notions of tensors and cotensors to provide a clear characterization for cocomplete and complete $\CQ$-categories.

\begin{defn} \label{tensor_cotensor_def}
Let $\bbA$ be a $\CQ$-category.
\begin{itemize}
\item [\rm (1)] For $x\in\bbA_0$ and $f\in\CP(tx)$, the \emph{tensor} of $f$ and $x$, denoted by $f\otimes x$, is an object in $\bbA_0$ of type $t(f\otimes x)=tf$ such that
    $$\bbA(f\otimes x,-)=\bbA(x,-)\lda f.$$
    A $\CQ$-category $\bbA$ is \emph{tensored} if the tensor $f\otimes x$ exists for all choices of $x$ and $f$.
\item [\rm (2)] For $x\in\bbA_0$ and $f\in\CPd(tx)$, the \emph{cotensor} of $f$ and $x$, denoted by $f\rat x$, is an object in $\bbA_0$ of type $t(f\rat x)=tf$ such that
    $$\bbA(-,f\rat x)=f\rda\bbA(-,x).$$
    A $\CQ$-category $\bbA$ is \emph{cotensored} if the cotensor $f\rat x$ exists for all choices of $x$ and $f$.
\end{itemize}
\end{defn}

\begin{exmp}
\begin{itemize}
\item[\rm (1)] For all $x\in\bbA_0$, $1_{tx}\otimes x\cong x\cong 1_{tx}\rat x$.
\item[\rm (2)] If $\bbA$ is tensored and nonempty, then each $\bbA_X$ has a bottom element $\bot_{\bbA_X}$ and for all $x\in\bbA_0$, $\bot_{tx,X}\otimes x\cong\bot_{\bbA_X}$.
\item[\rm (3)] If $\bbA$ is cotensored and nonempty, then each $\bbA_X$ has a top element $\top_{\bbA_X}$ and for all $y\in\bbA_0$, $\bot_{X,ty}\rat y\cong\top_{\bbA_X}$.
\end{itemize}
\end{exmp}

The following proposition is a direct consequence of Definition \ref{tensor_cotensor_def}.

\begin{prop} \label{tensor_la_ra}
Let $\bbA$ be a $\CQ$-category.
\begin{itemize}
\item[\rm (1)] If $\bbA$ is tensored, then for all $x\in\bbA_0$,
               $$(-)\otimes x\dv\bbA(x,-):\CP(tx)\rhu\bbA.$$
\item[\rm (2)] If $\bbA$ is cotensored, then for all $x\in\bbA_0$,
               $$\bbA(-,x)\dv(-)\rat x:\bbA\rhu\CPd(tx).$$
\end{itemize}
\end{prop}

\begin{prop} \label{A_x_yi_distribution}
Let $\bbA$ be a $\CQ$-category and $x\in\bbA_0$, $Y\in\CQ_0$.
\begin{itemize}
\item[\rm (1)] If $\bbA$ is tensored, then for each subset $\{y_j\}\subseteq\bbA_Y$,
$$\bbA\Big(x,\bw_j y_j\Big)=\bw_j\bbA(x,y_j)$$
whenever the meet $\displaystyle\bw_j y_j$ in $\bbA_Y$ exists, where $\displaystyle\bw$ on the right hand denotes the meet in $\CQ(tx,Y)$.
\item[\rm (2)] If $\bbA$ is cotensored, then for each subset $\{y_j\}\subseteq\bbA_Y$,
$$\bbA\Big(\bv_j y_j,x\Big)=\bw_j\bbA(y_j,x)$$
whenever the join $\displaystyle\bv_j y_j$ in $\bbA_Y$ exists, where $\displaystyle\bw$ on the right hand denotes the meet in $\CQ(Y,tx)$.
\end{itemize}
\end{prop}

\begin{proof}
We prove (1) for example. Suppose the meet $\displaystyle\bw_j y_j$ in $\bbA_Y$ exists, then for all $j$, by Proposition \ref{underlying_order},
$$\bbA\Big(x,\bw_j y_j\Big)\leq\bbA(x,y_j).$$
If $f\in\CQ(tx,Y)$ satisfies $f\leq\bbA(x,y_j)$ for all $j$, then
$$1_Y\leq\bbA(x,y_j)\lda f=\bbA(f\otimes x,y_j)$$
because $\bbA$ is tensored. Thus $f\otimes x\leq y_j$ for all $j$ in $\bbA_Y$, and it follows that $f\otimes x\leq\displaystyle\bw_j y_j$. Hence
$$1_Y\leq\bbA\Big(f\otimes x,\displaystyle\bw_j y_j\Big)=\bbA\Big(x,\bw_j y_j\Big)\lda f,$$
and consequently $f\leq\bbA\Big(x,\displaystyle\bw_j y_j\Big)$.  Therefore, $\bbA\Big(x,\displaystyle\bw_j y_j\Big)=\displaystyle\bw_j\bbA(x,y_j)$.
\end{proof}

\begin{prop} \label{tensor_sup}
Let $\bbA$ be a $\CQ$-category.
\begin{itemize}
\item[\rm (1)] For $x\in\bbA_0$ and $f\in\CP(tx)$,
    $$f\otimes x=\sup(f\circ\sY x)$$
    whenever the tensor or the supremum exists.
\item[\rm (2)] For $x\in\bbA_0$ and $g\in\CPd(tx)$,
    $$g\rat x=\inf(\sYd x\circ g)$$
    whenever the cotensor or the infimum exists.
\end{itemize}
\end{prop}

\begin{proof}
We prove (1) for example. Note that $f\circ\sY x\in\PA$ because $f$ is viewed as a $\CQ$-distributor $*_{tx}\oto *_{tf}$ between one-object $\CQ$-categories. Then
\begin{align*}
\bbA(f\otimes x,-)&=\bbA(x,-)\lda f&\text{(Definition \ref{tensor_cotensor_def})}\\
&=(\bbA\lda\bbA(-,x))\lda f&\text{(Remark \ref{distributor_notion})}\\
&=\bbA\lda(f\circ\sY x)&\text{(Proposition \ref{arrow_calculation}(5))}\\
&=\bbA(\sup(f\circ\sY x),-).&\text{(Example \ref{sup_def})}
\end{align*}
Thus $f\otimes x=\sup(f\circ\sY x)$.
\end{proof}

\begin{exmp} \label{PA_tensor}
Let $\bbA$ be a $\CQ$-category.
\begin{itemize}
\item[\rm (1)] $\PA$ is a tensored and cotensored $\CQ$-category in which
$$f\otimes\mu=f\circ\mu,\quad g\rat\mu=g\rda\mu$$
for all $\mu\in\PA$ and $f\in\CP(t\mu)$, $g\in\CPd(t\mu)$.
\item[\rm (2)] $\PdA$ is a tensored and cotensored $\CQ$-category in which
$$f\otimes\lam=\lam\lda f,\quad g\rat\lam=\lam\circ g$$
for all $\lam\in\PdA$ and $f\in\CP(t\lam)$, $g\in\CPd(t\lam)$.
\end{itemize}
\end{exmp}

\begin{defn}
A $\CQ$-category $\bbA$ is \emph{order-complete} if each $\bbA_X$ admits all joins (or equivalently, all meets) in the underlying preorder.
\end{defn}

\begin{defn} \label{conical_def}
Let $\bbA$ be a $\CQ$-category and $X\in\CQ_0$.
\begin{itemize}
\item[\rm (1)] The \emph{conical colimit} of a subset $\{x_i\}\subseteq\bbA_X$ is the supremum of the join $\displaystyle\bv_i\sY x_i$ in $\CQ\text{-}\Dist(\bbA,*_X)$. $\bbA$ is \emph{conically cocomplete} if it admits all conical colimits.
\item[\rm (2)] The \emph{conical limit} of a subset $\{x_i\}\subseteq\bbA_X$ is the infimum of the join $\displaystyle\bv_i\sYd x_i$ in $\CQ\text{-}\Dist(*_X,\bbA)$. $\bbA$ is \emph{conically complete} if it admits all conical limits.
\end{itemize}
\end{defn}

It should be noted that by Remark \ref{PdA_QDist_order}, the join $\displaystyle\bv_i\sYd x_i$ in $\CQ\text{-}\Dist(*_X,\bbA)$ in the meet $\displaystyle\bw_i\sYd x_i$ in $(\PdA)_X$.

\begin{prop} {\rm\cite{Stubbe2006}} \label{join_sup}
Let $\bbA$ be a $\CQ$-category, $X\in\CQ_0$ and $\{x_i\}\subseteq\bbA_X$.
\begin{itemize}
\item[\rm (1)] If the conical colimit of $\{x_i\}$ exists, then it is also the join of $\{x_i\}$ in $\bbA_X$, i.e.,
    $$\sup\bv_i\sY x_i=\bv_i x_i.$$
    Conversely, if $\bbA$ is cotensored, then the join of $\{x_i\}$ in $\bbA_X$, when it exists, is also the conical colimit of $\{x_i\}$.
\item[\rm (2)] If the conical limit of $\{x_i\}$ exists, then it is also the meet of $\{x_i\}$ in $\bbA_X$, i.e.,
    $$\inf\bv_i\sYd x_i=\bw_i x_i.$$
    Conversely, if $\bbA$ is tensored, then the meet of $\{x_i\}$ in $\bbA_X$, when it exists, is also the conical limit of $\{x_i\}$.
\end{itemize}
\end{prop}

\begin{proof}
We prove (1) for example. If the conical colimit $\sup\displaystyle\bv_i\sY x_i$ exists, then for all $i$,
\begin{align*}
\bbA\Big(x_i,\sup\bv_i\sY x_i\Big)&=\bbA\Big(\sup\bv_i\sY x_i,-\Big)\rda\bbA(x_i,-)&\text{(Remark \ref{distributor_notion}(3))}\\
&=\Big(\bbA\lda\bv_i\sY x_i\Big)\rda\bbA(x_i,-)&\text{(Example \ref{sup_def})}\\
&\geq(\bbA\lda\bbA(-,x_i))\rda\bbA(x_i,-)&\text{(Proposition \ref{arrow_calculation}(3))}\\
&=\bbA(x_i,-)\rda\bbA(x_i,-)&\text{(Remark \ref{distributor_notion}(6))}\\
&=\bbA(x_i,x_i)&\text{(Remark \ref{distributor_notion}(3))}\\
&\geq 1_X.
\end{align*}
Thus $x_i\leq\sup\displaystyle\bv_i\sY x_i$. Suppose that $y\in\bbA_x$ and $x_i\leq y$ for all $i$, then
\begin{align*}
\bbA\Big(\sup\bv_i\sY x_i,y\Big)&=\bbA(-,y)\lda\bv_i\sY x_i&\text{(Example \ref{sup_def})}\\
&=\bw_i(\bbA(-,y)\lda\bbA(-,x_i))&\text{(Proposition \ref{arrow_calculation}(3))}\\
&=\bw_i\bbA(x_i,y)&\text{(Remark \ref{distributor_notion}(2))}\\
&\geq 1_X.
\end{align*}
Thus $\sup\displaystyle\bv_i\sY x_i\leq y$. Therefore, $\sup\displaystyle\bv_i\sY x_i=\displaystyle\bv_i x_i$.

Conversely, if $\bbA$ is cotensored and the join $\displaystyle\bv_i x_i$ in $\bbA_X$ exists, then
\begin{align*}
\bbA\Big(\bv_i x_i,-\Big)&=\bw_i\bbA(x_i,-)&\text{(Proposition \ref{A_x_yi_distribution})}\\
&=\bw_i(\bbA\lda\bbA(-,x_i))&\text{(Remark \ref{distributor_notion})}\\
&=\bbA\lda\bv_i\sY x_i.&\text{(Proposition \ref{arrow_calculation}(3))}
\end{align*}
Thus $\displaystyle\bv_i x_i=\sup\bv_i\sY x_i$.
\end{proof}

\begin{cor} \label{cotensor_order_conical_complete}
Let $\bbA$ be a $\CQ$-category.
\begin{itemize}
\item[\rm (1)] If $\bbA$ is conically cocomplete, then $\bbA$ is order-complete. Conversely, if $\bbA$ is cotensored and order-complete, then $\bbA$ is conically cocomplete.
\item[\rm (2)] If $\bbA$ is conically complete, then $\bbA$ is order-complete. Conversely, if $\bbA$ is tensored and order-complete, then $\bbA$ is conically complete.
\end{itemize}
\end{cor}

\begin{thm} {\rm\cite{Stubbe2005,Stubbe2006,Stubbe2013a}} \label{complete_cocomplete_equivalent}
For a $\CQ$-category $\bbA$, the following conditions are equivalent:
\begin{itemize}
\item[\rm (1)] $\bbA$ is cocomplete.
\item[\rm (2)] $\bbA$ is complete.
\item[\rm (3)] Each $\mu\in\PA$ has a supremum.
\item[\rm (4)] Each $\lam\in\PdA$ has an infimum.
\item[\rm (5)] $\sY$ has a left inverse (up to isomorphism) $\sup:\PA\to\bbA$ in $\CQ$-$\Cat$.
\item[\rm (6)] $\sYd$ has a left inverse (up to isomorphism) $\inf:\PdA\to\bbA$ in $\CQ$-$\Cat$.
\item[\rm (7)] $\sY$ has a left adjoint $\sup:\PA\to\bbA$ in $\CQ$-$\Cat$.
\item[\rm (8)] $\sYd$ has a right adjoint $\inf:\PdA\to\bbA$ in $\CQ$-$\Cat$.
\item[\rm (9)] $\bbA$ is tensored and conically cocomplete.
\item[\rm (10)] $\bbA$ is cotensored and conically complete.
\item[\rm (11)] $\bbA$ is tensored, cotensored and order-complete.
\end{itemize}
In this case, for each $\mu\in\PA$ and $\lam\in\PdA$,
$$\sup\mu=\bv_{a\in\bbA_0}\mu(a)\otimes a,\quad\inf\lam=\bw_{a\in\bbA_0}\lam(a)\rat a,$$
where $\displaystyle\bv$ and $\displaystyle\bw$ denote respectively the join in $\bbA_{t\mu}$ and the meet in $\bbA_{t\lam}$.
\end{thm}

\begin{proof}
(1)$\iff$(3) and (2)$\iff$(4) follows immediately from Example \ref{sup_def} and Proposition \ref{colim_supremum}.

(1)${}\Lra{}$(5) and (2)${}\Lra{}$(6): Proposition \ref{sup_functor}.

(5)${}\Lra{}$(7): Suppose that $\sup\circ\sY\cong 1_{\bbA}$. Then for each $\mu\in\PA$,
\begin{align*}
\mu&=\PA(\sY-,\mu)&(\text{Yoneda lemma})\\
&\leq\bbA(\sup\circ\sY-,\sup\mu)\\
&=\bbA(-,\sup\mu)&(\text{Proposition \ref{isomorphic_condition}})\\
&=\sY\circ\sup\mu.&\text{(Formula (\ref{Yoneda_def}))}
\end{align*}
Thus $1_{\PA}\leq\sY\circ\sup$, and consequently $\sup\dv\sY:\PA\rhu\bbA$.

(6)${}\Lra{}$(8): Similar to (5)${}\Lra{}$(7).

(7)${}\Lra{}$(3): For all $\mu\in\PA$ and $x\in\bbA_0$,
\begin{align*}
\bbA(\sup\mu,x)&=\PA(\mu,\sY x)&\text{(Proposition \ref{adjoint_condition})}\\
&=\sY x\lda\mu&\text{(Equation (\ref{PA_arrow}))}\\
&=\bbA(-,x)\lda\mu&\text{(Formula (\ref{Yoneda_def}))}\\
&=(\bbA\lda\mu)(x).&\text{(Remark \ref{distributor_notion}(7))}
\end{align*}
Thus each $\mu\in\PA$ has a supremum $\sup\mu$.

(8)${}\Lra{}$(4): Similar to (7)${}\Lra{}$(3).

(5)${}\Lra{}$(6): Define $\inf=\sup\circ\lb:\PdA\to\PA\to\bbA$, then for all $x\in\bbA_0$,
\begin{align*}
\inf\circ\sYd x&=\sup\circ\lb\circ\sYd x\\
&=\sup(\bbA(x,-)\rda\bbA)&\text{(Example \ref{ub_lb_def})}\\
&=\sup(\bbA(-,x))&\text{(Remark \ref{distributor_notion}(8))}\\
&=\sup\circ\sY x&\text{(Formula (\ref{Yoneda_def}))}\\
&\cong x.
\end{align*}

(6)${}\Lra{}$(5): $\sup=\inf\circ\ub:\PA\to\PdA\to\bbA$ is the required left inverse of $\sY$ in $\CQ$-$\Cat$ by similar calculation to (5)${}\Lra{}$(6).

%

(3)${}\Lra{}$(9) and (4)${}\Lra{}$(10) are immediate consequences of Proposition \ref{tensor_sup} and Definition \ref{conical_def}.

(9)${}\Lra{}$(3): For each $\mu\in\PA$, by Proposition \ref{join_sup} we know that the conical colimit of $\{\mu(a)\otimes a\mid a\in\bbA_0\}\subseteq\bbA_{t\mu}$ is the join of $\{\mu(a)\otimes a\mid a\in\bbA_0\}$ in $\bbA_{t\mu}$, i.e.,
$$\sup\bv\limits_{a\in\bbA_0}\sY(\mu(a)\otimes a)=\bv_{a\in\bbA_0}\mu(a)\otimes a.$$
It follows that
\begin{align*}
\bbA\Big(\bv_{a\in\bbA_0}\mu(a)\otimes a,-\Big)&=\bbA\lda\bv\limits_{a\in\bbA_0}\sY(\mu(a)\otimes a)&\text{(Example \ref{sup_def})}\\
&=\bw\limits_{a\in\bbA_0}\bbA\lda\bbA(-,\mu(a)\otimes a)&\text{(Proposition \ref{arrow_calculation}(3))}\\
&=\bw_{a\in\bbA_0}\bbA(\mu(a)\otimes a,-)&\text{(Remark \ref{distributor_notion}(6))}\\
&=\bw_{a\in\bbA_0}\bbA(a,-)\lda\mu(a)&\text{(Definition \ref{tensor_cotensor_def})}\\
&=\bbA\lda\mu.&\text{(Proposition \ref{Dist_quantaloid}(4))}
\end{align*}
Thus $\sup\displaystyle\bv\limits_{a\in\bbA_0}\sY(\mu(a)\otimes a)=\displaystyle\bv_{a\in\bbA_0}\mu(a)\otimes a=\sup\mu$.

(10)${}\Lra{}$(4): For each $\lam\in\PdA$, by similar calculation to (9)${}\Lra{}$(3) one obtains
$$\inf\bv_{a\in\bbA_0}\sYd(\lam(a)\rat a)=\bw_{a\in\bbA_0}\lam(a)\rat a=\inf\lam.$$



Now we have (1)$\iff$(2)$\iff$(3)$\iff$(4)$\iff$(5)$\iff$(6)$\iff$(7)$\iff$(8)$\iff$(9)$\iff$(10). Finally we show that (3)+(4)+(9)${}\Lra{}$(11) and (11)${}\Lra{}$(9) to finish the proof.

(3)+(4)+(9)${}\Lra{}$(11): Since (3) and (4) hold, by Proposition \ref{tensor_sup} we have that $\bbA$ is tensored and cotensored. That $\bbA$ is order-complete follows from (9) and Corollary \ref{cotensor_order_conical_complete}.

(11)${}\Lra{}$(9) is an immediate consequence of Corollary \ref{cotensor_order_conical_complete}.
\end{proof}

\begin{exmp} \label{PA_complete}
Let $\bbA$ be a $\CQ$-category.
\begin{itemize}
\item[\rm (1)] $\PA$ is a complete $\CQ$-category in which
$$\sup\Phi=\bv_{\mu\in\PA}\Phi(\mu)\circ\mu=\Phi\circ(\sY_{\bbA})_{\nat}$$
$$\bfig
\ptriangle/->`<-`<-/[\PA`*_{t\Phi}`\bbA; \Phi`(\sY_{\bbA})_{\nat}`\sup\Phi] \place(250,500)[\circ] \place(0,250)[\circ]\place(250,250)[\circ]
\efig$$
for all $\Phi\in\CP(\PA)$ and
$$\inf\Psi=\bw_{\mu\in\PA}\Psi(\mu)\rda\mu=\Psi\rda(\sY_{\bbA})_{\nat}$$
$$\bfig
\btriangle[\bbA`*_{t\Psi}`\PA; \inf\Psi`(\sY_{\bbA})_{\nat}`\Psi]
\place(250,0)[\circ]\place(0,250)[\circ] \place(250,250)[\circ] \place(125,125)[\twoar(1,1)]
\efig$$
for all $\Psi\in\CPd(\PA)$, i.e., $\inf\Psi$ is the largest $\CQ$-distributor $\mu:\bbA\oto *_{t\Psi}$ such that $\Psi\circ\mu\leq (\sY_{\bbA})_{\nat}$.
\item[\rm (2)] $\PdA$ is a complete $\CQ$-category in which
$$\sup\Phi=\bw_{\lam\in\PdA}\lam\lda\Phi(\lam)=(\sYd_{\bbA})^{\nat}\lda\Phi$$
$$\bfig
\btriangle[\PdA`*_{t\Phi}`\bbA; \Phi`(\sYd_{\bbA})^{\nat}`\sup\Phi]
\place(250,0)[\circ]\place(0,250)[\circ] \place(250,250)[\circ] \place(125,125)[\twoar(1,1)]
\efig$$
for all $\Phi\in\CP(\PdA)$ and
$$\inf\Psi=\bv_{\lam\in\PdA}\lam\circ\Psi(\lam)=(\sYd_{\bbA})^{\nat}\circ\Psi$$
$$\bfig
\ptriangle/->`<-`<-/[\PdA`\bbA`*_{t\Psi}; (\sYd_{\bbA})^{\nat}`\Psi`\inf\Psi]
\place(250,500)[\circ] \place(0,250)[\circ]\place(250,250)[\circ]
\efig$$
for all $\Psi\in\CPd(\PdA)$.
\end{itemize}
In particular, $\CP X$ and $\CPd X$ are both complete $\CQ$-categories for all $X\in\CQ_0$.
\end{exmp}

\begin{prop} \label{F_functor_condition}
Let $F:\bbA\to\bbB$ be a type-preserving function between $\CQ$-categories. If $\bbA$ and $\bbB$ are tensored, then $F$ is a $\CQ$-functor if and only if
\begin{itemize}
\item[{\rm (1)}] For all $x\in\bbA_0$ and $f\in\CP(tx)$, $f\otimes_{\bbB}Fx\leq F(f\otimes_{\bbA}x)$;
\item[{\rm (2)}] $F:\bbA_0\to\bbB_0$ is order-preserving.
\end{itemize}
Dually, if $\bbA$ and $\bbB$ are cotensored, then $F$ is a $\CQ$-functor if and only if
\begin{itemize}
\item[{\rm (1)}] For all $x\in\bbA_0$, $f\in\CPd(tx)$, $F(f\rat_{\bbA}x)\leq f\rat_{\bbB}Fx$;
\item[{\rm (2)}] $F:\bbA_0\to\bbB_0$ is order-preserving.
\end{itemize}
\end{prop}

\begin{proof}
We prove the case that $\bbA$ and $\bbB$ are tensored for example. Suppose that $F$ is a $\CQ$-functor. Let $x\in\bbA_0$ and $f\in\CP(tx)$, then
\begin{align*}
\bbB(f\otimes_{\bbB}Fx,F(f\otimes_{\bbA}x))&=\bbB(Fx,F(f\otimes_{\bbA}x))\lda f&\text{(Definition \ref{tensor_cotensor_def})}\\
&\geq\bbA(x,f\otimes_{\bbA}x)\lda f&\text{(Definition \ref{Q_functor})}\\
&=\bbA(f\otimes_{\bbA}x,f\otimes_{\bbA}x)&\text{(Definition \ref{tensor_cotensor_def})}\\
&\geq 1_{tf}.
\end{align*}
Thus $f\otimes_{\bbB}Fx\leq F(f\otimes_{\bbA}x)$.

Let $x,x'\in\bbA_0$. If $x\leq x'$, then $tx=tx'=X$ and
$$1_X\leq\bbA(x,x')\leq\bbB(Fx,Fx').$$
Thus $Fx\leq Fx'$. It follows that $F:\bbA_0\to\bbB_0$ is order-preserving.

Conversely, for all $x,x'\in\bbA_0$, by Definition \ref{tensor_cotensor_def},
$$\bbA(\bbA(x,x')\otimes_{\bbA}x,x')=\bbA(x,x')\lda\bbA(x,x')\geq 1_{tx'},$$
thus $\bbA(x,x')\otimes_{\bbA}x\leq x'$ in $\bbA_0$. Since $F:\bbA_0\to\bbB_0$ is order-preserving, it follows that
$$\bbA(x,x')\otimes_{\bbB}Fx\leq F(\bbA(x,x')\otimes_{\bbA}x)\leq Fx'.$$
By Definition \ref{tensor_cotensor_def}, this means that
$$1_{tx'}\leq\bbB(\bbA(x,x')\otimes_{\bbB}Fx,Fx')=\bbB(Fx,Fx')\lda\bbA(x,x').$$
Thus $\bbA(x,x')\leq\bbB(Fx,Fx')$. Hence $F$ is a $\CQ$-functor.
\end{proof}

\begin{prop} {\rm\cite{Stubbe2006}} \label{F_la_ra_condition}
Let $F:\bbA\to\bbB$ be a type-preserving function between $\CQ$-categories. If $\bbA$ is tensored, then $F$ is a left adjoint $\CQ$-functor in $\CQ$-$\Cat$ if and only if
\begin{itemize}
\item[{\rm (1)}] $F$ preserves tensors in the sense that $F(f\otimes_{\bbA}x)= f\otimes_{\bbB}Fx$ for all $x\in\bbA_0$ and $f\in\CP(tx)$;
\item[{\rm (2)}] $F:\bbA_0\to\bbB_0$ is a left adjoint in ${\bf 2}$-$\Cat$.
\end{itemize}
Dually, if $\bbA$ is cotensored, then $F$ is a right adjoint $\CQ$-functor in $\CQ$-$\Cat$ if and only if
\begin{itemize}
\item[{\rm (1)}] $F$ preserves cotensors in the sense that $F(f\rat_{\bbA}x)=f\rat_{\bbB}Fx$ for all $x\in\bbA_0$ and $f\in\CPd(tx)$;
\item[{\rm (2)}] $F:\bbA_0\to\bbB_0$ is a right adjoint in ${\bf 2}$-$\Cat$.
\end{itemize}
\end{prop}

\begin{proof}
We prove the case that $\bbA$ is tensored for example. Suppose that $F:\bbA\to\bbB$ is a left adjoint $\CQ$-functor with a right adjoint $G:\bbB\to\bbA$. Let $x\in\bbA_0$ and $f\in\CP(tx)$, then
\begin{align*}
\bbB(F(f\otimes_{\bbA}x),-)&=\bbA(f\otimes_{\bbA}x,G-)&\text{(Proposition \ref{adjoint_condition})}\\
&=\bbA(x,G-)\lda f&\text{(Definition \ref{tensor_cotensor_def})}\\
&=\bbB(Fx,-)\lda f.&\text{(Proposition \ref{adjoint_condition})}
\end{align*}
Thus $F(f\otimes_{\bbA}x)=f\otimes_{\bbB}Fx$.

Let $x\in\bbA_0$ and $y\in\bbB_0$, then by Proposition \ref{adjoint_condition},
\begin{align*}
Fx\leq y\ \text{in}\ \bbB_0&\iff tx=ty=X\ \text{and}\ 1_X\leq\bbB(Fx,y)\\
&\iff tx=ty=X\ \text{and}\ 1_X\leq\bbA(x,Gy)\\
&\iff x\leq Gy\ \text{in}\ \bbA_0.
\end{align*}
Thus $F:\bbA_0\to\bbB_0$ is a left adjoint of $G:\bbB_0\to\bbA_0$ in ${\bf 2}$-$\Cat$.

Conversely, Suppose that $F:\bbA_0\to\bbB_0$ has a right adjoint $G:\bbB_0\to\bbA_0$ in ${\bf 2}$-$\Cat$, then $G$ is necessarily type-preserving because for all $y\in\bbB_0$,
$$Gy\leq Gy\iff FGy\leq y$$
implies $ty=t(FGy)=t(Gy)$. In order to prove $F\dv G:\bbA\rhu\bbB$, by Proposition \ref{adjoint_condition}, it suffices to show that $\bbB(Fx,y)=\bbA(x,Gy)$ for all $x\in\bbA_0$ and $y\in\bbB_0$.

By Definition \ref{tensor_cotensor_def},
$$1_{ty}\leq\bbA(x,Gy)\lda\bbA(x,Gy)=\bbA(\bbA(x,Gy)\otimes_{\bbA}x,Gy).$$
It follows that
$$1_{ty}\leq\bbB(F(\bbA(x,Gy)\otimes_{\bbA}x),y)=\bbB(Fx,y)\lda\bbA(x,Gy),$$
where the first inequality holds because $F\dv G:\bbA_0\rhu\bbB_0$ in ${\bf 2}$-$\Cat$, and the second equality follows from $F(\bbA(x,Gy)\otimes_{\bbA}x)=\bbA(x,Gy)\otimes_{\bbB}Fx$. Thus $\bbA(x,Gy)\leq\bbB(Fx,y)$.

For the reverse inequality, since $F(\bbB(Fx,y)\otimes_{\bbA}x)=\bbB(Fx,y)\otimes_{\bbB}Fx$, we have
$$1_{ty}\leq\bbB(Fx,y)\lda\bbB(Fx,y)=\bbB(F(\bbB(Fx,y)\otimes_{\bbA}x),y).$$
It follows from $F\dv G:\bbA_0\rhu\bbB_0$ in ${\bf 2}$-$\Cat$ that
$$1_{ty}\leq\bbA(\bbB(Fx,y)\otimes_{\bbA}x,Gy)=\bbA(x,Gy)\lda\bbB(Fx,y).$$
Thus $\bbB(Fx,y)\leq\bbA(x,Gy)$, completing the proof.
\end{proof}

The following corollary indicates that left adjoint $\CQ$-functors between complete $\CQ$-categories are exactly suprema-preserving $\CQ$-functors, while right adjoint $\CQ$-functors between complete $\CQ$-categories are exactly infima-preserving $\CQ$-functors.

\begin{cor} {\rm\cite{Stubbe2006}} \label{left_adjoint_preserves_sup}
Let $F:\bbA\to\bbB$ be a $\CQ$-functor between $\CQ$-categories, with $\bbA$ complete.
\begin{itemize}
\item[\rm (1)] $F:\bbA\to\bbB$ is a left adjoint in $\CQ$-$\Cat$ if and only if $F$ preserves suprema in the sense that $F(\sup_{\bbA}\mu)=\sup_{\bbB}F^{\ra}(\mu)$ for all $\mu\in\PA$.
    $$\bfig
    \square[\PA`\PB`\bbA`\bbB;F^{\ra}`\sup_{\bbA}`\sup_{\bbB}`F]
    \efig$$
\item[\rm (2)] $F:\bbA\to\bbB$ is a right adjoint in $\CQ$-$\Cat$ if and only if $F$ preserves infima in the sense that $F(\inf_{\bbA}\mu)=\inf_{\bbB}F^{\nra}(\mu)$ for all $\mu\in\PdA$.
    $$\bfig
    \square[\PdA`\PdB`\bbA`\bbB;F^{\nra}`\inf_{\bbA}`\inf_{\bbB}`F]
    \efig$$
\end{itemize}
\end{cor}

\begin{proof}
We prove (1) for example. Suppose that $F:\bbA\to\bbB$ has a right adjoint $G:\bbB\to\bbA$ in $\CQ$-$\Cat$. Then for all $\mu\in\PA$,
\begin{align*}
\bbB(F({\sup}_{\bbA}\mu),-)&=\bbA({\sup}_{\bbA}\mu,G-)&(F\dv G:\bbA\rhu\bbB)\\
&=\bbA(-,G-)\lda\mu&(\text{Example \ref{sup_def}})\\
&=F_{\nat}\lda\mu.&(\text{Proposition \ref{adjoint_graph}})
\end{align*}
Thus $F(\sup_{\bbA}\mu)=\colim_{\mu}F=\sup_{\bbB}F^{\ra}(\mu)$.

Conversely, by Proposition \ref{F_la_ra_condition}, it suffices to show that $F$ preserves tensors and $F:\bbA_0\to\bbB_0$ is a left adjoint in ${\bf 2}$-$\Cat$. For all $x\in\bbA_0$ and $f\in\CP(tx)$,
\begin{align*}
F(f\otimes_{\bbA}x)&=F({\sup}_{\bbA}(f\circ\sY_{\bbA} x))&\text{(Proposition \ref{tensor_sup})}\\
&={\sup}_{\bbB}F^{\ra}(f\circ\sY_{\bbA} x)\\
&={\sup}_{\bbB}(f\circ\sY_{\bbA} x\circ F^{\nat})&\text{(Definition \ref{direct_inverse_image_def})}\\
&={\sup}_{\bbB}(f\circ\bbB(-,Fx))\\
&={\sup}_{\bbB}(f\circ\sY_{\bbB}\circ Fx)\\
&=f\otimes_{\bbB}Fx.&\text{(Proposition \ref{tensor_sup})}
\end{align*}
Thus $F$ preserves tensors.

Note that when $\bbA$ is complete, $F:\bbA_0\to\bbB_0$ is a left adjoint in ${\bf 2}$-$\Cat$ if and only if $F:\bbA_X\to\bbB_X$ preserves underlying joins for each $X\in\CQ_0$. Let $X\in\CQ_0$ and $\{x_i\}\subseteq\bbA_X$,
\begin{align*}
F\Big(\bv_i x_i\Big)&=F\Big({\sup}_{\bbA}\bv_i\sY_{\bbA} x_i\Big)&\text{(Proposition \ref{join_sup})}\\
&={\sup}_{\bbB}F^{\ra}\Big(\bv_i\sY_{\bbA} x_i\Big)\\
&={\sup}_{\bbB}\Big(\bv_i\sY_{\bbA} x_i\Big)\circ F^{\nat}&\text{(Definition \ref{direct_inverse_image_def})}\\
&={\sup}_{\bbB}\bv_i(\sY_{\bbA}x_i\circ F^{\nat})&\text{(Proposition \ref{Dist_quantaloid} and Formula (\ref{quantaloid_distributive}))}\\
&={\sup}_{\bbB}\bv_i\bbB(-,Fx_i)\\
&={\sup}_{\bbB}\bv_i\sY_{\bbB}(Fx_i)\\
&=\bv_i Fx_i.&\text{(Proposition \ref{join_sup})}
\end{align*}
Thus $F:\bbA_X\to\bbB_X$ preserves underlying joins, completing the proof.
\end{proof}

\section{Free cocompletion and completion}

Let $\bbA$ and $\bbB$ be $\CQ$-categories. Each $\CQ$-functor $F:\bbA\to\PdB$ corresponds to a $\CQ$-distributor $\hF:\bbA\oto\bbB$ given by
\begin{equation} \label{hF_def}
\hF(x,y)=(Fx)(y)
\end{equation}
for all $x\in\bbA_0$ and $y\in\bbB_0$, and each $\CQ$-functor $G:\bbB\to\PA$ corresponds to a $\CQ$-distributor $\hG:\bbA\oto\bbB$ given by
\begin{equation} \label{hG_def}
\hG(x,y)=(Gy)(x)
\end{equation}
for all $x\in\bbA_0$ and $y\in\bbB_0$.

Conversely, each $\CQ$-distributor $\phi:\bbA\oto\bbB$ corresponds to two $\CQ$-functors
$$\ulphi:\bbA\to\PdB\quad\text{and}\quad\olphi:\bbB\to\PA$$
given by
\begin{equation} \label{ulphi_def}
\ulphi x=\phi(x,-)\quad\text{and}\quad\olphi y=\phi(-,y)
\end{equation}
for all $x\in\bbA_0$ and $y\in\bbB_0$.

\begin{prop} {\rm\cite{Stubbe2005}} \label{distributor_functor_bijection}
Let $\bbA$ and $\bbB$ be $\CQ$-categories.
\begin{itemize}
\item[\rm (1)] The correspondences $\phi\mapsto\ulphi$ and $F\mapsto\hF$ establish an isomorphism of posets
               $$\CQ\text{-}\Dist(\bbA,\bbB)\cong\CQ\text{-}\Cat^{\co}(\bbA,\PdB),$$
               where the symbol ``$\co$'' means reversing order in the hom-sets.
\item[\rm (2)] The correspondences $\phi\mapsto\olphi$ and $F\mapsto\hF$ establish an isomorphism of posets
               $$\CQ\text{-}\Dist(\bbA,\bbB)\cong\CQ\text{-}\Cat(\bbB,\PA).$$
\end{itemize}
\end{prop}

\begin{proof}
We prove (1) for example. It follows immediately from Equation (\ref{hF_def}) and (\ref{ulphi_def}) that
$$Fx=\hF(x,-)=\underline{\hF}x$$
and
$$\phi(x,-)=\ulphi x=\ulc\ulphi\urc(x,-)$$
for all $x\in\bbA_0$. Thus the correspondences $\phi\mapsto\ulphi$ and $F\mapsto\hF$ are mutual inverse. It remains to show that both of them are order-preserving. Indeed,
\begin{align*}
&\phi\leq\psi\ \text{in}\ \CQ\text{-}\Dist(\bbA,\bbB)\\
\iff&\forall x\in\bbA_0,\ulphi x=\phi(x,-)\leq\psi(x,-)=\ulpsi x\ \text{in}\ \CQ\text{-}\Dist\\
\iff&\forall x\in\bbA_0,\ulphi x\geq\ulpsi x\ \text{in}\ (\PdB)_0&\text{(Remark \ref{PdA_QDist_order})}\\
\iff&\ulphi\leq\ulpsi\ \text{in}\ \CQ\text{-}\Cat^{\co}(\bbA,\PdB)
\end{align*}
and
\begin{align*}
&F\leq G\ \text{in}\ \CQ\text{-}\Cat^{\co}(\bbA,\PdB)\\
\iff&\forall x\in\bbA_0,Fx\geq Gx\ \text{in}\ (\PdB)_0\\
\iff&\forall x\in\bbA_0,\hF(x,-)=Fx\leq Gx=\hG(x,-)\ \text{in}\ \CQ\text{-}\Dist(\bbA,\bbB)\\
\iff&\hF\leq\hG\ \text{in}\ \CQ\text{-}\Dist(\bbA,\bbB).
\end{align*}
\end{proof}

Given a $\CQ$-distributor $\phi:\bbA\oto\bbB$, composing with $\phi$ yields two $\CQ$-functors
$$\phi^{\dag}:\PdA\to\PdB\quad\text{and}\quad\phi^*:\PB\to\PA$$
defined by
\begin{equation} \label{phidag_def}
\phi^{\dag}(\mu)=\phi\circ\mu\quad\text{and}\quad\phi^*(\lam)=\lam\circ\phi.
\end{equation}
Thus, the correspondence $\phi\mapsto\phi^{\dag}$ induces a functor
$$(-)^{\dag}:\CQ\text{-}\Dist\to\CQ\text{-}\Cat,$$
and the correspondence $\phi\mapsto\phi^*$ induces a contravariant functor
$$(-)^*:\CQ\text{-}\Dist\to(\CQ\text{-}\Cat)^{\op}.$$
Recall that Proposition \ref{graph_cograph_functor} gives rise to a functor
$$(-)_{\nat}:\CQ\text{-}\Cat\to\CQ\text{-}\Dist$$
and a contravariant functor
$$(-)^{\nat}:(\CQ\text{-}\Cat)^{\op}\to\CQ\text{-}\Dist.$$

\begin{prop} {\rm\cite{Heymans2010}} \label{composition_graph_adjunction}
\begin{itemize}
\item[\rm (1)] $(-)_{\nat}:\CQ\text{-}\Cat\to\CQ\text{-}\Dist$ is a left adjoint of $(-)^{\dag}:\CQ\text{-}\Dist\to\CQ\text{-}\Cat$.
\item[\rm (2)] $(-)^{\nat}:(\CQ\text{-}\Cat)^{\op}\to\CQ\text{-}\Dist$ is a right adjoint of $(-)^*:\CQ\text{-}\Dist\to(\CQ\text{-}\Cat)^{\op}$.
\end{itemize}
\end{prop}

\begin{proof}
We prove (1) for example. We show that the bijection of sets
$$\underline{(-)}:\CQ\text{-}\Dist(\bbA,\bbB)\cong\CQ\text{-}\Cat(\bbA,\PdB)$$
in Proposition \ref{distributor_functor_bijection} is natural in $\bbA$ and $\bbB$. This follows from the commutativity of the diagrams below for each $\CQ$-functor $F:\bbA'\to\bbA$ and $\CQ$-distributor $\psi:\bbB\oto\bbB'$.
\begin{equation} \label{composition_graph_adjunction_natural_first}
\bfig
\square<1350,500>[\CQ\text{-}\Dist(\bbA,\bbB)`\CQ\text{-}\Cat(\bbA,\PdB)`\CQ\text{-}\Dist(\bbA',\bbB)`\CQ\text{-}\Cat(\bbA',\PdB);
\underline{(-)}`\CQ\text{-}\Dist(F_{\nat},\bbB)`\CQ\text{-}\Cat(F,\PdB)`\underline{(-)}]
\efig
\end{equation}
\begin{equation} \label{composition_graph_adjunction_natural_second}
\bfig
\square<1350,500>[\CQ\text{-}\Dist(\bbA,\bbB)`\CQ\text{-}\Cat(\bbA,\PdB)`\CQ\text{-}\Dist(\bbA,\bbB')`\CQ\text{-}\Cat(\bbA,\PdB');
\underline{(-)}`\CQ\text{-}\Dist(\bbA,\psi)`\CQ\text{-}\Cat(\bbA,\psi^{\dag})`\underline{(-)}]
\efig
\end{equation}
Indeed, for each $\CQ$-distributor $\phi:\bbA\oto\bbB$ and $x'\in\bbA'_0$,
\begin{align*}
\ulphi\circ Fx'&=\phi(Fx',-)&\text{(Equation (\ref{ulphi_def}))}\\
&=(\phi\circ F_{\nat})(x',-)&\text{(Proposition \ref{graph_cograph_distributor})}\\
&=(\underline{\phi\circ F_{\nat}})x'.&\text{(Equation (\ref{ulphi_def}))}
\end{align*}
Thus Diagram (\ref{composition_graph_adjunction_natural_first}) commutes. For each $\CQ$-distributor $\phi:\bbA\oto\bbB$ and $x\in\bbA_0$,
\begin{align*}
(\psi^{\dag}\circ\ulphi)x&=\psi(\ulphi x)&\text{(Equation (\ref{phidag_def}))}\\
&=(\psi\circ\phi)(x,-)&\text{(Equation (\ref{ulphi_def}))}\\
&=(\underline{\psi\circ\phi})x.&\text{(Equation (\ref{ulphi_def}))}
\end{align*}
Thus Diagram (\ref{composition_graph_adjunction_natural_second}) commutes.
\end{proof}

Skeletal complete $\CQ$-categories and left adjoint $\CQ$-functors constitute a subcategory of $\CQ$-$\Cat$ and we denote it by $\CQ$-$\CCat$. Dually, we denote by $\CQ\text{-}\CCat^{\dag}$ the subcategory of $\CQ$-$\Cat$ consisting of skeletal complete $\CQ$-categories and right adjoint $\CQ$-functors.

Recall that Proposition \ref{Fra_composition} gives rise to a functor $\CP:\CQ\text{-}\Cat\to\CQ\text{-}\Cat$ that sends a $\CQ$-functor $F:\bbA\to\bbB$ to the left adjoint $\CQ$-functor $F^{\ra}:\PA\to\PB$ between skeletal complete $\CQ$-categories. Thus, $\CP$ can be viewed as a functor
$$\CP:\CQ\text{-}\Cat\to\CQ\text{-}\CCat.$$
Similarly, the functor $\CPd:\CQ\text{-}\Cat\to\CQ\text{-}\Cat$ that sends a $\CQ$-functor $F:\bbA\to\bbB$ to the right adjoint $\CQ$-functor $F^{\nra}:\PdA\to\PdB$ between skeletal complete $\CQ$-categories can be viewed as a functor
$$\CP^{\dag}:\CQ\text{-}\Cat\to\CQ\text{-}\CCat^{\dag}.$$

\begin{prop} \label{P_factor_through_graph}
Let $\CP$ and $\CPd$ be functors as defined above.
\begin{itemize}
\item[\rm (1)] $\CP$ factor through $(-)^{\nat}$ via $(-)^*$, i.e., the diagram
$$\bfig
\qtriangle<1000,500>[\CQ\text{-}\Cat`(\CQ\text{-}\Dist)^{\op}`\CQ\text{-}\CCat;(-)^{\nat}`\CP`(-)^*]
\efig$$
commutes, where $(-)^*$ is viewed as a functor $(\CQ\text{-}\Dist)^{\op}\to\CQ\text{-}\CCat$.
\item[\rm (2)] $\CPd$ factors through $(-)_{\nat}$ via $(-)^{\dag}$, i.e., the diagram
$$\bfig
\qtriangle<1000,500>[\CQ\text{-}\Cat`\CQ\text{-}\Dist`\CQ\text{-}\CCat^{\dag};(-)_{\nat}`\CPd`(-)^{\dag}]
\efig$$
commutes, where $(-)^{\dag}$ is viewed as a functor $\CQ\text{-}\Dist\to\CQ\text{-}\CCat^{\dag}$.
\end{itemize}
\end{prop}

\begin{proof}
By Definition \ref{direct_inverse_image_def}, it is easy to see that for each $\CQ$-functor $F:\bbA\to\bbB$,
$$F^{\ra}=(F^{\nat})^*\quad\text{and}\quad F^{\nra}=(F_{\nat})^{\dag}.$$
\end{proof}

\begin{prop} {\rm\cite{Stubbe2005}} \label{P_free_cocompletion}
\begin{itemize}
\item[\rm (1)] $\CP:\CQ\text{-}\Cat\to\CQ\text{-}\CCat$ is a left adjoint of the forgetful functor $\CQ\text{-}\CCat\to\CQ\text{-}\Cat$.
\item[\rm (2)] $\CPd:\CQ\text{-}\Cat\to\CQ\text{-}\CCat^{\dag}$ is a left adjoint of the forgetful functor $\CQ\text{-}\CCat^{\dag}\to\CQ\text{-}\Cat$.
\end{itemize}
\end{prop}

\begin{proof}
(1) By Proposition \ref{Yoneda_natural}, $\sY=\{\sY_{\bbA}\}$ is a natural transformation from the identity functor on $\CQ$-$\Cat$ to $\CP$. We show that $\sY$ is the unit of the desired adjunction.

Let $\bbA$ be a $\CQ$-category, $\bbB$ a skeletal complete $\CQ$-category and $F:\bbA\to\bbB$ a $\CQ$-functor. We claim that there is a unique left adjoint $\CQ$-functor $G:\PA\to\bbB$ such that the following diagram commutes:
\begin{equation} \label{P_factor_throgh_Yoneda}
\bfig
\qtriangle[\bbA`\PA`\bbB;\sY_{\bbA}`F`G]
\efig
\end{equation}
Define $G=\sup_{\bbB}\circ F^{\ra}:\PA\to\PB\to\bbB$. By Proposition \ref{direct_inverse_image_adjunction} and Theorem \ref{complete_cocomplete_equivalent}, $G$ is the composition of the left adjoint $\CQ$-functors $F^{\ra}$ and $\sup_{\bbB}$, thus $G$ is also a left adjoint $\CQ$-functor. For all $x\in\bbA_0$,
\begin{align*}
G\circ\sY_{\bbA}x&={\sup}_{\bbB}\circ F^{\ra}\circ\sY_{\bbA}x&(\text{Definition of }G)\\
&={\sup}_{\bbB}\circ\sY_{\bbB}\circ Fx&\text{(Proposition \ref{Yoneda_natural})}\\
&=Fx.&(\text{Proposition \ref{sup_functor} and }\bbB\text{ is skeletal})
\end{align*}
Thus Diagram (\ref{P_factor_throgh_Yoneda}) commutes. Suppose that $H:\PA\to\bbB$ is another left adjoint $\CQ$-functor making Diagram (\ref{P_factor_throgh_Yoneda}) commute. Then for each $\mu\in\PA$,
\begin{align*}
H(\mu)&=H(\mu\circ\bbA)\\
&=H\Big(\bv_{x\in\bbA_0}\mu(x)\circ\bbA(-,x)\Big)&(\text{Proposition \ref{Dist_quantaloid}})\\
&=H\Big(\bv_{x\in\bbA_0}\mu(x)\otimes_{\PA}\sY_{\bbA}x\Big)&(\text{Example \ref{PA_tensor}})\\
&=\bv_{x\in\bbA_0}\mu(x)\otimes_{\bbB}(H\circ\sY_{\bbA}x)&(\text{Proposition \ref{F_la_ra_condition}})\\
&=\bv_{x\in\bbA_0}\mu(x)\otimes_{\bbB}Fx.&(\text{Diagram (\ref{P_factor_throgh_Yoneda})})
\end{align*}
Consequently,
\begin{align*}
\bbB(H(\mu),-)&=\bbB\Big(\bv_{x\in\bbA_0}\mu(x)\otimes_{\bbB}Fx,-\Big)\\
&=\bw_{x\in\bbA_0}\bbB(\mu(x)\otimes_{\bbB} Fx,-)&\text{(Proposition \ref{A_x_yi_distribution})}\\
&=\bw_{x\in\bbA_0}\bbB(Fx,-)\lda\mu(x)&\text{(Definition \ref{tensor_cotensor_def})}\\
&=F_\nat \lda\mu&(\text{Proposition \ref{Dist_quantaloid}})\\
&=\bbB({\colim}_{\mu}F,-)&(\text{Definition \ref{limit_colimit_def}})\\
&=\bbB({\sup}_{\bbB}\circ F^{\ra}(\mu),-)&\text{(Proposition \ref{colim_supremum})}
\end{align*}
Since $\bbB$ is skeletal, it follows that $H=\sup_{\bbB}\circ F^{\ra}=G$.

(2) Similar to (1), one can prove that $\sYd$ is the unit of the desired adjunction. Explicitly, for each $\CQ$-category $\bbA$, skeletal complete $\CQ$-category $\bbB$ and $\CQ$-functor $F:\bbA\to\bbB$, there is a unique right adjoint $\CQ$-functor $G=\inf_{\bbB}\circ F^{\nra}:\PdA\to\PdB\to\bbB$ making the following diagram commutes:
$$\bfig
\qtriangle[\bbA`\PdA`\bbB;\sYd_{\bbA}`F`G]
\efig$$
\end{proof}

Proposition \ref{P_free_cocompletion} implies that $\PA$ is the free cocompletion of a $\CQ$-category $\bbA$, and $\PdA$ is the free completion of $\bbA$. This means that
\begin{itemize}
\item[\rm (1)] each $\CQ$-functor $F:\bbA\to\bbB$ into a cocomplete $\CQ$-category factors uniquely (up to isomorphism) through the Yoneda embedding $\sY_{\bbA}$ via a left adjoint $\CQ$-functor $\PA\to\bbB$;
\item[\rm (2)] each $\CQ$-functor $F:\bbA\to\bbB$ into a complete $\CQ$-category factors uniquely (up to isomorphism) through the co-Yoneda embedding $\sYd_{\bbA}$ via a right adjoint $\CQ$-functor $\PdA\to\bbB$.
\end{itemize}
$$\bfig
\qtriangle[\bbA`\PA`\bbB;\sY_{\bbA}`F`\sup_{\bbB}\circ F^{\ra}]
\qtriangle(1500,0)[\bbA`\PdA`\bbB;\sYd_{\bbA}`F`\inf_{\bbB}\circ F^{\nra}]
\efig$$

\section{Infomorphisms}

In this section, we introduce the crucial notion in this dissertation, that of infomorphisms between $\CQ$-distributors. An infomorphism between $\CQ$-distributors is what a Chu transform between Chu spaces \cite{Barr1991,Pratt1995}. The terminology ``infomorphism'' is from computer science \cite{Barwise1997,Ganter2007}.

\begin{defn} \label{infomorphism_def}
Given $\CQ$-distributors $\phi:\bbA\oto\bbB$ and $\psi:\bbA'\oto\bbB'$, an \emph{infomorphism} $(F,G):\phi\to\psi$ is a pair of $\CQ$-functors $F:\bbA\to\bbA'$ and $G:\bbB'\to\bbB$
such that $G^\nat \circ\phi=\psi\circ F_\nat$, or equivalently, $\phi(-,G-)=\psi(F-,-)$.
$$\bfig
\square[\bbA`\bbB`\bbA'`\bbB';\phi`F_\nat `G^\nat`\psi]
\place(250,0)[\circ]
\place(250,500)[\circ]
\place(0,250)[\circ]
\place(500,250)[\circ]
\efig$$
\end{defn}

An adjunction $F\dv G:\bbA\rhu\bbB$ in $\CQ$-$\Cat$ is exactly an infomorphism from the identity $\CQ$-distributor on $\bbA$ to the identity $\CQ$-distributor on $\bbB$. Thus, infomorphisms are an extension of adjoint $\CQ$-functors.

$\CQ$-distributors and infomorphisms constitute a category $\CQ$-$\Info$.

\begin{prop} \label{Y_functor_Cat_Info}
Let $F:\bbA\to\bbB$ be a $\CQ$-functor, then
$$(F,F^{\la}):((\sY_{\bbA})_{\nat}:\bbA\oto\PA)\to((\sY_{\bbB})_{\nat}:\bbB\oto\PB)$$
is an infomorphism.
\end{prop}

\begin{proof}
For all $x\in\bbA_0$ and $\lam\in\PB$,
\begin{align*}
(\sY_{\bbA})_{\nat}(x,F^{\la}(\lam))&= \PA(\sY_\bbA(x),F^{\la}(\lam))\\
&=F^{\la}(\lam)(x)&\text{(Yoneda lemma)}\\
&=\lam(Fx)&\text{(Proposition \ref{F_la_lam_Fx})}\\
&= \PB(\sY_\bbB(Fx),\lam)&\text{(Yoneda lemma)}\\
&=(\sY_{\bbB})_{\nat}(Fx,\lam).
\end{align*}
Hence the conclusion holds.
\end{proof}

The above proposition gives rise to a fully faithful functor $\BY:\CQ\text{-}\Cat\to\CQ\text{-}\Info$ that sends each $\CQ$-category $\bbA$ to the graph $(\sY_{\bbA})_{\nat}$ of the Yoneda embedding.

\begin{prop} \label{Y_U_adjunction}
$\BY:\CQ\text{-}\Cat\to\CQ\text{-}\Info$ is a left adjoint of the forgetful functor $\BU:\CQ\text{-}\Info\to\CQ\text{-}\Cat$ that sends an infomorphism
$$(F,G):(\phi:\bbA\oto\bbB)\to(\psi:\bbA'\oto\bbB')$$
to the $\CQ$-functor $F:\bbA\to\bbA'$.
\end{prop}

\begin{proof}
It is clear that $\BU\circ\BY={\bf id}_{\CQ\text{-}\Cat}$, the identity functor on $\CQ$-$\Cat$. Thus $\{1_{\bbA}\}$ is a natural transformation from ${\bf id}_{\CQ\text{-}\Cat}$ to $\BU\circ\BY$. It remains to show that for each $\CQ$-category $\bbA$, $\CQ$-distributor $\psi:\bbA'\to\bbB'$ and $\CQ$-functor $H:\bbA\to\bbA'$, there is a unique infomorphism
$$(F,G):\BY(\bbA)\to(\psi:\bbA'\oto\bbB')$$
such that the diagram
$$\bfig
\qtriangle<700,500>[\bbA`\BU\circ\BY(\bbA)`\bbA';1_{\bbA}`H`\BU(F,G)]
\efig$$
is commutative. By definition, $\BY(\bbA)$ is the graph $(\sY_{\bbA})_{\nat}:\bbA\oto\PA$ and $\BU(F,G)=F$. Thus, we only need to show that there is a unique $\CQ$-functor $G:\bbB'\to\PA$ such that
$$(H,G):((\sY_{\bbA})_{\nat}:\bbA\oto\PA)\to(\psi:\bbA'\oto\bbB')$$
is an infomorphism.

Let $G=H^{\la}\circ\olpsi:\bbB'\to\PA'\to\PA$. Then $$(H,G):((\sY_{\bbA})_{\nat}:\bbA\oto\PA) \to(\psi:\bbA'\oto\bbB')$$
is an infomorphism since
$$(\sY_{\bbA})_{\nat}(x,Gy')=(Gy')(x)=H^{\la}\circ\olpsi(y')(x)=\olpsi(y')(Hx)=\psi(Hx,y')$$
for all $x\in\bbA_0$ and $y'\in\bbB'_0$. This proves the existence of $G$.

To see the uniqueness of $G$, suppose that $G':\bbB'\to\PA$ is another $\CQ$-functor such that
$$(H,G'):((\sY_{\bbA})_{\nat}:\bbA\oto\PA)\to(\psi:\bbA'\oto\bbB')$$
is an infomorphism. Then for all $x\in\bbA_0$ and $y'\in\bbB'_0$,
\begin{align*}
(G'y')(x)&=(\sY_{\bbA})_{\nat}(x,G'y')\\
&=\psi(Hx,y')\\
&=\olpsi(y')(Hx)\\
&=H^{\la}\circ\olpsi(y')(x)\\
&=(Gy')(x),
\end{align*}
hence $G'=G$.
\end{proof}

Similar to Proposition \ref{Y_functor_Cat_Info}, one can check that sending a $\CQ$-functor $F:\bbA\to\bbB$ to the infomorphism
$$(F^{\nla},F):((\sYd_{\bbB})^{\nat}:\PdB\oto\bbB)\to((\sYd_{\bbA})^{\nat}:\PdA\oto\bbA)$$
induces a fully faithful functor $\BY^{\dag}:\CQ\text{-}\Cat\to(\CQ\text{-}\Info)^{\op}$.

\begin{prop} \label{Y_U_adjunction_contravariant}
$\BY^{\dag}:\CQ\text{-}\Cat\to(\CQ\text{-}\Info)^{\op}$ is a left adjoint of the contravariant forgetful functor $(\CQ\text{-}\Info)^{\op}\to\CQ\text{-}\Cat$ that sends each infomorphism
$$(F,G):(\phi:\bbA\oto\bbB)\to(\psi:\bbA'\oto\bbB')$$
to the $\CQ$-functor $G:\bbB'\to\bbB$.
\end{prop}

\begin{proof}
Similar to Proposition \ref{Y_U_adjunction}.
\end{proof}

\chapter{$\CQ$-closure spaces} \label{closure_space}

The notions of $\CQ$-closure operators and $\CQ$-closure systems describe the structure of monads and their algebras in $\CQ$-categories. A $\CQ$-closure space is a $\CQ$-category $\bbA$ equipped with a monad on the $\CQ$-category $\PA$ of contravariant presheaves, and has a similar structure to closure spaces in topology. We will discuss the relations between $\CQ$-closure spaces and complete $\CQ$-categories.

\section{$\CQ$-closure systems and $\CQ$-closure operators}

We first describe the monad and comonad structures on $\CQ$-categories in the terminologies of $\CQ$-closure operators and $\CQ$-interior operators.

\begin{defn} \label{closure_system}
Let $\bbA$ be a $\CQ$-category.
\begin{itemize}
\item[\rm (1)] An isomorphism-closed $\CQ$-subcategory $\bbB$ of $\bbA$ is a \emph{$\CQ$-closure system} of $\bbA$ if the inclusion $\CQ$-functor $I:\bbB\to\bbA$ is a right adjoint in $\CQ$-$\Cat$.
\item[\rm (2)] An isomorphism-closed $\CQ$-subcategory $\bbB$ of $\bbA$ is a \emph{$\CQ$-interior system} of $\bbA$ if the inclusion $\CQ$-functor $I:\bbB\to\bbA$ is a left adjoint in $\CQ$-$\Cat$.
\end{itemize}
\end{defn}

\begin{defn} \label{closure_operator}
Let $\bbA$ be a $\CQ$-category.
\begin{itemize}
\item[\rm (1)] A $\CQ$-functor $F:\bbA\to\bbA$ is a \emph{$\CQ$-closure operator} on $\bbA$ if $1_{\bbA}\leq F$ and $F^2\cong F$.
\item[\rm (2)] A $\CQ$-functor $F:\bbA\to\bbA$ is a \emph{$\CQ$-interior operator} on $\bbA$ if $F\leq 1_{\bbA}$ and $F^2\cong F$.
\end{itemize}
\end{defn}

\begin{exmp} \label{adjunction_closure_interior}
Let $F\dv G:\bbA\rhu\bbB$ be an adjunction in $\CQ$-$\Cat$. Then $G\circ F:\bbA\to\bbA$ is a $\CQ$-closure operator and $F\circ G:\bbB\to\bbB$ is a $\CQ$-interior operator.
\end{exmp}

\begin{prop} \label{closure_system_operator}
Let $\bbA$ be a $\CQ$-category, $\bbB$ an isomorphism-closed $\CQ$-subcategory of $\bbA$. The following conditions are equivalent:
\begin{itemize}
\item[\rm (1)] $\bbB$ is a $\CQ$-closure system of $\bbA$.
\item[\rm (2)] There is a $\CQ$-closure operator $F:\bbA\to\bbA$ such that $\bbB_0=\{x\in\bbA_0:Fx\cong x\}$.
\end{itemize}
\end{prop}

\begin{proof}
(1)${}\Lra{}$(2): If the inclusion $\CQ$-functor $I:\bbB\to\bbA$ has a left adjoint $G:\bbA\to\bbB$, let $F=I\circ G$, then $F:\bbA\to\bbA$ is a $\CQ$-closure operator. Since $Fx=Gx\in\bbB_0$ for all $x\in\bbA_0$ and $\bbB$ is isomorphism-closed, it is clear that $\{x\in\bbA_0:Fx\cong x\}\subseteq\bbB_0$.

Conversely, for all $x\in\bbB_0$,
$$\bbB(Fx,x)=\bbB(Gx,x)=\bbA(x,Ix)=\bbA(x,x)\geq 1_{tx},$$
and $\bbB(x,Fx)\geq 1_{tx}$ holds trivially, hence $x\cong Fx$, as required.

(2)${}\Lra{}$(1): We show that the inclusion $\CQ$-functor $I:\bbB\to\bbA$ is a right adjoint. View $F$ as a $\CQ$-functor from $\bbA$ to $\bbB$, then $1_{\bbA}\leq I\circ F$. Since $F^2\cong F$, it follows that $F\circ I\cong 1_{\bbB}$. Thus $F\dv I:\bbA\rhu\bbB$, as required.
\end{proof}

Dually, we have the following proposition for $\CQ$-interior systems and $\CQ$-interior operators.

\begin{prop} \label{interior_system_operator}
Let $\bbA$ be a $\CQ$-category, $\bbB$ an isomorphism-closed $\CQ$-subcategory of $\bbA$. The following conditions are equivalent:
\begin{itemize}
\item[\rm (1)] $\bbB$ is a $\CQ$-interior system of $\bbA$.
\item[\rm (2)] There is a $\CQ$-interior operator $F:\bbA\to\bbA$ such that $\bbB_0=\{x\in\bbA_0:Fx\cong x\}$.
\end{itemize}
\end{prop}

\begin{rem}
For a $\CQ$-category $\bbA$, a $\CQ$-closure operator $F:\bbA\to\bbA$ is exactly a monad \cite{MacLane1998} on $\bbA$. Proposition \ref{closure_system_operator} states that a $\CQ$-closure system of $\bbA$ is exactly the category of algebras for a monad on $\bbA$.

Dually, a $\CQ$-interior operator $F:\bbA\to\bbA$ is exactly a comonad on $\bbA$. Proposition \ref{interior_system_operator} states that a $\CQ$-interior system of $\bbA$ is exactly the category of coalgebras for a comonad on $\bbA$.

The terminologies ``$\CQ$-closure operator'' and ``$\CQ$-interior operator'' come from their similarity to closure operators and interior operators in topology.
\end{rem}

\begin{prop} \label{closure_system_complete}
Each $\CQ$-closure system or $\CQ$-interior system of a complete $\CQ$-category is itself a complete $\CQ$-category.
\end{prop}

\begin{proof}
Let $\bbB$ be a $\CQ$-closure system of a complete $\CQ$-category $\bbA$. By Proposition \ref{closure_system_operator}, there is a $\CQ$-closure operator $F:\bbA\to\bbA$ such that $\bbB_0=\{x\in\bbA_0:Fx\cong x\}$. View $F$ as a $\CQ$-functor from $\bbA$ to $\bbB$, then $F\circ I\cong 1_{\bbB}$, where $I:\bbB\to\bbA$ is the inclusion $\CQ$-functor. Thus
\begin{align*}
F({\sup}_{\bbA}I^{\ra}(\mu))&={\sup}_{\bbB}F^{\ra}\circ I^{\ra}(\mu)&(\text{Corollary \ref{left_adjoint_preserves_sup}})\\
&={\sup}_{\bbB}(F\circ I)^{\ra}(\mu)&(\text{Proposition \ref{Fra_composition}})\\
&={\sup}_{\bbB}(1_{\bbB})^{\ra}(\mu)\\
&={\sup}_{\bbB}\mu
\end{align*}
for all $\mu\in\PB$. Therefore, it follows from Proposition \ref{complete_cocomplete_equivalent} that $F(\bbA)$ is a complete $\CQ$-category.

Similarly one can prove that each $\CQ$-interior system of a complete $\CQ$-category is itself a complete $\CQ$-category.
\end{proof}

\begin{prop} \label{closure_system_infima_closed}
Let $\bbA$ be a complete $\CQ$-category, $\bbB$ an isomorphism-closed $\CQ$-subcategory of $\bbA$, and $I:\bbB\to\bbA$ the inclusion $\CQ$-functor.
\begin{itemize}
\item[\rm (1)] $\bbB$ is a $\CQ$-closure system of $\bbA$ if and only if $\bbB$ is closed with respect to infima in $\bbA$ in the sense that ${\inf}_{\bbA}I^{\nra}(\lam)\in\bbB_0$ for all $\lam\in\PdB$.
\item[\rm (2)] $\bbB$ is a $\CQ$-interior system of $\bbA$ if and only if $\bbB$ is closed with respect to suprema in $\bbA$ in the sense that ${\sup}_{\bbA}I^{\ra}(\mu)\in\bbB_0$ for all $\mu\in\PB$.
\end{itemize}
\end{prop}

\begin{proof}
We prove (1) for example. Note that a $\CQ$-closure system is a complete $\CQ$-category by Proposition \ref{closure_system_complete}, thus the necessity follows immediately by applying Corollary \ref{left_adjoint_preserves_sup} to the inclusion $\CQ$-functor $I:\bbB\to\bbA$.

For the sufficiency, note that for all $\lam\in\PdB$ and $y\in\bbB_0$,
\begin{align*}
\bbB(y,{\inf}_{\bbA}I^{\nra}(\lam))&=\bbA(y,{\inf}_{\bbA}I^{\nra}(\lam))\\
&=I^{\nra}(\lam)\rda\bbA(y,-)&(\text{Example \ref{sup_def}})\\
&=(I_{\nat}\circ\lam)\rda\bbA(y,-)&(\text{Definition \ref{direct_inverse_image_def}})\\
&=\lam\rda(I^{\nat}\circ\bbA(y,-))&(\text{Proposition \ref{adjoint_arrow_calculation}(2)})\\
&=\lam\rda I^{\nat}(y,-)\\
&=\lam\rda\bbA(y,I-)\\
&=\lam\rda\bbB(y,-).
\end{align*}
Thus ${\inf}_{\bbA}I^{\nra}(\lam)$ is the infimum of $\lam$ in $\bbB$, i.e., ${\inf}_{\bbA}I^{\nra}(\lam)=\inf_{\bbB}\lam$. Therefore, $\bbB$ is a complete $\CQ$-category and $I:\bbB\to\bbA$ preserves infima. Then it follows from Corollary \ref{left_adjoint_preserves_sup} that $I$ is a right adjoint in $\CQ$-$\Cat$, and consequently $\bbB$ is a $\CQ$-closure system of $\bbA$.
\end{proof}

\begin{prop} \label{closure_system_cotensor_meet}
Let $\bbA$ be a complete $\CQ$-category with tensor $\otimes$ and cotensor ${}\rat{}$, $\bbB$ an isomorphism-closed $\CQ$-subcategory of $\bbA$. Then $\bbB$ is a $\CQ$-closure system of $\bbA$ if and only if
\begin{itemize}
\item[\rm (1)] for every subset $\{x_i\}\subseteq\bbB_0$ of the same type $X$, the meet $\displaystyle\bw\limits_i x_i$ in $\bbA_X$ belongs to $\bbB_0$;
\item[\rm (2)] for each $x\in\bbB_0$ and $f\in\CPd(tx)$, the cotensor $f\rat x$ in $\bbA$ belongs to $\bbB_0$.
\end{itemize}
Dually, $\bbB$ is a $\CQ$-interior system of $\bbA$ if and only if
\begin{itemize}
\item[\rm (1)] for every subset $\{x_i\}\subseteq\bbB_0$ of the same type $X$, the join $\displaystyle\bv\limits_i x_i$ in $\bbA_X$ belongs to $\bbB_0$;
\item[\rm (2)] for each $x\in\bbB_0$ and $f\in\CP(tx)$, the tensor $f\otimes x$ in $\bbA$ belongs to $\bbB_0$.
\end{itemize}
\end{prop}

\begin{proof}
We prove the case of $\CQ$-closure system for example. Let $I:\bbB\to\bbA$ be the inclusion $\CQ$-functor.

{\bf Necessity.} By Proposition \ref{closure_system_complete}, $\bbB$ is itself a complete $\CQ$-category with cotensor $\rat_{\bbB}$. Since $I$ is a right adjoint $\CQ$-functor between complete $\CQ$-categories, it follows from Proposition \ref{F_la_ra_condition} that $I$ preserves cotensors. Thus
$$f\rat x=f\rat Ix=I(f\rat_{\bbB} x)=f\rat_{\bbB} x\in\bbB_0$$
for each $x\in\bbB_0$ and $f\in\CPd(tx)$.

That the meet $\displaystyle\bw\limits_i x_i$ in $\bbA_X$ belongs to $\bbB_0$ for each subset $\{x_i\}\subseteq\bbB_X$ can be obtained similarly.

{\bf Sufficiency.} For each $x\in\bbB_0$ and $f\in\CPd(tx)$, since the cotensor $f\rat x$ in $\bbA$ belongs to $\bbB_0$, it follows that for each $y\in\bbB_0$,
$$\bbB(y,f\rat x)=\bbA(y,f\rat x)=f\rda\bbA(y,x)=f\rda\bbB(y,x).$$
This means that $f\rat x$ is the cotensor of $f$ and $x$ in $\bbB$, i.e., $f\rat x=f\rat_{\bbB}x$. Hence, $\bbB$ is a cotensored $\CQ$-category and it is clear that $I$ preserves cotensors in $\bbB$.

Similarly one can prove that if the meet $\displaystyle\bw\limits_i x_i$ of a subset $\{x_i\}\subseteq\bbB_X$ in $\bbA_X$ belongs to $\bbB_0$, then it is also the meet of $\{x_i\}$ in $\bbB_X$. Thus $\bbB$ is order-complete and $I$ preserves underlying meets in each $\bbB_X$. This means that $I:\bbB_0\to\bbA_0$ is a right adjoint in {\bf 2}-$\Cat$.

Therefore, $I:\bbB\to\bbA$ is a right adjoint in $\CQ$-$\Cat$ by Proposition \ref{F_la_ra_condition}, and the conclusion thus follows.
\end{proof}

An immediate consequence of Proposition \ref{closure_system_cotensor_meet} is that the infimum in a $\CQ$-closure system $\bbB$ of a complete $\CQ$-category $\bbA$ can be calculated as
\begin{equation} \label{closure_system_infimum}
{\inf}_{\bbB}\lam=\bw_{b\in\bbB_0}(\lam(b)\rat b)
\end{equation}
for $\lam\in\PdB$, where the cotensors and meets are calculated in $\bbA$.

Dually, the supremum in a $\CQ$-interior system $\bbB$ of a complete $\CQ$-category $\bbA$ can be calculated as
\begin{equation} \label{interior_system_supremum}
{\sup}_{\bbB}\mu=\bv_{b\in\bbB_0}(\mu(b)\otimes b)
\end{equation}
$\mu\in\PB$, where the tensors and joins are calculated in $\bbA$.

\section{$\CQ$-closure spaces}

In this section, we pay attention to monads on the $\CQ$-category of contravariant presheaves, and introduce the notion of $\CQ$-closure spaces.

\begin{defn}
A \emph{$\CQ$-closure space} is a pair $(\bbA,C)$ that consists of a $\CQ$-category $\bbA$ and a $\CQ$-closure operator $C:\PA\to\PA$. A \emph{continuous $\CQ$-functor} $F:(\bbA,C)\to(\bbB,D)$ between $\CQ$-closure spaces is a $\CQ$-functor $F:\bbA\to\bbB$ such that
$$F^{\ra}\circ C\leq D\circ F^{\ra}.$$
\end{defn}

The category of $\CQ$-closure spaces and continuous $\CQ$-functors is denoted by $\CQ$-$\Cls$.

\begin{rem}
If $C$ and $D$ are viewed as monads on $\PA$ and $\PB$ respectively, then a $\CQ$-functor $F:\bbA\to\bbB$ is continuous between $\CQ$-closure spaces $(\bbA,C)$ and $(\bbB,D)$ if and only if $F^{\ra}:\PA\to\PB$ is a lax map of monads from $C$ to $D$ in the sense of \cite{Leinster2004}.
\end{rem}

Note that for a $\CQ$-closure space $(\bbA,C)$, the $\CQ$-closure operator $C$ is idempotent since $\PA$ is skeletal. Let $C(\PA)$ denote the $\CQ$-subcategory of $\PA$ consisting of the fixed points of $C$. Since $\PA$ is a complete $\CQ$-category,  $C(\PA)$ is also a complete $\CQ$-category. A contravariant presheaf $\bbA\oto*_X$ is said to be \emph{closed} in the $\CQ$-closure space $(\bbA,C)$ if it belongs to $C(\PA)$.

The following proposition states that continuous $\CQ$-functors behave in a manner similar to the continuous maps between topological spaces: the inverse image of a closed contravariant presheaf is closed.

\begin{prop} \label{F_continuous_condition}
A $\CQ$-functor $F:\bbA\to\bbB$ is continuous between $\CQ$-closure spaces $(\bbA,C)$ and $(\bbB,D)$ if and only if $F^{\la}(\lam)\in C(\PA)$ whenever $\lam\in D(\PB)$.
\end{prop}

\begin{proof}
It suffices to show that $F^{\ra}\circ C\leq D\circ F^{\ra}$ if and only if $C\circ F^{\la}\circ D\leq F^{\la}\circ D$.

Suppose $F^{\ra}\circ C\leq D\circ F^{\ra}$, then
$$F^{\ra}\circ C\circ F^{\la}\circ D\leq D\circ F^{\ra}\circ F^{\la}\circ D\leq D\circ D=D,$$
and consequently $C\circ F^{\la}\circ D\leq F^{\la}\circ D$.

Conversely, suppose $C\circ F^{\la}\circ D\leq F^{\la}\circ D$, then
$$C\leq C\circ F^{\la}\circ F^{\ra}\leq C\circ F^{\la}\circ D\circ F^{\ra}\leq F^{\la}\circ D\circ F^{\ra},$$
and consequently $F^{\ra}\circ C\leq D\circ F^{\ra}$.
\end{proof}

Given a $\CQ$-category $\bbA$, there are naturally two $\CQ$-closure spaces with $\bbA$ being the underlying $\CQ$-category. One is the \emph{discrete} $\CQ$-closure space $(\bbA,1_{\PA})$, in which every contravariant presheaf $\mu\in\PA$ is closed. The other one is the \emph{trivial} $\CQ$-closure space $(\bbA,T_{\bbA})$ given by
$$T_{\bbA}(\mu)(x)=\top_{tx,t\mu}$$
for all $\mu\in\PA$ and $x\in\bbA_0$, in which $\mu\in\PA$ is closed if and only if $\mu$ is the largest $\CQ$-distributor in $\CQ\text{-}\Dist(\bbA,*_{t\mu})$.

It is easy to see that each $\CQ$-functor $F:\bbA\to\bbB$ is continuous between the $\CQ$-closure spaces $(\bbA,1_{\PA})$ and $(\bbB,1_{\PB})$, and also continuous between the $\CQ$-closure spaces $(\bbA,T_{\bbA})$ and $(\bbB,T_{\bbB})$. Thus, we obtain a functor
$$\BD:\CQ\text{-}\Cat\to\CQ\text{-}\Cls$$
that sends a $\CQ$-category $\bbA$ to the corresponding discrete $\CQ$-closure space $(\bbA,1_{\PA})$, and a functor
$$\BT:\CQ\text{-}\Cat\to\CQ\text{-}\Cls$$
that sends a $\CQ$-category $\bbA$ to the corresponding trivial $\CQ$-closure space $(\bbA,T_{\bbA})$.

It is well known that the forgetful functor from the category of topological spaces to the category of (small) sets has a left adjoint creating the discrete topology on a set, and a right adjoint creating the trivial topology on a set. We have the following analogue conclusion for $\CQ$-closure spaces.

\begin{prop} \label{free_closure_space}
The forgetful functor $\CQ\text{-}\Cls\to\CQ\text{-}\Cat$ has a left adjoint $\BD:\CQ\text{-}\Cat\to\CQ\text{-}\Cls$, and a right adjoint $\BT:\CQ\text{-}\Cat\to\CQ\text{-}\Cls$.
\end{prop}

\begin{proof}
It is not difficult to verify that the correspondence $F\mapsto F$ induces a bijection
$$\CQ\text{-}\Cat(\bbA,\bbB)\cong\CQ\text{-}\Cls((\bbA,1_{\PA}),(\bbB,D))$$
natural in each $\CQ$-category $\bbA$ and $\CQ$-closure space $(\bbB,D)$. Also, the correspondence $F\mapsto F$ induces a bijection
$$\CQ\text{-}\Cat(\bbA,\bbB)\cong\CQ\text{-}\Cls((\bbA,C),(\bbB,T_{\bbB}))$$
natural in each $\CQ$-closure space $(\bbA,C)$ and $\CQ$-category $\bbB$. The conclusion thus follows.
\end{proof}

\section{Relationship with complete $\CQ$-categories}

It follows from Proposition \ref{F_continuous_condition} that a continuous $\CQ$-functor $F:(\bbA,C)\to(\bbB,D)$ between $\CQ$-closure spaces induces a pair of $\CQ$-functors
$$F^{\triangleright}=D\circ F^{\ra}:C(\PA)\to D(\PB)\quad\text{and}\quad F^{\triangleleft}=F^{\la}:D(\PB)\to C(\PA).$$

\begin{prop}
If $F:(\bbA,C)\to(\bbB,D)$ is a continuous $\CQ$-functor between $\CQ$-closure spaces, then $F^{\triangleright}\dv F^{\triangleleft}:C(\PA)\rhu D(\PB)$.
\end{prop}

\begin{proof}
It is sufficient to check that
$$\PB(D\circ F^{\ra}(\mu),\lam)=\PB(F^{\ra}(\mu),\lam)$$
for all $\mu\in C(\PA)$ and $\lam\in D(\PB)$ since it holds that $\PA(\mu,F^\la(\lambda))=\PB(F^\ra(\mu),\lambda)$. Indeed, since $D$ is  a $\CQ$-closure operator,
\begin{align*}
\PB(F^{\ra}(\mu),\lam)&\leq\PB(D\circ F^{\ra}(\mu),D(\lam))\\
&=\PB(D\circ F^{\ra}(\mu),\lam)\\
&=\lam\lda(D\circ F^{\ra}(\mu))\\
&\leq\lam\lda F^{\ra}(\mu)\\
&=\PB(F^{\ra}(\mu),\lam),
\end{align*}
hence $\PB(D\circ F^{\ra}(\mu),\lam)=\PB(F^{\ra}(\mu),\lam)$.
\end{proof}

The above proposition gives rise to a functor
$$\CT:\CQ\text{-}\Cls\to\CQ\text{-}\CCat$$
that maps each continuous $\CQ$-functor
$$F:(\bbA,C)\to(\bbB,D)$$
to the left adjoint $\CQ$-functor
$$F^{\triangleright}:C(\PA)\to D(\PB)$$
between skeletal complete $\CQ$-categories.

For each complete $\CQ$-category $\bbA$, it follows from Theorem \ref{complete_cocomplete_equivalent} and Example \ref{adjunction_closure_interior} that $C_{\bbA}=\sY\circ\sup:\PA\to\PA$ is a $\CQ$-closure operator, hence $(\bbA,C_{\bbA})$ is a $\CQ$-closure space.

\begin{prop}
If $F:\bbA\to\bbB$ is a left adjoint $\CQ$-functor between complete $\CQ$-categories, then $F:(\bbA,C_{\bbA})\to(\bbB,C_{\bbB})$ is a continuous $\CQ$-functor.
\end{prop}

\begin{proof}
For all $\mu\in\PA$,
\begin{align*}
F^{\ra}\circ C_{\bbA}(\mu)&=C_{\bbA}(\mu)\circ F^{\nat}\\
&=\bbA(-,{\sup}_{\bbA}\mu)\circ F^{\nat}\\
&\leq\bbB(F-,F({\sup}_{\bbA}\mu))\circ F^{\nat}\\
&=F_{\nat}(-,F({\sup}_{\bbA}\mu))\circ F^{\nat}\\
&\leq\bbB(-,F({\sup}_{\bbA}\mu))&(F_{\nat}\dv F^{\nat}:\bbA\rhu\bbB\ \text{in}\ \CQ\text{-}\Dist)\\
&=\bbB(-,{\sup}_{\bbB}F^{\ra}(\mu))&(\text{Corollary \ref{left_adjoint_preserves_sup}})\\
&=C_{\bbB}\circ F^{\ra}(\mu).
\end{align*}
Hence $F:(\bbA,C_{\bbA})\to(\bbB,C_{\bbB})$ is continuous.
\end{proof}

The above proposition gives a functor
$$\CD:\CQ\text{-}\CCat\to\CQ\text{-}\Cls.$$

\begin{prop} \label{complete category representation}
The functor $\CT\circ\CD$ is naturally isomorphic to the identity functor on $\CQ\text{-}\CCat$. In particular, each skeletal complete $\CQ$-category $\bbA$ is isomorphic to $\CT\circ\CD(\bbA)$.
\end{prop}

\begin{proof}
For each $x\in\bbA_0$, since
\begin{equation} \label{CY=Y}
\sY_{\bbA}x=\sY_{\bbA}\circ{\sup}_{\bbA}\circ\sY_{\bbA}x=C_{\bbA}\circ\sY_{\bbA}x,
\end{equation}
it follows that
$$\CT\circ\CD(\bbA)=C_{\bbA}(\PA)=\{\sY_{\bbA} x\mid x\in\bbA_0\}.$$
Thus we get that $\sY_{\bbA}x$ is closed in the $\CQ$-closure space $(\bbA,C_{\bbA})$. By Yoneda lemma, the correspondence $x\mapsto\sY_{\bbA}x$ induces a fully faithful $\CQ$-functor $\sY_{\bbA}:\bbA\to C_{\bbA}(\PA)$. It is clear that $\sY_{\bbA}$ is surjective, hence an isomorphism of skeletal $\CQ$-categories.

To see the naturality of $\{\sY_{\bbA}\}$, for each left $\CQ$-functor $F:\bbA\to\bbB$ between skeletal complete $\CQ$-categories, we prove the commutativity of the following diagram:
$$\bfig
\square<700,500>[\bbA`C_{\bbA}(\PA)`\bbB`C_{\bbB}(\PB);\sY_{\bbA}`F`\CT\circ\CD(F)=C_{\bbB}\circ F^{\ra}`\sY_{\bbB}]
\efig$$
Indeed, by Proposition \ref{Yoneda_natural} and Equation (\ref{CY=Y}), it follows immediately that
$$C_{\bbB}\circ F^{\ra}\circ\sY_{\bbA}=C_{\bbB}\circ\sY_{\bbB}\circ F=\sY_{\bbB}\circ F.$$
\end{proof}

By Proposition \ref{complete category representation}, if we identify a skeletal complete $\CQ$-category $\bbA$ with the $\CQ$-subcategory $C_{\bbA}(\PA)$ of $\PA$, then the functor $\CT:\CQ\text{-}\Cls\to\CQ\text{-}\CCat$ can be viewed as a left inverse of $\CD:\CQ\text{-}\CCat\to\CQ\text{-}\Cls$.

\begin{thm} \label{T_D_adjunction}
$\CT:\CQ\text{-}\Cls\to\CQ\text{-}\CCat$ is a left inverse (up to isomorphism) and left adjoint of $\CD:\CQ\text{-}\CCat\to \CQ\text{-}\Cls$.
\end{thm}

\begin{proof}
It remains to show that $\CT$ is a left adjoint of $\CD$ . Given a $\CQ$-closure space $(\bbA,C)$, denote $C(\PA)$ by $\bbX$, then $\CD\circ\CT(\bbA,C)=(\bbX,C_{\bbX})$. Let $\eta_{(\bbA,C)}=C\circ\sY_{\bbA}:\bbA\to\bbX$. We show that $\eta=\{\eta_{(\bbA,C)}\}$ is a natural transformation from the identity functor  to $\CD\circ\CT$ and it is the unit of the desired adjunction.

{\bf Step 1}. $\eta_{(\bbA,C)}:(\bbA,C)\to (\bbX,C_{\bbX})$ is a continuous $\CQ$-functor, i.e. $\eta_{(\bbA,C)}^{\ra}\circ C\leq C_{\bbX}\circ\eta_{(\bbA,C)}^{\ra}$.

Firstly, we show that $C(\mu)=\sup_{\bbX}\circ\eta_{(\bbA,C)}^{\ra}(\mu)$ for all $\mu\in\PA$. Consider the diagram:
$$\bfig
\Square[\CP(\PA)`\PA`\PX`\bbX;\sup_{\PA}`C^\ra`C`\sup_{\bbX}]
\morphism(-600,500)<600,0>[\PA`\CP(\PA);\sY_{\bbA}^{\ra}]
\morphism(-600,500)|l|<600,-500>[\PA`\CP\bbX;\eta_{(\bbA,C)}^\ra]
\efig$$
The commutativity of the left triangle follows from $\eta_{(\bbA,C)}=C\circ \sY_{\bbA}$. Since $C:\PA\to\bbX$ is a left adjoint in $\CQ$-$\Cat$ (obtained in the proof of Proposition \ref{closure_system_operator}), it preserves suprema (Corollary \ref{left_adjoint_preserves_sup}), thus the right square commutes. The whole diagram is then commutative. For each $\mu\in\PA$, we have that
\begin{equation} \label{mu_sup_ymu}
\mu=\mu\circ\bbA=\mu\circ\sY_{\bbA}^{\nat}\circ(\sY_{\bbA})_{\nat}=\sY_{\bbA}^\ra(\mu)\circ(\sY_{\bbA})_{\nat}={\sup}_{\PA}\circ\sY_{\bbA}^\ra(\mu),
\end{equation}
where the second equality comes from the fact that the Yoneda embedding $\sY_{\bbA}$ is fully faithful and Proposition \ref{fully_faithful_graph_cograph}(1), while the last equality comes from Example \ref{PA_complete}. Consequently,
$$C(\mu)=C\circ{\sup}_{\PA}\circ\sY_{\bbA}^\ra(\mu)={\sup}_{\bbX}\circ\eta_{(\bbA,C)}^{\ra}(\mu)$$
for all $\mu\in\PA$.

Secondly, we show that $\eta_{(\bbA,C)}^{\ra}(\mu)\leq\sY_{\bbX}(\mu)=\bbX(-,\mu)$ for each $\mu\in\bbX$. Indeed,
\begin{align*}
\eta_{(\bbA,C)}^{\ra}(\mu)&=\mu\circ\eta_{(\bbA,C)}^{\nat}\\
&=\mu\circ(C\circ\sY_{\bbA})^{\nat}\\
&=\PA(\sY_\bbA-,\mu)\circ\sY_{\bbA}^{\nat}\circ C^{\nat}&(\text{Yoneda lemma})\\
&=(\sY_{\bbA})_{\nat}(-,\mu)\circ\sY_{\bbA}^{\nat}\circ C^{\nat}& \\
&\leq\PA(-,\mu)\circ C^{\nat}&((\sY_{\bbA})_{\nat}\dv\sY_{\bbA}^{\nat}:\bbA\rhu\PA\ \text{in}\ \CQ\text{-}\Dist)\\
&\leq\bbX(C-,\mu)\circ C^{\nat}&(C\ \text{is a}\ \CQ\text{-functor and}\ C(\mu)=\mu)\\
&=C_{\nat}(-,\mu)\circ C^{\nat}\\
&\leq\bbX(-,\mu).&(C_{\nat}\dv C^{\nat}:\PA\rhu\bbX\ \text{in}\ \CQ\text{-}\Dist)
\end{align*}

Therefore, for all $\mu\in\PA$,
$$\eta_{(\bbA,C)}^{\ra}\circ C(\mu)\leq\sY_{\bbX}\circ{\sup}_{\bbX}\circ\eta_{(\bbA,C)}^{\ra}(\mu)=C_{\bbX}\circ\eta_{(\bbA,C)}^{\ra}(\mu),$$
as desired.

{\bf Step 2}. $\eta=\{\eta_{(\bbA,C)}\}$ is a natural transformation. Let $F:(\bbA,C)\to(\bbB,D)$ be a continuous $\CQ$-functor, we must show that
$$D\circ\sY_{\bbB}\circ F=\eta_{(\bbB,D)}\circ F=\CD\circ\CT\circ F\circ\eta_{(\bbA,C)}=D\circ F^{\ra}\circ C\circ\sY_{\bbA}.$$

Firstly, since $C$ is a $\CQ$-closure operator, by Proposition \ref{Yoneda_natural},
$$\sY_{\bbB}\circ F=F^{\ra}\circ\sY_{\bbA}\leq F^{\ra}\circ C\circ\sY_{\bbA},$$
and consequently $D\circ\sY_{\bbB}\circ F\leq D\circ F^{\ra}\circ C\circ\sY_{\bbA}$.

Secondly, the continuity of $F$ leads to
$$F^{\ra}\circ C\circ\sY_{\bbA}\leq D\circ F^{\ra}\circ\sY_{\bbA}=D\circ\sY_{\bbB}\circ F,$$
hence $D\circ F^{\ra}\circ C\circ\sY_{\bbA}\leq D\circ\sY_{\bbB}\circ F$.

{\bf Step 3}. $\eta_{(\bbA,C)}:(\bbA,C)\to (\bbX,C_{\bbX})$ is universal in the sense that for any skeletal complete $\CQ$-category $\bbB$ and continuous $\CQ$-functor $F:(\bbA,C)\to (\bbB,C_{\bbB})$, there exists a unique left adjoint $\CQ$-functor $\overline{F}:\bbX\to\bbB$ that makes the following diagram commute:
\begin{equation} \label{eta_universal}
\bfig
\qtriangle<600,500>[(\bbA,C)`(\bbX,C_{\bbX})`(\bbB,C_{\bbB});\eta_{(\bbA,C)}`F`\overline{F}]
\efig
\end{equation}

{\bf Existence.} Let $\overline{F}=\sup_{\bbB}\circ F^\ra:\bbX\to\bbB$ be the following composition of $\CQ$-functors
$$\bbX\hookrightarrow\PA\to^{F^\ra}\PB\to^{\sup_{\bbB}}\bbB. $$

First,  $\overline{F}:\bbX\to\bbB$ is a left adjoint in $\CQ$-$\Cat$. Indeed, $\overline{F}$ has a right adjoint $G:\bbB\to\bbX$ given by $G=F^{\triangleleft}\circ\sY_{\bbB}$. $G$ is well-defined since $\sY_{\bbB}b$ is a closed  in $(\bbB,C_{\bbB})$ for each $b\in\bbB_0$. For all $\mu\in\bbX_0$ and $y\in\bbB_0$, it holds that
\begin{align*}
\bbB(\overline{F}(\mu),y)&=\bbB(-,y)\lda F^{\ra}(\mu)\\
&=\bbB(-,y)\lda(\mu\circ F^{\nat})\\
&=(\bbB(-,y)\circ F_\nat)\lda\mu&(\text{Proposition \ref{adjoint_arrow_calculation}(2)})\\
&=F_\nat (-,y)\lda\mu&(\text{Remark \ref{distributor_notion}(5)})\\
&=\PA(\mu,F^{\triangleleft}\circ\sY_{\bbB}y)&(\text{Definition of}\ F_{\nat}\ \text{and}\ F^{\triangleleft})\\
&=\bbX(\mu,Gy),
\end{align*}
hence  $\overline{F}$ is a left adjoint of $G$.

Second,  $F=\overline{F}\circ\eta_{(\bbA,C)}$. Note that for all $x\in\bbA_0$,
\begin{align*}
\bbB(Fx,-)&=F_{\nat}(x,-)\\
&=(\bbB\lda F^{\nat})(x,-)& (\text{Proposition \ref{adjoint_arrow_calculation}(1)})\\
&=\bbB\lda F^{\nat}(-,x)\\
&=\bbB\lda(\sY_{\bbB}\circ Fx)\\
&=\bbB\lda(F^{\ra}\circ\sY_{\bbA}x),&(\text{Proposition \ref{Yoneda_natural}})
\end{align*}
thus $F=\sup_{\bbB}\circ F^{\ra}\circ\sY_{\bbA}$. Consequently
\begin{align*}
\overline{F}\circ\eta_{(\bbA,C)}&={\sup}_{\bbB}\circ F^{\ra}\circ C\circ\sY_{\bbA}\\
&\leq{\sup}_{\bbB}\circ C_{\bbB}\circ F^{\ra}\circ\sY_{\bbA}&(F\ \text{is continuous})\\
&={\sup}_{\bbB}\circ \sY_{\bbB}\circ{\sup}_{\bbB}\circ F^{\ra}\circ\sY_{\bbA}\\
&={\sup}_{\bbB}\circ F^{\ra}\circ\sY_{\bbA}&({\sup}_{\bbB}\dv\sY_{\bbB}:\PB\rhu\bbB)\\
&=F.
\end{align*}
Conversely, since $C$ is a $\CQ$-closure operator, it is clear that
$$F={\sup}_{\bbB}\circ F^{\ra}\circ\sY_{\bbA}\leq{\sup}_{\bbB}\circ F^{\ra}\circ C\circ\sY_{\bbA}=\overline{F}\circ\eta_{(\bbA,C)},$$
hence $F\cong\overline{F}\circ\eta_{(\bbA,C)}$, and consequently $F=\overline{F}\circ\eta_{(\bbA,C)}$ since $\bbB$ is skeletal.

{\bf Uniqueness.} Suppose $H:\bbX\to\bbB$ is another left adjoint $\CQ$-functor that makes Diagram (\ref{eta_universal}) commute. For each $\mu\in\bbX$, since $C:\PA\to\bbX$ is a left adjoint in $\CQ$-$\Cat$, we have
$$\mu=C(\mu)=C(\mu\circ\bbA)=C\Big(\bv_{x\in\bbA_0}\mu(x)\circ\sY_{\bbA}x\Big)=\bv_{x\in\bbA_0}\mu(x)\otimes_{\bbX}C(\sY_{\bbA}x),$$
where the last equality follows from Example \ref{PA_tensor} and Proposition \ref{F_la_ra_condition}. It follows that
\begin{align*}
H(\mu)&=H\Big(\bv_{x\in\bbA_0}\mu(x)\otimes_{\bbX}C(\sY_{\bbA}x)\Big) \\
&=\bv_{x\in\bbA_0}\mu(x)\otimes_{\bbB}(H\circ\eta_{(\bbA,C)}(x))&\text{(Proposition \ref{F_la_ra_condition})}\\
&=\bv_{x\in\bbA_0}\mu(x)\otimes_{\bbB} Fx.
\end{align*}
Consequently,
\begin{align*}
\bbB(H(\mu),-)&=\bbB\Big(\bv_{x\in\bbA_0}\mu(x)\otimes_{\bbB} Fx,-\Big)\\
&=\bw_{x\in\bbA_0}\bbB(\mu(x)\otimes_{\bbB} Fx,-)&\text{(Proposition \ref{A_x_yi_distribution})}\\
&=\bw_{x\in\bbA_0}\bbB(Fx,-)\lda\mu(x)\\
&=F_\nat \lda\mu.
\end{align*}
Since $\bbB$ is skeletal, it follows that $H(\mu)=\colim_{\mu}F=\sup_{\bbB}\circ F^{\ra}(\mu)$. Therefore, $H=\sup_{\bbB}\circ F^{\ra}=\overline{F}$.
\end{proof}

The following corollary shows that the free cocompletion functor of $\CQ$-categories factors through the category of $\CQ$-closure spaces via a left adjoint functor $\BD:\CQ\text{-}\Cat\to\CQ\text{-}\Cls$ and a left adjoint functor $\CT:\CQ\text{-}\Cls\to\CQ\text{-}\CCat$.

\begin{cor} \label{P_factor_through_Cls}
Both the outer triangle and the inner triangle of the diagram
$$\bfig
\qtriangle/@{>}@<2pt>`@{>}@<-2pt>`@{>}@<2pt>/<1000,700>[\CQ\text{-}\Cat`\CQ\text{-}\Cls`\CQ\text{-}\CCat;\BD`\CP`\CT]
\qtriangle|brl|/@{<-}@<-2pt>`@{<-}@<2pt>`@{<-}@<-2pt>/<1000,700>[\CQ\text{-}\Cat`\CQ\text{-}\Cls`\CQ\text{-}\CCat;\BU`\BU`\CD]
\efig$$
commute, where $\BU$ denote the obvious forgetful functors.
\end{cor}

\begin{proof}
Follows immediately from the definitions of these functors.
\end{proof}

\chapter{Isbell Adjunctions in $\CQ$-categories} \label{Isbell_adjunction}

The Isbell adjunction introduced in this chapter generalize the classical Isbell adjunction in category theory (see Proposition \ref{Isbell_classical}). Each $\CQ$-distributor between $\CQ$-categories gives rise to an Isbell adjunction between $\CQ$-categories of contravariant presheaves and covariant presheaves, and hence to a monad. We prove that this process is functorial from the category of $\CQ$-distributors and infomorphisms to the category of complete $\CQ$-categories and left adjoint $\CQ$-functors. Furthermore, the free cocompletion functor of $\CQ$-categories factors through this functor.

\section{Isbell Adjunctions}

Given a $\CQ$-distributor $\phi:\bbA\oto\bbB$, define a pair of $\CQ$-functors
$$\uphi:\PA\to\PdB\quad\text{and}\quad\dphi:\PdB\to\PA$$
by
\begin{equation} \label{uphi_def}
\uphi(\mu)=\phi\lda\mu\quad\text{and}\quad\dphi(\lam)=\lam\rda\phi.
\end{equation}
It should be warned that $\uphi$ and $\dphi$ are both contravariant with respect to local orders in $\CQ$-$\Dist$ by Remark \ref{PdA_QDist_order}, i.e.,
\begin{equation} \label{uphi_contravariant}
\forall\mu_1,\mu_2\in\PA,\mu_1\leq\mu_2{}\Lra{}\uphi(\mu_2)\leq\uphi(\mu_1)
\end{equation}
and
\begin{equation} \label{dphi_contravariant}
\forall\lam_1,\lam_2\in\PdB,\lam_1\leq\lam_2{}\Lra{}\dphi(\lam_2)\leq\dphi(\lam_1).
\end{equation}

\begin{prop} \label{uphi-dphi-adjunction}
$\uphi\dv\dphi:\PA\rhu\PdB$ in $\CQ$-$\Cat$.
\end{prop}

\begin{proof}
For all $\mu\in\PA$ and $\lam\in\PdB$,
\begin{align*}
\PdB(\uphi(\mu),\lam)&=\lam\rda\uphi(\mu)\\
&=\lam\rda(\phi\lda\mu)\\
&=(\lam\rda\phi)\lda\mu\\
&=\dphi(\lam)\lda\mu\\
&=\PA(\mu,\dphi(\lam)).
\end{align*}
Hence the conclusion holds.
\end{proof}

We would like to stress that
\begin{equation} \label{uphi_dphi_order}
\mu\leq \dphi\circ\uphi(\mu)\quad\text{and}\quad\lam\leq\uphi\circ\dphi(\lam)
\end{equation}
for all $\mu\in\PA$ and $\lam\in\PdB$ by Remark \ref{PdA_QDist_order}.

The Isbell adjunction presented in Proposition \ref{ub_lb_adjunction} is a special case of Proposition \ref{uphi-dphi-adjunction}, i.e., letting $\bbB=\bbA$ and $\phi=\bbA$. So, the adjunction  $\uphi\dv\dphi:\PA\rhu\PdB$ is a generalization of the Isbell adjunction. As we shall see, all adjunctions between $\PA$ and $\PdB$ are of this form, and will  be called Isbell adjunctions by abuse of language.

Recall that each $\CQ$-functor $F:\bbA\to\PdB$ corresponds to a $\CQ$-distributor $\hF:\bbA\oto\bbB$ given by Equation (\ref{hF_def}), and each $\CQ$-functor $G:\bbB\to\PA$ corresponds to a $\CQ$-distributor $\hG:\bbA\oto\bbB$ given by Equation (\ref{hG_def}).

\begin{prop} \label{uphi_dphi_Yoneda}
Let $\phi:\bbA\oto\bbB$ be a $\CQ$-distributor, then $\ulc\uphi\circ\sY_{\bbA}\urc=\phi=\ulc\dphi\circ\sYd_{\bbB}\urc$.
\end{prop}

\begin{proof}
For all $x\in\bbA_0$ and $y\in\bbB_0$,
\begin{align*}
\ulc\uphi\circ\sY_{\bbA}\urc(x,y)&=(\uphi\circ\sY_{\bbA}x)(y)\\
&=(\phi\lda(\sY_{\bbA}x))(y)\\
&=\phi(-,y)\lda\bbA(-,x)\\
&=\phi(x,y)\\
&=\bbB(y,-)\rda\phi(x,-)\\
&=((\sYd_{\bbB}y)\rda\phi)(x)\\
&=(\dphi\circ\sYd_{\bbB}y)(x)\\
&=\ulc\dphi\circ\sYd_{\bbB}\urc(x,y),
\end{align*}
showing that the conclusion holds.
\end{proof}

Given a $\CQ$-distributor $\phi:\bbA\oto\bbB$,  it follows from Example \ref{adjunction_closure_interior} that $\dphi\circ\uphi:\PA\to\PA$ is a $\CQ$-closure operator and $\uphi\circ\dphi:\PdB\to\PdB$ is a $\CQ$-interior operator. For each $y\in\bbB_0$, since
\begin{equation} \label{phiF_fixed}
\olphi y=\phi(-,y)=\dphi\circ\sYd_{\bbB}y= \dphi\circ\uphi\circ\dphi\circ\sYd_{\bbB}y,
\end{equation}
it follows that $\olphi y=\phi(-,y)$ is closed in the $\CQ$-closure space $(\bbA, \dphi\circ\uphi)$. Dually, for all $x\in\bbA_0$,
\begin{equation} \label{Fphi_fixed}
\ulphi x=\phi(x,-)=\uphi\circ\sY_{\bbA}x=\uphi\circ\dphi\circ\uphi\circ\sY_{\bbA}x
\end{equation}
is a fixed point of the $\CQ$-interior operator $\uphi\circ\dphi$. These facts will be used in the proofs of Theorem \ref{F_U_adjunction} and Theorem \ref{complte_category_sup_dense}.

\begin{thm} \label{Isbell_distributor_bijection}
Let $\bbA$ and $\bbB$ be $\CQ$-categories. The correspondence $\phi\mapsto\uphi$ is an isomorphism of posets
$$\CQ\text{-}\Dist(\bbA,\bbB)\cong\CQ\text{-}\CCat^{\co}(\PA,\PdB).$$
\end{thm}

\begin{proof}
Let $F:\PA\to\PdB$ be a left adjoint $\CQ$-functor. We show that the correspondence $F\mapsto\ulc F\circ\sY_{\bbA}\urc$ is an inverse of the correspondence $\phi\mapsto\uphi$, and thus they are both isomorphisms of posets between $\CQ\text{-}\Dist(\bbA,\bbB)$ and $\CQ\text{-}\CCat^{\co}(\PA,\PdB)$.

Firstly, we show that both of the correspondences are order-preserving. Indeed,
\begin{align*}
&\phi\leq\psi\ \text{in}\ \CQ\text{-}\Dist(\bbA,\bbB)\\
\iff&\forall\mu\in\PA,\uphi(\mu)=\phi\lda\mu\leq\psi\lda\mu=\upsi(\mu)\ \text{in}\ \CQ\text{-}\Dist\\
\iff&\forall\mu\in\PA,\uphi(\mu)\geq\upsi(\mu)\ \text{in}\ (\PdB)_0\\
\iff&\uphi\leq\upsi\ \text{in}\ \CQ\text{-}\CCat^{\co}(\PA,\PdB)
\end{align*}
and
\begin{align*}
&F\leq G\ \text{in}\ \CQ\text{-}\CCat^{\co}(\PA,\PdB)\\
\iff&\forall\mu\in\PA,F(\mu)\geq G(\mu)\ \text{in}\ (\PdB)_0\\
\iff&\forall\mu\in\PA,F(\mu)\leq G(\mu)\ \text{in}\ \CQ\text{-}\Dist(\bbA,\bbB)\\
\Lra{}&\forall x\in\bbA_0,F\circ\sY_{\bbA}x\leq G\circ\sY_{\bbA}x\ \text{in}\ \CQ\text{-}\Dist(\bbA,\bbB)\\
\iff&\forall x\in\bbA_0,\ulc F\circ\sY_{\bbA}\urc(x,-)\leq\ulc G\circ\sY_{\bbA}\urc(x,-)\ \text{in}\ \CQ\text{-}\Dist(\bbA,\bbB)\\
\iff&\ulc F\circ\sY_{\bbA}\urc\leq\ulc G\circ\sY_{\bbA}\urc\ \text{in}\ \CQ\text{-}\Dist(\bbA,\bbB).
\end{align*}

Secondly,  $F=(\ulc F\circ\sY_{\bbA}\urc)_{\ua}$. For all $\mu\in\PA$, since $F$ is a left adjoint in $\CQ$-$\Cat$, by Example \ref{PA_tensor} and Proposition \ref{F_la_ra_condition} we have
\begin{align*}
F(\mu)&=F(\mu\circ\bbA)\\
&=F\Big(\bv_{x\in\bbA_0}\mu(x)\circ\sY_{\bbA}x\Big)\\
&=\bw_{x\in\bbA_0}(F\circ\sY_{\bbA}x)\lda\mu(x)\\
&=\ulc F\circ\sY_{\bbA}\urc\lda\mu\\
&=(\ulc F\circ\sY_{\bbA}\urc)_{\ua}(\mu).
\end{align*}

Finally, $\phi=\ulc\uphi\circ\sY_{\bbA}\urc$. This is obtained in Proposition \ref{uphi_dphi_Yoneda}.
\end{proof}

As a summary of Proposition \ref{distributor_functor_bijection} and  Theorem \ref{Isbell_distributor_bijection}, we have the following isomorphisms of posets:
\begin{align}
\CQ\text{-}\Dist(\bbA,\bbB)&\cong\CQ\text{-}\Cat^{\co}(\bbA,\PdB)\nonumber\\
&\cong\CQ\text{-}\Cat(\bbB,\PA)\label{Isbell_isomorphism} \\
&\cong\CQ\text{-}\CCat^{\co}(\PA,\PdB).\nonumber
\end{align}

\section{Functoriality of the Isbell adjunction}

In this section, we show that the construction of Isbell adjunctions is functorial from the category $\CQ$-$\Info$ to $\CQ$-$\Cls$, and thus to $\CQ$-$\CCat$.

\begin{prop} \label{F_G_info_F_uphi_dphi_continuous}
Let $(F,G):\phi\to\psi$ be an infomorphism between $\CQ$-distributors $\phi:\bbA\oto\bbB$ and $\psi:\bbA'\oto\bbB'$. Then $F:(\bbA,\uphi\circ\dphi)\to(\bbA',\upsi\circ\dpsi)$ is a continuous $\CQ$-functor.
\end{prop}

\begin{proof}
Consider the following diagram:
$$\bfig
\square|alrb|[\PA`\PdB`\PA'`\PdB';\uphi`F^{\ra}`G^{\nla}`\upsi]
\square(500,0)/>``>`>/[\PdB`\PA`\PdB'`\PA';\dphi``F^{\ra}`\dpsi]
\efig$$
We must prove $F^{\ra}\circ\dphi\circ\uphi\leq\dpsi\circ\upsi\circ F^{\ra}$. To this end, it suffices to check that
\begin{itemize}
\item[\rm (a)] the left square commutes if and only if $(F,G):\phi\to\psi$ is an infomorphism; and
\item[\rm (b)] $F^{\ra}\circ\dphi\leq\dpsi\circ G^{\nla}$ if and only if $G^{\nat}\circ\phi\leq\psi\circ F_{\nat}$.
\end{itemize}

For (a), suppose $G^{\nla}\circ\uphi=\upsi\circ F^{\ra}$, then for all $x\in\bbA_0$,
\begin{align*}
G^{\nat}\circ\phi(x,-)&=G^{\nla}(\phi(x,-))&(\text{Definition \ref{direct_inverse_image_def}})\\
&=G^{\nla}(\uphi\circ\sY_{\bbA}x)&\text{(Proposition \ref{uphi_dphi_Yoneda})}\\
&=\upsi(F^{\ra}\circ\sY_{\bbA}x)\\
&=\upsi(\sY_{\bbA}x\circ F^{\nat})&(\text{Definition \ref{direct_inverse_image_def}})\\
&=\psi\lda(\sY_{\bbA}x\circ F^{\nat})&(\text{Equation (\ref{uphi_def})})\\
&=(\psi\circ F_{\nat})\lda\bbA(-,x)&(\text{Proposition \ref{adjoint_arrow_calculation}(2)})\\
&=\psi\circ F_{\nat}(x,-).
\end{align*}

Conversely, if $(F,G):\phi\to\psi$ is an infomorphism, then for all $\mu\in\PA$,
\begin{align*}
G^{\nla}\circ\uphi(\mu)&=G^{\nat}\circ(\phi\lda\mu)&(\text{Equation (\ref{uphi_def})})\\
&=(G^{\nat}\circ\phi)\lda\mu&(\text{Proposition \ref{adjoint_arrow_calculation}(3)})\\
&=(\psi\circ F_{\nat})\lda\mu&(\text{Definition \ref{infomorphism_def}})\\
&=\psi\lda(\mu\circ F^{\nat})&(\text{Proposition \ref{adjoint_arrow_calculation}(2)})\\
&=\upsi\circ F^{\ra}(\mu).
\end{align*}

For (b), suppose $F^{\ra}\circ\dphi\leq\dpsi\circ G^{\nla}$, then for all $y'\in\bbB'_0$,
\begin{align*}
G^{\nat}(-,y')\circ\phi&=G_{\nat}(y',-)\rda\phi&(\text{Proposition \ref{adjoint_arrow_calculation}(1)})\\
&=\dphi(G_{\nat}(y',-))&(\text{Equation (\ref{uphi_def})})\\
&\leq F^{\la}\circ F^{\ra}\circ\dphi(G_{\nat}(y',-))&(F^{\ra}\dv F^{\la}:\PA\rhu\PA')\\
&\leq F^{\la}\circ\dpsi\circ G^{\nla}(G_{\nat}(y',-))\\
&=F^{\la}\circ\dpsi\circ G^{\nla}\circ G^{\nra}\circ\sYd_{\bbB'}y'&(\text{Definition \ref{direct_inverse_image_def}})\\
&\leq F^{\la}\circ\dpsi\circ\sYd_{\bbB'}y'&\text{(Formula (\ref{FLa_FRa_adjuntion}) and (\ref{dphi_contravariant}))}\\
&=F^{\la}(\psi(-,y'))&\text{(Proposition \ref{uphi_dphi_Yoneda})}\\
&=\psi(-,y')\circ F_{\nat}.&(\text{Definition \ref{direct_inverse_image_def}})
\end{align*}

Conversely, if $G^{\nat}\circ\phi\leq\psi\circ F_{\nat}$, then for all $\lam\in\PdB$,
\begin{align*}
F^{\ra}\circ\dphi(\lam)&=(\lam\rda\phi)\circ F^{\nat}&(\text{Equation (\ref{uphi_def})})\\
&\leq((G^{\nat}\circ\lam)\rda(G^{\nat}\circ\phi))\circ F^{\nat}\\
&\leq((G^{\nat}\circ\lam)\rda(\psi\circ F_{\nat}))\circ F^{\nat}\\
&\leq(G^{\nat}\circ\lam)\rda(\psi\circ F_{\nat}\circ F^{\nat})\\
&\leq(G^{\nat}\circ\lam)\rda\psi&(F_{\nat}\dv F^{\nat}:\bbA\rhu\bbB)\\
&=\dpsi\circ G^{\nla}(\lam).&(\text{Equation (\ref{uphi_def})})
\end{align*}
This completes the proof.
\end{proof}

By virtue of Proposition \ref{F_G_info_F_uphi_dphi_continuous} we obtain a functor $\CU:\CQ\text{-}\Info\to\CQ\text{-}\Cls$ that sends an infomorphism
$$(F,G):(\phi:\bbA\oto\bbB)\to(\psi:\bbA'\oto\bbB')$$
to a continuous $\CQ$-functor
$$F:(\bbA,\dphi\circ\uphi)\to(\bbA',\dpsi\circ\upsi).$$

Given a $\CQ$-closure space $(\bbA,C)$, define a $\CQ$-distributor $\zeta_C:\bbA\oto C(\PA)$ by
$$\zeta_C(x,\mu)=\mu(x)$$
for all $x\in\bbA_0$ and $\mu\in C(\PA)$. It is clear that $\zeta_C$ is obtained by restricting the domain and the codomain of the $\CQ$-distributor
\begin{equation} \label{lam_circ_mu}
\PdA\oto\PA, \quad (\lam,\mu)\mapsto \mu\circ\lam.
\end{equation}

Given a continuous $\CQ$-functor $F:(\bbA,C)\to(\bbB,D)$ between $\CQ$-closure spaces, consider the $\CQ$-functor $F^{\triangleleft}:D(\PB)\to C(\PA)$ that sends each closed contravariant presheaf $\lam$ to $F^{\triangleleft}(\lam)=F^{\la}(\lam)$. Then similar to Proposition \ref{Y_functor_Cat_Info} one can check that
$$(F,F^{\triangleleft}):(\zeta_C:\bbA\oto C(\PA))\to(\zeta_D:\bbB\oto D(\PB))$$
is an infomorphism. Thus, we obtain a functor $\CF:\CQ\text{-}\Cls\to\CQ\text{-}\Info$.

\begin{thm} \label{F_U_adjunction}
$\CF:\CQ\text{-}\Cls\to\CQ\text{-}\Info$ is a left adjoint and right inverse of $\CU:\CQ\text{-}\Info\to\CQ\text{-}\Cls$.
\end{thm}

\begin{proof}
{\bf Step 1.} $\CF$ is a right inverse of $\CU$.

For each $\CQ$-closure space $(\bbA,C)$, by the definition of the functor $\CF$, $\CF(\bbA,C)$ is the $\CQ$-distributor $\zeta_C:\bbA\oto C(\PA)$, where $\zeta_C(x,\mu)=\mu(x)$ for all $x\in\bbA_0$ and $\mu\in C(\PA)$. In order to prove $\CU\circ\CF(\bbA,C)=(\bbA,C)$, we show that $C=\zeta_C^\da\circ(\zeta_C)_\ua$.

For all $\mu\in\PA$ and $\lam\in C(\PA)$, since $C$ is a $\CQ$-functor,
$$\lam\lda\mu=\PA(\mu,\lam)\leq\PA(C(\mu),\lam)=\lam\lda C(\mu),$$
and consequently $C(\mu)\leq(\lam\lda\mu)\rda\lam$. Since $C$ is a $\CQ$-closure operator, we have
$$(C(\mu)\lda\mu)\rda C(\mu)\leq 1_{t\mu}\rda C(\mu)=C(\mu),$$
hence
\begin{align*}
C(\mu)&=\bw_{\lam\in C(\PA)}(\lam\lda\mu)\rda\lam\\
&=\bw_{\lam\in C(\PA)}(\zeta_C(-,\lam)\lda\mu)\rda\zeta_C(-,\lam)\\
&=\bw_{\lam\in C(\PA)}(\zeta_C)_\ua(\mu)(\lam)\rda\zeta_C(-,\lam)\\
&=\zeta_C^\da\circ(\zeta_C)_\ua(\mu),
\end{align*}
as required.

{\bf Step 2.} $\CF$ is a left adjoint of $\CU$.

For each $\CQ$-closure space $(\bbA,C)$, ${\rm id}_{(\bbA,C)}:(\bbA,C)\to\CU\circ\CF(\bbA,C)$ is clearly a continuous $\CQ$-functor and $\{{\rm id}_{(\bbA,C)}\}$ is a natural transformation from the identity functor on $\CQ$-$\Cls$ to $\CU\circ\CF$. Thus, it remains to show that for each $\CQ$-distributor $\psi:\bbA'\oto\bbB'$ and each continuous $\CQ$-functor $H:(\bbA,C)\to(\bbA',\dpsi\circ\upsi)$, there is a unique infomorphism
$$(F,G):\CF(\bbA,C)\to(\psi:\bbA'\oto\bbB')$$
such that the diagram
$$\bfig
\qtriangle<1000,500>[(\bbA,C)`\CU\circ\CF(\bbA,C)`(\bbA', \dpsi\circ\upsi);{\rm id}_{(\bbA,C)}`H`\CU(F,G)]
\efig$$
is commutative.

By definition, $\CF(\bbA,C)=\zeta_C:\bbA\oto C(\PA)$ and $\CU(F,G)=F$, where $\zeta_C(x,\mu)=\mu(x)$. Thus, we only need to show that there is a unique $\CQ$-functor $G:\bbB'\to C(\PA)$ such that
$$(H,G):(\zeta_C:\bbA\oto C(\PA))\to(\psi:\bbA'\oto\bbB')$$
is an infomorphism.

Let $G=H^{\triangleleft}\circ\olpsi:\bbB'\to C(\PA)$. That $G$ is well-defined follows from the fact that $\olpsi y'\in\dpsi\circ\upsi(\PA')$ for all $y'\in\bbB'_0$ by Equation (\ref{phiF_fixed}) and  that $H:(\bbA,C)\to(\bbA',\dpsi\circ\upsi)$ is  continuous.   Now we check that
$$(H,G):(\zeta_C:\bbA\oto C(\PA))\to(\psi:\bbA'\oto\bbB')$$
is an infomorphism. This is easy since
$$\zeta_C(x,Gy')=(Gy')(x)=H^{\triangleleft}\circ\olpsi(y')(x) =\olpsi(y')(Hx)=\psi(Hx,y')$$
for all $x\in\bbA_0$ and $y'\in\bbB'_0$. This proves the existence of $G$.

To see the uniqueness of $G$, suppose that $G':\bbB'\to C(\PA)$ is another $\CQ$-functor such that
$$(H,G'):(\zeta_C:\bbA\oto C(\PA))\to(\psi:\bbA'\oto\bbB')$$
is an infomorphism. Then for all $x\in\bbA_0$ and $y'\in\bbB'_0$,
\begin{align*}
(G'y')(x)&=\zeta_C(x,G'y')\\
&=\psi(Hx,y')\\
&=\olpsi(y')(Hx)\\
&=H^{\triangleleft}\circ\olpsi(y')(x)\\
&=(Gy')(x),
\end{align*}
hence $G'=G$.
\end{proof}

\begin{cor} \label{Cls_coreflective_Info}
The category $\CQ$-$\Cls$ is a coreflective subcategory of $\CQ$-$\Info$.
\end{cor}

The composition of
$$\CU:\CQ\text{-}\Info\to\CQ\text{-}\Cls$$
and
$$\CT:\CQ\text{-}\Cls\to\CQ\text{-}\CCat$$
gives a functor
\begin{equation} \label{M_def}
\CM=\CT\circ\CU:\CQ\text{-}\Info\to\CQ\text{-}\CCat
\end{equation}
$$\bfig
\Vtriangle/->`@{->}@<-2.5pt>`@{<-}@<2.5pt>/[\CQ\text{-}\Info`\CQ\text{-}\CCat`\CQ\text{-}\Cls;\CM`\CU`\CT]
\morphism(500,0)/@{->}@<-2.5pt>/<-500,500>[\CQ\text{-}\Cls`\CQ\text{-}\Info;\CF]
\morphism(500,0)/@{<-}@<2.5pt>/<500,500>[\CQ\text{-}\Cls`\CQ\text{-}\CCat;\CD]
\efig$$
that sends a $\CQ$-distributor $\phi:\bbA\oto\bbB$ to a complete $\CQ$-category $\dphi\circ\uphi(\PA)$. Conversely, since $\CF$ is a right inverse of $\CU$ (Theorem \ref{F_U_adjunction}) and $\CT$ is a left inverse of $\CD$ (up to isomorphism, Theorem \ref{T_D_adjunction}),  we have the following

\begin{thm} \label{complete_category_fca}
Every skeletal complete $\CQ$-category is isomorphic to $\CM(\phi)$ for some $\CQ$-distributor $\phi$.
\end{thm}

The following proposition shows that the free cocompletion functor of $\CQ$-categories factors through the functor $\CM$.

\begin{prop} \label{M_Yoneda_PA}
The diagram
$$\bfig
\qtriangle<700,500>[\CQ\text{-}\Cat`\CQ\text{-}\Info`\CQ\text{-}\CCat;\BY`\CP`\CM]
\efig$$
commutes.
\end{prop}

\begin{proof}
First, $\CM((\sY_{\bbA})_{\nat})= ((\sY_{\bbA})_{\nat})^{\da}\circ ((\sY_{\bbA})_{\nat})_{\ua}(\PA)=\PA$ for each $\CQ$-category $\bbA$. To see this, it suffices to check that
$$\mu= ((\sY_{\bbA})_{\nat})^{\da}\circ ((\sY_{\bbA})_{\nat})_{\ua}(\mu)= ((\sY_{\bbA})_{\nat}\lda\mu)\rda(\sY_{\bbA})_{\nat}$$
for all $\mu\in\PA$. On one hand, by Yoneda lemma we have
$$(\sY_{\bbA})_{\nat}\lda\mu=(\sY_{\bbA})_{\nat}\lda(\sY_{\bbA})_{\nat}(-,\mu)\geq\PA(\mu,-),$$
thus
$$((\sY_{\bbA})_{\nat}\lda\mu)\rda(\sY_{\bbA})_{\nat}\leq\PA(\mu,-)\rda(\sY_{\bbA})_{\nat}=(\sY_{\bbA})_{\nat}(-,\mu)=\mu.$$
On the other hand, $\mu\leq((\sY_{\bbA})_{\nat}\lda\mu)\rda(\sY_{\bbA})_{\nat}$ holds trivially.

Second, it is trivial that for each $\CQ$-functor $F:\bbA\to\bbB$,
$$\CM\circ\BY(F)=F^\ra=\CP(F).$$

Therefore, the conclusion holds.
\end{proof}

\section{$\CQ$-state property systems}

Corollary \ref{Cls_coreflective_Info} says that the category $\CQ$-$\Cls$ is a coreflective subcategory of $\CQ$-$\Info$. In this section we show that $\CQ$-$\Cls$ is equivalent to  a subcategory of $\CQ$-$\Info$. This equivalence is a generalization of that between closure spaces and state property systems in \cite{Aerts1999}.

\begin{defn}
A $\CQ$-state property system is a triple $(\bbA,\bbB,\phi)$, where $\bbA$ is a $\CQ$-category, $\bbB$ is a skeletal complete $\CQ$-category and $\phi:\bbA\oto\bbB$ is a $\CQ$-distributor, such that
\begin{itemize}
\item[\rm (1)] $\phi(-,{\inf}_{\bbB}\lam)=\lam\rda\phi$ for all $\lam\in\PdB$,
\item[\rm (2)] $\bbB(y,y')=\phi(-,y')\lda\phi(-,y)$ for all $y,y'\in\bbB_0$.
\end{itemize}
\end{defn}

$\CQ$-state property systems and infomorphisms constitute a category $\CQ$-$\Sp$, which is a subcategory of $\CQ$-$\Info$.

\begin{exmp}
For each $\CQ$-closure space $(\bbA,C)$,
$(\bbA,C(\PA),\zeta_C)$ is a $\CQ$-state property system.
First, for all $\Psi\in\CPd(C(\PA))$, it follows from Example \ref{PA_complete} and Equation (\ref{closure_system_infimum})  that
\begin{align*}
\zeta_C(-,{\inf}_{C(\PA)}\Psi)&= {\inf}_{C(\PA)}\Psi\\
&= \bw_{\mu\in C(\PA)}\Psi(\mu)\rda\mu\\
&=\bw_{\mu\in C(\PA)}\Psi(\mu)\rda\zeta_C(-,\mu)\\
&=\Psi\rda\zeta_C.
\end{align*}
Second, it is trivial that
$$C(\PA)(\mu,\lam)=\lam\lda\mu=\zeta_C(-,\lam) \lda\zeta_C(-,\mu)$$
for all $\mu,\lam\in C(\PA)$.
\end{exmp}

Therefore, the codomain of the functor $\CF:\CQ\text{-}\Cls\to\CQ\text{-}\Info$ can be restricted to the subcategory $\CQ$-$\Sp$.

\begin{thm} \label{F_U_equivalence}
The functors $\CF:\CQ\text{-}\Cls\to\CQ\text{-}\Sp$ and $\CU:\CQ\text{-}\Sp\to\CQ\text{-}\Cls$ establish an equivalence of categories.
\end{thm}

\begin{proof}
It is shown in Theorem \ref{F_U_adjunction}  that $\CU\circ\CF={\bf id}_{\CQ\text{-}\Cls}$, so, it suffices to prove that $\CF\circ\CU\cong {\bf id}_{\CQ\text{-}\Sp}$.

Given a $\CQ$-state property system $(\bbA,\bbB,\phi)$, we have by definition
$$\CF\circ\CU(\bbA,\bbB,\phi)=(\bbA,\dphi\circ\uphi(\PA), \zeta_{\dphi\circ\uphi}).$$
By virtue of Equation (\ref{phiF_fixed}), the images of the $\CQ$-functor $\olphi:\bbB\to\PA$ are contained in $\dphi\circ\uphi(\PA)$, so, it can be viewed as a $\CQ$-functor $\olphi:\bbB\to\dphi\circ\uphi(\PA)$. Since for any $x\in\bbA_0$ and $y\in\bbB_0$,
$$\phi(x,y)=(\olphi y)(x)=\zeta_{\dphi\circ\uphi}(x,\olphi y),$$
it follows that
$\eta_{\phi}=(1_{\bbA},\olphi)$ is an infomorphism from $\zeta_{\dphi\circ\uphi}: \bbA\oto\dphi\circ\uphi(\PA)$ to $\phi:\bbA\oto\bbB$.
Hence $\eta_{\phi}$ is a morphism from  $\CF\circ\CU(\bbA,\bbB,\phi)$ to  $(\bbA,\bbB,\phi)$ in $\CQ$-$\Sp$.
We claim that $\eta$ is a natural isomorphism from $\CF\circ\CU$ to the identity functor ${\bf id}_{\CQ\text{-}\Sp}$.

Firstly,  $\eta_{\phi}$ is an isomorphism. It suffices to show that
$$\olphi:\bbB\to\dphi\circ\uphi(\PA)$$
is an isomorphism between $\CQ$-categories.

Since
$$\bbB(y,y')=\phi(-,y')\lda\phi(-,y)=\PA(\olphi y,\olphi y')$$
for all $y,y'\in\bbB_0$, it follows that $\olphi$ is fully faithful. For each $\mu\in\PA$, let $y=\inf_{\bbB}\uphi(\mu)$, then
$$\olphi y=\phi(-,y)=\phi(-,{\inf}_{\bbB}\uphi(\mu))=\uphi(\mu)\rda\phi=\dphi\circ\uphi(\mu),$$
hence $\olphi$ is surjective. Since $\bbB$ is skeletal, we deduce that $\olphi:\bbB\to\dphi\circ\uphi(\PA)$ is an isomorphism.

Secondly, $\eta$ is natural. For this, we check the commutativity of the following diagram for any infomorphism $(F,G):(\bbA,\bbB,\phi)\to(\bbA',\bbB',\psi)$ between $\CQ$-state property systems:
$$\bfig
\Square[\CF\circ\CU(\bbA,\bbB,\phi)`(\bbA,\bbB,\phi)`\CF\circ\CU(\bbA',\bbB',\psi)`(\bbA',\bbB',\psi);
(1_{\bbA},\olphi)`(F,F^{\triangleleft})`(F,G)`(1_{\bbA'},\olpsi)] \efig$$
In fact, the equality $F\circ 1_{\bbA}=1_{\bbA'}\circ F$ is clear; and for all $x\in\bbA_0$ and $y'\in\bbB'_0$,
$$\olphi\circ G(y')(x)=\phi(x,Gy')=\psi(Fx,y')=\olpsi(y')(Fx)=F^{\triangleleft}\circ\olpsi(y')(x),$$
thus the conclusion follows.
\end{proof}

Together with Theorem \ref{T_D_adjunction} we have

\begin{cor}
The composition
$$\CT\circ\CU:\CQ\text{-}\Sp\to\CQ\text{-}\CCat$$
is a left adjoint of
$$\CF\circ\CD:\CQ\text{-}\CCat\to \CQ\text{-}\Sp.$$
\end{cor}

\section{Characterizations of $\CM(\phi)$} \label{Characterizations_of_Mphi}

For each $\CQ$-distributor $\phi:\bbA\oto\bbB$, we have the following characterization of the complete $\CQ$-category $\CM(\phi)$.

\begin{prop} \label{closure_system_fca}
Let $\bbX$ be a skeletal complete $\CQ$-category. The following conditions are equivalent:
\begin{itemize}
\item[\rm (1)] $\bbX$ is isomorphic to $\CM(\phi)$ for some $\CQ$-distributor $\phi:\bbA\oto\bbB$.
\item[\rm (2)] $\bbX$ is isomorphic to a $\CQ$-closure system of $\PA$ for some $\CQ$-category $\bbA$.
\item[\rm (3)] $\bbX$ is isomorphic to a quotient object of $\PA$ for some $\CQ$-category $\bbA$ in the category $\CQ$-$\CCat$, i.e., a subobject of $\PA$ in $(\CQ\text{-}\CCat)^{\op}$.
\end{itemize}
\end{prop}

\begin{proof}
(1)${}\Lra{}$(2): Trivial.

(2)${}\Lra{}$(1): It follows from Proposition \ref{closure_system_operator} that $\bbX$ is isomorphic to $C(\PA)$ for some $\CQ$-closure space $(\bbA,C)$. Let $\zeta_C=\CF(\bbA,C):\bbA\oto C(\PA)$, since $\CU\circ\CF={\bf id}_{\CQ\text{-}\Cls}$ (by Theorem \ref{F_U_adjunction}), we have
$$\CM(\zeta_C)=\CT\circ\CU(\zeta_C)=\CT\circ\CU\circ\CF(\bbA,C)=\CT(\bbA,C)=C(\PA).$$
This implies $\bbX\cong\CM(\zeta_C)$.

(2)$\iff$(3): It is easily seen that $\bbB$ is a $\CQ$-closure system of $\PA$ if and only if the inclusion $\CQ$-functor $I:\bbB\to\PA$ has a left adjoint $F:\PA\to\bbB$, where $F$ is epic in the category $\CQ$-$\CCat$, and equivalently a monic from $\bbB$ to $\PA$ in $(\CQ\text{-}\CCat)^{\op}$.
\end{proof}

In the rest this section, we present a characterization of $\CM(\phi)$ for a $\CQ$-distributor $\phi:\bbA\oto\bbB$ through $\sup$-dense and $\inf$-dense $\CQ$-functors.

Given a $\CQ$-distributor $\phi:\bbA\oto\bbB$, let  $\CM_{\phi}(\bbA,\bbB)$ denote the set of pairs $(\mu,\lam)\in\PA\times\PdB$  such that $\lam=\uphi(\mu)$ and $\mu=\dphi(\lam)$. $\CM_{\phi}(\bbA,\bbB)$ becomes a $\CQ$-typed set if we  assign $t(\mu,\lam)=t\mu=t\lam$. For $(\mu_1,\lam_1),(\mu_2,\lam_2)\in\CM_{\phi}(\bbA,\bbB)$, let
\begin{equation} \label{B_A_B_phi_order}
\CM_{\phi}(\bbA,\bbB)((\mu_1,\lam_1),(\mu_2,\lam_2))=\PA(\mu_1,\mu_2)=\PdB(\lam_1,\lam_2),
\end{equation}
Then $\CM_{\phi}(\bbA,\bbB)$ becomes a $\CQ$-category.

The projection
$$\pi_1:\CM_{\phi}(\bbA,\bbB)\to\PA,\quad (\mu,\lam)\mapsto\mu$$
is clearly a fully faithful $\CQ$-functor. Since the image of $\pi_1$ is exactly the set of fixed points of the $\CQ$-closure operator $\dphi\circ\uphi:\PA\to\PA$, we obtain that $\CM_{\phi}(\bbA,\bbB)$ is isomorphic to the complete $\CQ$-category $\CM(\phi)=\dphi\circ\uphi(\PA)$.

Similarly, the projection
$$\pi_2:\CM_{\phi}(\bbA,\bbB)\to\PdB,\quad (\mu,\lam)\mapsto\lam$$
is also a fully faithful $\CQ$-functor and the image of $\pi_2$ is exactly the set of fixed points of the $\CQ$-interior operator $\uphi\circ\dphi:\PdB\to\PdB$. Hence $\CM_{\phi}(\bbA,\bbB)$ is also isomorphic to the complete $\CQ$-category $\uphi\circ\dphi(\PdB)$, which is a $\CQ$-interior system of the skeletal complete $\CQ$-category $\PdB$.

Equation (\ref{B_A_B_phi_order}) shows that
$$\uphi:\dphi\circ\uphi(\PA)\to\uphi\circ\dphi(\PdB)$$
and
$$\dphi:\uphi\circ\dphi(\PdB)\to\dphi\circ\uphi(\PA)$$
are inverse to each other. Therefore, $\CM(\phi)(=\dphi\circ\uphi(\PA))$,  $\uphi\circ\dphi(\PdB)$  and $\CM_{\phi}(\bbA,\bbB)$ are isomorphic to each other.

\begin{defn} \label{sup_dense}
Let $F:\bbA\to\bbB$ be a $\CQ$-functor.
\begin{itemize}
\item[\rm (1)] $F$ is \emph{$\sup$-dense}  if for any $y\in\bbB_0$, there is some $\mu\in\PA$ such that $y=\sup_{\bbB}F^{\ra}(\mu)$.
\item[\rm (2)] $F$ is \emph{$\inf$-dense}  if for any $y\in\bbB_0$, there is some $\lam\in\PdA$ such that $y=\inf_{\bbB}F^{\nra}(\lam)$.
\end{itemize}
\end{defn}

\begin{exmp} \label{Yoneda_sup_dense}
For each $\CQ$-category $\bbA$, the Yoneda embedding $\sY:\bbA\to\PA$ is $\sup$-dense. Indeed, we have that
$$\mu={\sup}_{\PA}\circ\sY^\ra(\mu)$$
for all $\mu\in\PA$ (see Equation (\ref{mu_sup_ymu}) in the proof of Theorem \ref{T_D_adjunction}).

Dually, the co-Yoneda embedding $\sYd:\bbA\to\PdA$ is $\inf$-dense.
\end{exmp}

The following characterization of $\CM_{\phi}(\bbA,\bbB)$ (hence $\CM(\phi)$)  extends Theorem 4.8 in \cite{Lai2009} to the general setting.

\begin{thm} \label{complte_category_sup_dense}
Given a $\CQ$-distributor $\phi:\bbA\oto\bbB$, a skeletal complete $\CQ$-category $\bbX$ is isomorphic to $\CM_{\phi}(\bbA,\bbB)$  if and only if there exist a $\sup$-dense $\CQ$-functor $F:\bbA\to\bbX$ and an $\inf$-dense $\CQ$-functor $G:\bbB\to\bbX$ such that $\phi=G^{\nat}\circ F_{\nat}=\bbX(F-,G-)$.
\end{thm}

\begin{proof}
{\bf Necessity.} It suffices to prove the case $\bbX=\CM_{\phi}(\bbA,\bbB)$. Define $\CQ$-functors $F:\bbA\to\bbX$ and $G:\bbB\to\bbX$ by
\begin{equation} \label{Fa_Gb_def}
Fa=(\dphi\circ\ulphi a,\ulphi a),\quad Gb=(\olphi b,\uphi\circ\olphi b),
\end{equation}
then $F$ and $G$ are well defined by Equations (\ref{phiF_fixed}) and (\ref{Fphi_fixed}). It follows that
\begin{align*}
\bbX(F-,G-)&=\PA(\dphi\circ\ulphi-,\olphi-)\\
&=\PA(\dphi\circ\ulphi-,\dphi\circ\sYd_{\bbB}-)&(\text{Equation (\ref{phiF_fixed})})\\
&=\PdB(\uphi\circ\dphi\circ\ulphi-,\sYd_{\bbB}-)&(\text{Proposition \ref{uphi-dphi-adjunction}})\\
&=\PdB(\ulphi-,\sYd_{\bbB}-)&(\text{Equation (\ref{Fphi_fixed})})\\
&=(\ulphi-)(-)&(\text{Yoneda lemma})\\
&=\phi.
\end{align*}
Now we show that $F:\bbA\to\bbX$ is $\sup$-dense. For all $(\mu,\lam),(\mu',\lam')\in\bbX_0$,
\begin{align*}
\bbX((\mu,\lam),(\mu',\lam'))&=\lam'\rda\lam\\
&=\lam'\rda\uphi(\mu)\\
&=\lam'\rda(\phi\lda\mu)\\
&=(\lam'\rda\phi)\lda\mu\\
&=\PdB(\ulphi-,\lam')\lda\mu\\
&=\bbX(F-,(\mu',\lam'))\lda\mu&(\text{Equation (\ref{B_A_B_phi_order})})\\
&=F_{\nat}(-,(\mu',\lam'))\lda\mu,
\end{align*}
thus $(\mu,\lam)=\colim_{\mu}F={\sup}_{\bbX}\circ F^{\ra}(\mu)$, as desired.

That $G:\bbB\to\bbX$ is $\inf$-dense can be proved similarly.

{\bf Sufficiency.} We show that the type-preserving function
$$H:\bbX\to\CM_{\phi}(\bbA,\bbB),\quad Hx=(F_{\nat}(-,x),G^{\nat}(x,-))$$
is an isomorphism of $\CQ$-categories.

{\bf Step 1.} $\bbX=F_{\nat}\lda F_{\nat}=G^{\nat}\rda G^{\nat}$.

For all $x\in\bbX_0$, since $F:\bbA\to\bbX$ is $\sup$-dense, there is some $\mu\in\PA$ such that $x=\sup_{\bbX}F^{\ra}(\mu)$, thus
\begin{equation} \label{x_sup_Fmu}
\bbX(x,-)=\bbX\lda F^{\ra}(\mu)=\bbX\lda(\mu\circ F^{\nat})=(\bbX\circ F_{\nat})\lda\mu=F_{\nat}\lda\mu,
\end{equation}
where the third equality follows from Proposition \ref{adjoint_arrow_calculation}(2). Consequently
\begin{align*}
\bbX(x,-)&\leq F_{\nat}\lda F_{\nat}(-,x)\\
&\leq(F_{\nat}\lda F_{\nat}(-,x))\circ\bbX(x,x)\\
&=(F_{\nat}\lda F_{\nat}(-,x))\circ(F_{\nat}(-,x)\lda\mu)&(\text{Equation (\ref{x_sup_Fmu})})\\
&\leq F_{\nat}\lda\mu\\
&=\bbX(x,-),&(\text{Equation (\ref{x_sup_Fmu})})
\end{align*}
hence $\bbX(x,-)=F_{\nat}\lda F_{\nat}(-,x)=(F_{\nat}\lda F_{\nat})(x,-)$.

Since $G:\bbB\to\bbX$ is $\inf$-dense, similar calculations lead to $\bbX=G^{\nat}\rda G^{\nat}$.

{\bf Step 2.} $Hx\in\CM_{\phi}(\bbA,\bbB)$ for all $x\in\bbX_0$, thus $H$ is well defined. Indeed,
\begin{align*}
\uphi(F_{\nat}(-,x))&=\phi\lda F_{\nat}(-,x)\\
&=(G^{\nat}\circ F_{\nat})\lda F_{\nat}(-,x)&(\phi=G^{\nat}\circ F_{\nat})\\
&=G^{\nat}\circ(F_{\nat}\lda F_{\nat}(-,x))&(\text{Proposition \ref{adjoint_arrow_calculation}(3)})\\
&=G^{\nat}\circ\bbX(x,-)&(\text{Step 1})\\
&=G^{\nat}(x,-).
\end{align*}
Similar calculation shows that $\dphi(G^{\nat}(x,-))=F_{\nat}(-,x)$. Hence, $Hx\in\CM_{\phi}(\bbA,\bbB)$.

{\bf Step 3.} $H$ is a fully faithful $\CQ$-functor. Indeed, for all $x,x'\in\bbX_0$, by Step 1,
$$\bbX(x,x')=F_{\nat}(-,x')\lda F_{\nat}(-,x)=\PA(F_{\nat}(-,x),F_{\nat}(-,x'))=\CM_{\phi}(\bbA,\bbB)(Hx,Hx').$$

{\bf Step 4.} $H$ is surjective. For each pair $(\mu,\lam)\in\CM_{\phi}(\bbA,\bbB)$, we must show that there is some $x\in\bbX_0$ such that $F_{\nat}(-,x)=\mu$ and $G^{\nat}(x,-)=\lam$. Indeed, let $x=\sup_{\bbX}F^{\ra}(\mu)$, then
\begin{align*}
G^{\nat}(x,-)&=G^{\nat}\circ\bbX(x,-)\\
&=G^{\nat}\circ(F_{\nat}\lda\mu)&(\text{Equation (\ref{x_sup_Fmu})})\\
&=(G^{\nat}\circ F_{\nat})\lda\mu&(\text{Proposition \ref{adjoint_arrow_calculation}(3)})\\
&=\phi\lda\mu&( \phi=G^{\nat}\circ F_{\nat})\\
&=\uphi(\mu)\\
&=\lam,
\end{align*}
and it follows that $F_{\nat}(-,x)=\dphi(G^{\nat}(x,-))=\dphi(\lam)=\mu$.
\end{proof}

\section{The MacNeille completion}

Given a preordered set $A$, the MacNeille completion \cite{Davey2002,MacNeille1937} (or Dedekind-MacNeille completion) of $A$ is the set $\BM(A)$ of pairs $(L,U)$ of subsets of $A$, such that
$$L=\lb U\quad\text{and}\quad U=\ub L,$$
where $\lb U$ is the set of lower bounds of $U$ and $\ub L$ is the set of upper bounds of $L$. The preorder on $\BM(A)$ is given by
$$(L_1,U_1)\leq(L_2,U_2)\iff L_1\subseteq L_2\iff U_2\subseteq U_1.$$

The notion of MacNeille completion has been extended to categories enriched over a commutative unital quantale \cite{Bvelohlavek2004,Lai2007,Wagner1994} (see Chapter \ref{Applications} for more discussion about unital quantales). In this section, we investigate the MacNeille completion of a $\CQ$-category $\bbA$ by considering the special $\CQ$-distributor $\bbA:\bbA\oto\bbA$ and the induced Isbell adjunction
$$\ub\dv\lb:\PA\rhu\PdA$$
presented in Proposition \ref{ub_lb_adjunction}.

\begin{defn}
A \emph{cut} in a $\CQ$-category $\bbA$ is a pair $(\mu,\lam)\in\PA\times\PdA$ such that
$$\lam=\ub\mu\quad\text{and}\quad\mu=\lb\lam.$$
\end{defn}

\begin{defn} \label{MacNeille_completion}
The \emph{MacNeille completion} of a $\CQ$-category $\bbA$ is the $\CQ$-category $\BM(\bbA)$, in which
\begin{itemize}
\item[\rm (1)] $\BM(\bbA)_0$ is the set of cuts in $\bbA$;
\item[\rm (2)] the type of a cut is
               $$t(\mu,\lam)=t\mu=t\lam;$$
\item[\rm (3)] for all $(\mu_1,\lam_1),(\mu_2,\lam_2)\in\BM(\bbA)_0$,
               \begin{equation} \label{MA_order}
               \BM(\bbA)((\mu_1,\lam_1),(\mu_2,\lam_2))=\PA(\mu_1,\mu_2)=\PdA(\lam_1,\lam_2).
               \end{equation}
\end{itemize}
\end{defn}

It is obvious that the MacNeille completion $\BM(\bbA)$ of a $\CQ$-category $\bbA$ is exactly the complete $\CQ$-category $\CM_{\bbA}(\bbA,\bbA)$. In particular, the MacNeille completion $\BM(A)$ of a {\bf 2}-category $A$ is just the MacNeille completion of preordered sets. Thus, the MacNeille completion of $\CQ$-categories is a special case of the construction $\CM_{\phi}(\bbA,\bbB)$ for a $\CQ$-distributor $\phi:\bbA\oto\bbB$.

\begin{prop}
For each $\CQ$-category $\bbA$, the assignment
$$x\mapsto(\sY x,\sYd x)$$
gives rise to a fully faithful $\CQ$-functor
$$M:\bbA\to\BM(\bbA).$$
\end{prop}

\begin{proof}
It is easy to see that $(\sY x,\sYd x)$ is a cut in $\bbA$ for any $x\in\bbA_0$. Then the conclusion is an immediate consequence of the Yoneda lemma and Definition \ref{MacNeille_completion}.
\end{proof}

\begin{prop}
For each $\CQ$-category $\bbA$, the $\CQ$-functor $M:\bbA\to\BM(\bbA)$ preserves all existing weighted colimits and weighted limits in $\bbA$.
\end{prop}

\begin{proof}
It suffices to show that $M$ preserves all existing suprema and infima in $\bbA$. Suppose that $\mu\in\PA$ and ${\sup}_{\bbA}\mu$ exists, we must show that
$$M({\sup}_{\bbA}\mu)={\sup}_{\BM(\bbA)}M^{\ra}(\mu)={\colim}_{\mu}M.$$
Indeed, for all $(\mu',\lam')\in\BM(\bbA)$,
\begin{align*}
\BM(\bbA)(M({\sup}_{\bbA}\mu),(\mu',\lam'))&=\lam'\rda(\sYd\circ{\sup}_{\bbA}\mu)\\
&=\lam'\rda\bbA({\sup}_{\bbA}\mu,-)\\
&=\lam'\rda(\bbA\lda\mu)\\
&=(\lam'\rda\bbA)\lda\mu\\
&=\PdA(\sYd-,\lam')\lda\mu\\
&=\BM(\bbA)(M-,(\mu',\lam'))\lda\mu&(\text{Equation (\ref{MA_order})})\\
&=M_{\nat}(-,(\mu',\lam'))\lda\mu.
\end{align*}
Thus the conclusion follows. Similarly one can prove that
$$M({\inf}_{\bbA}\lam)={\inf}_{\BM(\bbA)}M^{\nra}(\lam)={\lim}_{\lam}M$$
whenever ${\inf}_{\bbA}\lam$ exists for some $\lam\in\PdA$.
\end{proof}

The MacNeille completion $\BM(\bbA)$ is the ``smallest'' completion of a $\CQ$-category $\bbA$ in the following sense.

\begin{prop}
Let $F:\bbA\to\bbB$ be a fully faithful $\CQ$-functor, with $\bbB$ complete. Then $F$ factors through $M:\bbA\to\BM(\bbA)$ via a fully faithful $\CQ$-functor $\overline{F}:\BM(\bbA)\to\bbB$.
\begin{equation} \label{MacNeille_smallest}
\bfig
\qtriangle<700,500>[\bbA`\BM(\bbA)`\bbB;M`F`\overline{F}]
\efig
\end{equation}
In particular, if $\bbA$ is a skeletal complete $\CQ$-category, then $\BM(\bbA)$ is isomorphic to $\bbA$. Therefore, the process of the MacNeille completion is idempotent (up to isomorphism).
\end{prop}

\begin{proof}
{\bf Step 1.} For a fully faithful $\CQ$-functor $F:\bbA\to\bbB$,  a fixed point $\mu$ of $\lb_{\bbA}\circ\ub_{\bbA}$ is also a fixed point of $(F_{\nat})^{\da}\circ(F_{\nat})_{\ua}$. On one hand,
\begin{align*}
&(\bbA\lda\mu)\circ((F_{\nat}\lda\mu)\rda F_{\nat})\\
={}&((F^{\nat}\circ F_{\nat})\lda\mu)\circ((F_{\nat}\lda\mu)\rda F_{\nat})&(\text{Proposition \ref{fully_faithful_graph_cograph}})\\
={}&F^{\nat}\circ(F_{\nat}\lda\mu)\circ((F_{\nat}\lda\mu)\rda F_{\nat})&(\text{Proposition \ref{adjoint_arrow_calculation}(3)})\\
\leq{}&F^{\nat}\circ F_{\nat}\\
={}&\bbA,&(\text{Proposition \ref{fully_faithful_graph_cograph}})
\end{align*}
and consequently
$$(F_{\nat})^{\da}\circ(F_{\nat})_{\ua}(\mu)=(F_{\nat}\lda\mu)\rda F_{\nat}\leq(\bbA\lda\mu)\rda\bbA=\mu.$$
On the other hand, $\mu\leq(F_{\nat})^{\da}\circ(F_{\nat})_{\ua}(\mu)$ holds trivially. Thus $\mu=(F_{\nat})^{\da}\circ(F_{\nat})_{\ua}(\mu)$.

{\bf Step 2.} The existence of $\overline{F}$. We identify $\BM(\bbA)$ with the $\CQ$-category $\lb_{\bbA}\circ\ub_{\bbA}(\PA)$ of the fixed points of $\lb_{\bbA}\circ\ub_{\bbA}$, then the functor $M$ is exactly the Yoneda embedding $\sY_{\bbA}$. Define
$$\overline{F}={\sup}_{\bbB}\circ F^{\ra},$$
then the commutativity of Diagram (\ref{MacNeille_smallest}) follows immediately from Proposition \ref{Yoneda_natural} and \ref{sup_functor}. It remains to show that $\overline{F}$ is fully faithful.  Indeed, let $\mu$ and $\mu'$ be two fixed points of $\lb_{\bbA}\circ\ub_{\bbA}$, then
\begin{align*}
&\bbB(\overline{F}(\mu),\overline{F}(\mu'))\\
={}&\bbB({\sup}_{\bbB}\circ F^{\ra}(\mu),{\sup}_{\bbB}\circ F^{\ra}(\mu'))\\
={}&\bbB({\sup}_{\bbB}\circ F^{\ra}(\mu'),-)\rda\bbB({\sup}_{\bbB}\circ F^{\ra}(\mu),-)\\
={}&(F_{\nat}\lda\mu')\rda(F_{\nat}\lda\mu)&(\text{Proposition \ref{colim_supremum}})\\
={}&((F_{\nat}\lda\mu')\rda F_{\nat})\lda\mu\\
={}&\mu'\lda\mu&(\text{Step 1})\\
={}&\PA(\mu,\mu').
\end{align*}

{\bf Step 3.} In particular, when $\bbA$ is a skeletal complete $\CQ$-category, let $F=1_{\bbA}:\bbA\to\bbA$, then $\overline{1_{\bbA}}=\sup_{\bbA}$. It follows Step 2 (or Proposition \ref{sup_functor}) that $\sup_{\bbA}\circ\sY_{\bbA}=1_{\bbA}$. Conversely, for all $\mu\in\lb_{\bbA}\circ\ub_{\bbA}(\PA)$,
\begin{align*}
\sY_{\bbA}\circ{\sup}_{\bbA}\mu&=\bbA(-,{\sup}_{\bbA}\mu)\\
&=\bbA({\sup}_{\bbA}\mu,-)\rda\bbA\\
&=(\bbA\lda\mu)\rda\bbA\\
&=\lb_{\bbA}\circ\ub_{\bbA}\mu\\
&=\mu.
\end{align*}
Thus $\sY_{\bbA}\circ\sup_{\bbA}=1_{\BM(\bbA)}$. This means that $\sY_{\bbA}$ and $\sup_{\bbA}$ are both isomorphisms between $\bbA$ and $\BM(\bbA)$.
\end{proof}

Finally, we present a comparison of the free cocompletion $\PA$ and the MacNeille completion $\BM(\bbA)$ (identified with $\lb\circ\ub(\PA)$) of a $\CQ$-category $\bbA$.

\begin{prop}
Let $\bbA$ be a $\CQ$-category.
\begin{itemize}
\item[\rm (1)] The Yoneda embedding $\sY:\bbA\to\PA$ is $\sup$-dense.
\item[\rm (2)] If the codomain of $\sY$ is restricted to $\BM(\bbA)$, then $\sY:\bbA\to\BM(\bbA)$ is $\inf$-dense.
\end{itemize}
\end{prop}

\begin{proof}
(1) has been obtained in Example \ref{Yoneda_sup_dense}. For (2), note that $\sY:\bbA\to\BM(\bbA)$ is the composition of the $\inf$-dense $\CQ$-functor given in Theorem \ref{complte_category_sup_dense} (Equation (\ref{Fa_Gb_def})) and the isomorphic projection from $\CM_{\bbA}(\bbA,\bbA)$ to $\lb\circ\ub(\PA)$. The conclusion thus follows.
\end{proof}


Recall that the $\CQ$-closure system $\BM(\bbA)=\lb\circ\ub(\PA)$ of $\PA$ is closed with respect to infima in $\PA$ (see Proposition \ref{closure_system_infima_closed}). Therefore, in the case that $\bbA$ is skeletal, if we identify $\bbA$ with the $\CQ$-subcategory
$$\sY(\bbA)=\{\sY x\mid x\in\bbA_0\}$$
of $\PA$, then the above proposition implies that $\BM(\bbA)$ is the \emph{closure} \cite{Albert1988,Lai2007} of $\bbA$ under the formation of weighted limits in $\PA$. As a comparison, $\PA$ is the closure of $\bbA$ under the formation of weighted colimits in $\PA$.

\chapter{Kan Adjunctions in $\CQ$-categories} \label{Kan_adjunction}

In this chapter, we first recall Kan extensions of $\CQ$-functors presented in \cite{Stubbe2005}. Then we introduce Kan adjunctions between $\CQ$-categories of contravariant presheaves and covariant presheaves arise from $\CQ$-distributors. We prove that this process is contravariant functorial from the category of $\CQ$-distributors and infomorphisms to the category of complete $\CQ$-categories and left adjoint $\CQ$-functors. Also, the free cocompletion functor of $\CQ$-categories factors through this functor.

\section{Kan extensions of $\CQ$-functors} \label{Kan_extension_of_Q_functors}

Kan extensions of $\CV$-functors introduced in Section \ref{Kan_extensions_V_functors} naturally give rise to the definitions of Kan extensions of $\CQ$-functors.

\begin{defn} \label{Kan_def}
Let $K:\bbA\to\bbC$ be a $\CQ$-functor and $\bbB$ a $\CQ$-category.
\begin{itemize}
\item[\rm (1)] The \emph{left Kan extension} of a $\CQ$-functor $F:\bbA\to\bbB$ along $K:\bbA\to\bbC$, if it exists, is a $\CQ$-functor
$$\Lan_K F:\bbC\to\bbB$$
satisfying
\begin{equation} \label{left_Kan_def}
\Lan_K F\leq S\iff F\leq S\circ K
\end{equation}
for any other $\CQ$-functor $S:\bbC\to\bbB$.
$$\bfig
\Vtriangle/->`->`<-/[\bbA`\bbB`\bbC;F`K`S]
\place(500,300)[\Downarrow]
\place(1250,300)[\iff]
\Vtriangle(1500,0)|alm|/->`->`<--/[\bbA`\bbB`\bbC;F`K`\Lan_K F]
\place(2000,300)[\Downarrow]
\morphism(2000,0)|r|/{@{>}@/_2.5em/}/<500,500>[\bbC`\bbB;S]
\place(2350,180)[\twoar(1,-1)]
\efig$$
\item[\rm (2)] The \emph{right Kan extension} of a $\CQ$-functor $F:\bbA\to\bbB$ along $K:\bbA\to\bbC$, if it exists, is a $\CQ$-functor
$$\Ran_K F:\bbC\to\bbB$$
satisfying
\begin{equation} \label{right_Kan_def}
S\leq\Ran_K F\iff S\circ K\leq F
\end{equation}
for any other $\CQ$-functor $S:\bbC\to\bbB$.
$$\bfig
\Vtriangle/->`->`<-/[\bbA`\bbB`\bbC;F`K`S]
\place(500,300)[\Uparrow]
\place(1250,300)[\iff]
\Vtriangle(1500,0)|alm|/->`->`<--/[\bbA`\bbB`\bbC;F`K`\Ran_K F]
\place(2000,300)[\Uparrow]
\morphism(2000,0)|r|/{@{>}@/_2.5em/}/<500,500>[\bbC`\bbB;S]
\place(2330,110)[\twoar(-1,1)]
\efig$$
\end{itemize}
\end{defn}

Given a $\CQ$-functor $K:\bbA\to\bbC$ and another $\CQ$-category $\bbB$, composing with $K$ yields an order-preserving function
$$K^*:\CQ\text{-}\Cat(\bbC,\bbB)\to\CQ\text{-}\Cat(\bbA,\bbB)$$
between the preordered sets of $\CQ$-functors, which sends a $\CQ$-functor $S:\bbC\to\bbB$ to $S\circ K:\bbA\to\bbB$.
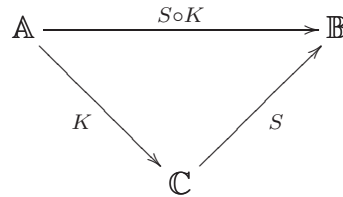
$$\bfig
\Vtriangle/->`->`<-/[\bbA`\bbB`\bbC;S\circ K`K`S]
\efig$$
If each $\CQ$-functor $F:\bbA\to\bbB$ has a left Kan extension $\Lan_K F:\bbC\to\bbB$ along $K$, then we obtain an adjunction in ${\bf 2}$-$\Cat$
\begin{equation} \label{LanK_Kstar_adjonit}
\Lan_K\dv K^*:\CQ\text{-}\Cat(\bbA,\bbB)\rhu\CQ\text{-}\Cat(\bbC,\bbB).
\end{equation}
Dually, if each $\CQ$-functor $F:\bbA\to\bbB$ has a right Kan extension $\Ran_K F:\bbC\to\bbB$ along $K$, then we obtain an adjunction in ${\bf 2}$-$\Cat$
\begin{equation} \label{Kstar_RanK_adjonit}
K^*\dv\Ran_K:\CQ\text{-}\Cat(\bbC,\bbB)\rhu\CQ\text{-}\Cat(\bbA,\bbB).
\end{equation}

Given a $\CQ$-functor $F:\bbA\to\bbB$, if the codomain $\bbB$ admits certain colimits (or equivalently, certain suprema), then the left Kan extension of $F$ along some $\CQ$-functor $K:\bbA\to\bbC$ can be constructed \emph{pointwise} for each object $c\in\bbC_0$. Dually, if the codomain $\bbB$ admits certain limits (or equivalently, certain infima), then the right Kan extension of $F$ along $K$ can be constructed pointwise.

\begin{prop} {\rm\cite{Stubbe2005}} \label{Kan_limit}
Let $F:\bbA\to\bbB$ and $K:\bbA\to\bbC$ be $\CQ$-functors.
\begin{itemize}
\item[\rm (1)] The left Kan extension of $F$ along $K$ can be computed by
$$(\Lan_K F)c={\colim}_{K_{\nat}(-,c)}F={\sup}_{\bbB}F^{\ra}(K_{\nat}(-,c))$$
if the weighted colimit exists for each $c\in\bbC_0$.
\item[\rm (2)] The right Kan extension of $F$ along $K$ can be computed by
$$(\Ran_K F)c={\lim}_{K^{\nat}(c,-)}F={\inf}_{\bbB}F^{\nra}(K^{\nat}(c,-))$$
if the weighted limit exists for each $c\in\bbC_0$.
\end{itemize}
\end{prop}

\begin{proof}
We prove (1) for example. For each $\CQ$-functor $S:\bbC\to\bbB$,
\begin{align*}
&\forall c\in\bbC_0,{\colim}_{K_{\nat}(-,c)}F\leq Sc\\
\iff&\forall c\in\bbC_0,1_{tc}\leq\bbB({\colim}_{K_{\nat}(-,c)}F,Sc)\\
\iff&\forall c\in\bbC_0,1_{tc}\leq F_{\nat}(-,Sc)\lda K_{\nat}(-,c)&\text{(Definition \ref{limit_colimit_def})}\\
\iff&\forall c\in\bbC_0,K_{\nat}(-,c)\leq F_{\nat}(-,Sc)\\
\iff&\forall c\in\bbC_0,K_{\nat}(-,c)\leq (S^{\nat}\circ F_{\nat})(-,c)&\text{(Proposition \ref{graph_cograph_distributor})}\\
\iff&K_{\nat}\leq S^{\nat}\circ F_{\nat}\\
\iff&S_{\nat}\circ K_{\nat}\leq F_{\nat}&(S_{\nat}\dv S^{\nat}:\bbC\rhu\bbB)\\
\iff&(S\circ K)_{\nat}\leq F_{\nat}&\text{(Proposition \ref{graph_cograph_functor})}\\
\iff&F\leq S\circ K.&\text{(Formula (\ref{functor_graph_order}))}
\end{align*}
Thus ${\colim}_{K_{\nat}(-,c)}F=(\Lan_K F)c$.
\end{proof}

By Proposition \ref{Kan_limit}, we arrive at the following corollary that extends the equivalent characterizations of complete $\CQ$-categories in Theorem \ref{complete_cocomplete_equivalent}.

\begin{cor} \label{Kan_complete}
Let $\bbB$ be a $\CQ$-category. The following conditions are equivalent:
\begin{itemize}
\item[\rm (1)] $\bbB$ is complete.
\item[\rm (2)] The left Kan extension of each $\CQ$-functor $F:\bbA\to\bbB$ along any $K:\bbA\to\bbC$ exists.
\item[\rm (3)] The right Kan extension of each $\CQ$-functor $F:\bbA\to\bbB$ along any $K:\bbA\to\bbC$ exists.
\item[\rm (4)] The left Kan extension of each $\CQ$-functor $F:\bbA\to\bbB$ along $\sY_{\bbA}:\bbA\to\PA$ exists.
\item[\rm (5)] The right Kan extension of each $\CQ$-functor $F:\bbA\to\bbB$ along $\sYd_{\bbA}:\bbA\to\PdA$ exists.
\end{itemize}
In this case,
$${\colim}_{\mu}F=(\Lan_{\sY_{\bbA}}F)\mu\quad\text{and}\quad{\lim}_{\lam}F=(\Ran_{\sYd_{\bbA}}F)\lam$$
for each $\CQ$-functor $F:\bbA\to\bbB$ and $\mu\in\PA$, $\lam\in\PdA$.
\end{cor}

\begin{proof}
We prove (1)$\iff$(2)$\iff$(4) for example.

(1)${}\Lra{}$(2): Proposition \ref{Kan_limit}.

(2)${}\Lra{}$(4): Trivial.

(4)${}\Lra{}$(1): For each $\CQ$-functor $F:\bbA\to\bbB$ and $\mu\in\PA$,
$${\colim}_{\mu}F={\colim}_{(\sY_{\bbA})_{\nat}(-,\mu)}F=(\Lan_{\sY_{\bbA}}F)\mu,$$
where the first equality follows from the Yoneda lemma, and the second equality follows from Proposition \ref{Kan_limit}. Thus $\bbB$ is complete.
\end{proof}

If the $\CQ$-functor $K:\bbA\to\bbC$ is fully faithful, then the pointwise Kan extension obtained in Proposition \ref{Kan_limit} is actually an extension, as indicated in the following conclusion.

\begin{cor} {\rm\cite{Stubbe2005}} \label{Kan_fully_faithful_extension}
Let $F:\bbA\to\bbB$ and $K:\bbA\to\bbC$ be $\CQ$-functors, with $K$ fully faithful.
\begin{itemize}
\item[\rm (1)] If the pointwise left Kan extension of $F$ along $K$ exists, then $(\Lan_K F)\circ K\cong F$.
\item[\rm (2)] If the pointwise right Kan extension of $F$ along $K$ exists, then $(\Ran_K F)\circ K\cong F$.
\end{itemize}
\end{cor}

\begin{proof}
We prove (1) for example. For each $a\in\bbA_0$,
\begin{align*}
\bbB((\Lan_K F)\circ Ka,-)&=\bbB({\colim}_{K_{\nat}(-,Ka)}F,-)&\text{(Proposition \ref{Kan_limit})}\\
&=\bbB({\colim}_{\bbC(K-,Ka)}F,-)\\
&=\bbB({\colim}_{\bbA(-,a)}F,-)&(K\text{ is fully faithful})\\
&=F_{\nat}\lda\bbA(-,a)&\text{(Definition \ref{limit_colimit_def})}\\
&=F_{\nat}(a,-)\\
&=\bbB(Fa,-).
\end{align*}
Thus $(\Lan_K F)\circ K\cong F$.
\end{proof}


\section{Kan adjunctions} \label{Kan_adjunctions_section}

Given a $\CQ$-distributor $\phi:\bbA\oto\bbB$, recall that composing with $\phi$ yields two $\CQ$-functors
$$\phi^*:\PB\to\PA\quad\text{and}\quad\phi^{\dag}:\PdA\to\PdB$$
defined by Equation (\ref{phidag_def}), i.e.,
$$\phi^*(\lam)=\lam\circ\phi\quad\text{and}\quad\phi^{\dag}(\mu)=\phi\circ\mu.$$
Define another two $\CQ$-functors
$$\phi_*:\PA\to\PB\quad\text{and}\quad\phi_{\dag}:\PdB\to\PdA$$
by
\begin{equation} \label{phistar_def}
\phi_*(\mu)=\mu\lda\phi\quad\text{and}\quad\phi_{\dag}(\lam)=\phi\rda\lam.
\end{equation}

We remind the readers that $\phi^{\dag}$ and $\phi_{\dag}$ are both covariant with respect to local orders in $\CQ$-$\Dist$. Indeed,
\begin{align*}
\mu_1\leq\mu_2\ \text{in} \ \CQ\text{-}\Dist\iff&\mu_2\leq\mu_1\ \text{in} \ (\PdA)_0&\text{(Remark \ref{PdA_QDist_order})}\\
\Lra{}&\phi^{\dag}(\mu_2)\leq\phi^{\dag}(\mu_1)\ \text{in} \ (\PdB)_0\\
\iff&\phi^{\dag}(\mu_1)\leq\phi^{\dag}(\mu_2)\ \text{in} \ \CQ\text{-}\Dist.
\end{align*}
Similarly one can deduce that
$$\lam_1\leq\lam_2\ \text{in} \ \CQ\text{-}\Dist{}\Lra{}\phi_{\dag}(\lam_1)\leq\phi_{\dag}(\lam_2)\ \text{in} \ \CQ\text{-}\Dist.$$

\begin{prop} \label{phistar_adjoint}
$\phi^*\dv\phi_*:\PB\rhu\PA$ and $\phi_{\dag}\dv\phi^{\dag}:\PdB\rhu\PdA$ in $\CQ$-$\Cat$.
\end{prop}

\begin{proof}
For all $\lam\in\PB$ and $\mu\in\PA$,
\begin{align*}
\PA(\phi^*(\lam),\mu)&=\mu\lda\phi^*(\lam)\\
&=\mu\lda(\lam\circ\phi)\\
&=(\mu\lda\phi)\lda\lam\\
&=\phi_*(\mu)\lda\lam\\
&=\PB(\lam,\phi_*(\mu)).
\end{align*}
Thus $\phi^*\dv\phi_*:\PB\rhu\PA$. Similarly, for all $\lam\in\PdB$ and $\mu\in\PdA$,
\begin{align*}
\PdA(\phi_{\dag}(\lam),\mu)&=\mu\rda\phi_{\dag}(\lam)\\
&=\mu\rda(\phi\rda\lam)\\
&=(\phi\circ\mu)\rda\lam\\
&=\phi^{\dag}(\mu)\rda\lam\\
&=\PdB(\lam,\phi^{\dag}(\mu)).
\end{align*}
Thus $\phi_{\dag}\dv\phi^{\dag}:\PdB\rhu\PdA$.
\end{proof}

We would like to stress that
\begin{equation} \label{phistar_adjuntion}
\lam\leq\phi_*\circ\phi^*(\lam)\quad\text{and}\quad\phi^*\circ\phi_*(\mu)\leq\mu
\end{equation}
for all $\lam\in\PB$ and $\mu\in\PA$; whereas
\begin{equation} \label{phidag_adjuntion}
\nu\leq\phi_{\dag}\circ\phi^{\dag}(\nu)\quad\text{and}\quad\phi^{\dag}\circ\phi_{\dag}(\ga)\leq\ga
\end{equation}
for all $\nu\in\PdA$ and $\ga\in\PdB$ by Remark \ref{PdA_QDist_order}.

\begin{prop} \label{phi_star_Yoneda}
Let $\phi:\bbA\oto\bbB$ be a $\CQ$-distributor, then $\ulc\phi^*\circ\sY_{\bbB}\urc=\phi=\ulc\phi^{\dag}\circ\sYd_{\bbA}\urc$.
\end{prop}

\begin{proof}
For all $x\in\bbA_0$ and $y\in\bbB_0$,
\begin{align*}
\ulc\phi^*\circ\sY_{\bbB}\urc(x,y)&=(\phi^*\circ\sY_{\bbB}y)(x)\\
&=(\sY_{\bbB}y\circ\phi)(x)\\
&=\bbB(-,y)\circ\phi(x,-)\\
&=\phi(x,y)\\
&=\phi(-,y)\circ\bbA(x,-)\\
&=(\phi\circ\sYd_{\bbA}x)(y)\\
&=(\phi^{\dag}\circ\sYd_{\bbA}x)(y)\\
&=\ulc\phi^{\dag}\circ\sYd_{\bbA}\urc(x,y),
\end{align*}
showing that the conclusion holds.
\end{proof}

Given a $\CQ$-distributor $\phi:\bbA\oto\bbB$, for each $y\in\bbB_0$, since
\begin{equation} \label{phiF_star_fixed}
\olphi y=\phi(-,y)=\phi^*\circ\sY_{\bbB}y=\phi^*\circ\phi_*\circ\phi^*\circ\sY_{\bbB}y,
\end{equation}
it follows that $\olphi y=\phi(-,y)$ is a fixed point of the $\CQ$-interior operator $\phi^*\circ\phi_*:\PB\to\PB$. Similarly, for all $x\in\bbA_0$,
\begin{equation} \label{phiF_dag_fixed}
\ulphi x=\phi(x,-)=\phi^{\dag}\circ\sYd_{\bbA}x=\phi^{\dag}\circ\phi_{\dag}\circ\phi^{\dag}\circ\sYd_{\bbA}x
\end{equation}
is a fixed point of the $\CQ$-closure operator $\phi^{\dag}\circ\phi_{\dag}:\PdB\to\PdB$.

If $\phi:\bbA\rhu\bbB$ is itself a left adjoint $\CQ$-distributor, then $\phi^*$ and $\phi_{\dag}$ are not only left adjoint $\CQ$-functors, but also right adjoint $\CQ$-functors as asserted in the following proposition. It is also an analogue of Proposition \ref{adjoint_graph} for $\CQ$-distributors.

\begin{prop} \label{adjoint_distributor_Kan}
Let $\phi:\bbA\oto\bbB$ and $\psi:\bbB\oto\bbA$ be a pair of $\CQ$-distributors. The following conditions are equivalent:
\begin{itemize}
\item[\rm (1)] $\phi\dv\psi:\bbA\rhu\bbB$ in $\CQ$-$\Dist$.
\item[\rm (2)] $\phi^*=\psi_*$.
\item[\rm (3)] $\psi^*\dv\phi^*:\PA\rhu\PB$ in $\CQ$-$\Cat$.
\item[\rm (4)] $\psi_*\dv\phi_*:\PB\rhu\PA$ in $\CQ$-$\Cat$.
\item[\rm (5)] $\phi_{\dag}=\psi^{\dag}$.
\item[\rm (6)] $\psi^{\dag}\dv\phi^{\dag}:\PdA\rhu\PdB$ in $\CQ$-$\Cat$.
\item[\rm (7)] $\psi_{\dag}\dv\phi_{\dag}:\PdB\rhu\PdA$ in $\CQ$-$\Cat$.
\end{itemize}
\end{prop}

\begin{proof}
(1)${}\Lra{}$(2): By Proposition \ref{adjoint_arrow_calculation}(1), for all $\lam\in\PB$,
$$\phi^*(\lam)=\lam\circ\phi=\lam\lda\psi=\psi_*(\lam).$$

(2)${}\Lra{}$(1): We must show that $\bbA\leq\psi\circ\phi$ and $\phi\circ\psi\leq\bbB$. Indeed, for all $x\in\bbA_0$ and $y\in\bbB_0$,
\begin{align*}
\psi(-,x)\circ\phi&=\phi^*(\psi(-,x))\\
&=\psi_*\circ\psi^*\circ\sY_{\bbA}x&\text{(Proposition \ref{phi_star_Yoneda})}\\
&\geq 1_{\PA}\circ\sY_{\bbA}x&\text{(Equation (\ref{phistar_adjuntion}))}\\
&=\bbA(-,x)
\end{align*}
and
\begin{align*}
\phi(-,y)\circ\psi&=\psi^*(\phi(-,y))\\
&=\psi^*\circ\phi^*\circ\sY_{\bbB}y&\text{(Proposition \ref{phi_star_Yoneda})}\\
&=\psi^*\circ\psi_*\circ\sY_{\bbB}y\\
&\leq 1_{\PB}\circ\sY_{\bbB}y&\text{(Equation (\ref{phistar_adjuntion}))}\\
&=\bbB(-,y).
\end{align*}

(1)$\iff$(3)$\iff$(4): Follows immediately from (1)$\iff$(2) and Corollary \ref{adjoint_arrow_representation}.

(1)$\iff$(5): Similar to (1)$\iff$(2).
%

(1)$\iff$(6)$\iff$(7): Follows immediately from (1)$\iff$(5) and Corollary \ref{adjoint_arrow_representation}.
\end{proof}

Therefore, if a $\CQ$-distributor $\phi$ has a right adjoint $\psi$ in $\CQ$-$\Dist$, then $\phi^*$ has both a right adjoint $\phi_*$ and a left adjoint $\psi^*$ in $\CQ$-$\Cat$.

In particular, given a $\CQ$-functor $F:\bbA\to\bbB$, since the cograph $F^{\nat}:\bbB\oto\bbA$ of $F$ is the right adjoint of the graph $F_\nat :\bbA\oto\bbB$ of $F$, it follows that $(F^{\nat})_*=(F_{\nat})^*$ is the right adjoint of $(F^{\nat})^*$, and $(F_{\nat})_{\dag}=(F^{\nat})^{\dag}$ is the left adjoint of $(F_{\nat})^{\dag}$.

Since $F^{\la}:\PB\to\PA$ and $F^{\nla}:\PdB\to\PdB$ are respectively the counterparts of the functor $-\circ F$ and $F\circ -$ for $\CQ$-categories, we arrive at the following conclusion which asserts that the adjunctions $\phi^*\dv\phi_*$ and $\phi_{\dag}\dv\phi^{\dag}$ generalize Kan extensions (Corollary \ref{monoidal_Kan}) in category theory.

\begin{thm} \label{why_kan}
For each $\CQ$-functor $F:\bbA\to\bbB$, it holds that
$$F^{\ra}=(F^{\nat})^*\dv (F^{\nat})_*= F^{\la}=(F_{\nat})^*\dv(F_{\nat})_*$$
and
$$(F^{\nat})_{\dag}\dv (F^{\nat})^{\dag}=F^{\nla}=(F_{\nat})_{\dag}\dv(F_{\nat})^{\dag}=F^{\nra}.$$
\end{thm}

\begin{proof}
It remains to prove $(F^{\nat})_*= F^{\la}$ and $(F^{\nat})^{\dag}=F^{\nla}$. Indeed, for each $\lam\in\PB$ and $\ga\in\PdB$, by Definition \ref{direct_inverse_image_def},
$$F^{\la}(\lam)=\lam\circ F_{\nat}=(F_{\nat})^*(\lam)\quad\text{and}\quad F^{\nla}(\ga)=F^{\nat}\circ\ga=(F^{\nat})^{\dag}(\ga).$$
\end{proof}

So, adjunctions of the forms $\phi^*\dv\phi_*:\PB\rhu\PA$ and $\phi_{\dag}\dv\phi^{\dag}:\PdB\rhu\PdA$ will be called Kan adjunctions by abuse of language.

\begin{rem}
In Section \ref{Kan_extension_of_Q_functors} we have seen that if each $\CQ$-functor $F:\bbA\to\bbB$ has left and right Kan extensions along $K:\bbA\to\bbC$, then the order-preserving function ``composing with $K$''
$$K^*:\CQ\text{-}\Cat(\bbC,\bbB)\to\CQ\text{-}\Cat(\bbA,\bbB)$$
is both a left and right adjoint in ${\bf 2}$-$\Cat$ (not in $\CQ$-$\Cat$, since $\CQ\text{-}\Cat(\bbA,\bbB)$ is not a $\CQ$-category in general).

However, if we consider the graph $F_{\nat}$ of $F$ instead, then the $\CQ$-functor ``composing with $F_{\nat}$''
$$(F_{\nat})^*:\PB\to\PA$$
is both a left and right adjoint in $\CQ$-$\Cat$, as presented in Theorem \ref{why_kan}.
\end{rem}

\begin{cor}  {\rm\cite{Pu2014}} \label{adjoint_image}
Let $F:\bbA\to\bbB$ and $G:\bbB\to\bbA$ be a pair of $\CQ$-functors. The following conditions are equivalent:
\begin{itemize}
\item[\rm (1)] $F\dv G:\bbA\rhu\bbB$.
\item[\rm (2)] $F^{\ra}\dv G^{\ra}:\PA\rhu\PB$.
\item[\rm (3)] $G^{\la}\dv F^{\la}:\PB\rhu\PA$.
\item[\rm (4)] $F^{\nra}\dv G^{\nra}:\PdA\rhu\PdB$.
\item[\rm (5)] $G^{\nla}\dv F^{\nla}:\PdB\rhu\PdA$.
\end{itemize}
\end{cor}

\begin{proof}
This is an immediate consequence of Proposition \ref{adjoint_graph},  Proposition \ref{adjoint_distributor_Kan} and Theorem \ref{why_kan}. We prove (1)$\iff$(2) for example.
\begin{align*}
&F\dv G:\bbA\rhu\bbB\\
\iff&G^{\nat}\dv F^{\nat}:\bbA\rhu\bbB&(\text{Proposition \ref{adjoint_graph}(4)})\\
\iff&(F^{\nat})^*\dv (G^{\nat})^*:\PA\rhu\PB&(\text{Proposition \ref{adjoint_distributor_Kan}(3)})\\
\iff&F^{\ra}\dv G^{\ra}:\PA\rhu\PB.&(\text{Theorem \ref{why_kan}})
\end{align*}
\end{proof}

The following theorem states that all adjunctions between $\PB$ and $\PA$ are of the form $\phi^*\dv\phi_*$, and all adjunctions between $\PdB$ and $\PdA$ are of the form $\phi_{\dag}\dv\phi^{\dag}$.

\begin{thm} \label{Kan_distributor_bijection}
Let $\bbA$ and $\bbB$ be $\CQ$-categories.
\begin{itemize}
\item[\rm (1)] The correspondence $\phi\mapsto\phi^*$ is an isomorphism of posets
$$\CQ\text{-}\Dist(\bbA,\bbB)\cong\CQ\text{-}\CCat(\PB,\PA).$$
\item[\rm (2)] The correspondence $\phi\mapsto\phi^{\dag}$ is an isomorphism of posets
$$\CQ\text{-}\Dist(\bbA,\bbB)\cong(\CQ\text{-}\CCat^{\dag})^{\co}(\PdA,\PdB).$$
\end{itemize}
\end{thm}

\begin{proof}
(1) Let $F:\PB\to\PA$ be a left adjoint $\CQ$-functor. We show that the correspondence $F\mapsto\ulc F\circ\sY_{\bbB}\urc$ is an inverse of the correspondence $\phi\mapsto\phi^*$, and thus they are both isomorphisms of posets between $\CQ\text{-}\Dist(\bbA,\bbB)$ and $\CQ\text{-}\CCat(\PB,\PA)$.

Firstly, we show that both of the correspondences are order-preserving. Indeed,
\begin{align*}
&\phi\leq\psi\ \text{in}\ \CQ\text{-}\Dist(\bbA,\bbB)\\
\iff&\forall\lam\in\PB,\phi^*(\lam)=\lam\circ\phi\leq\lam\circ\psi=\psi^*(\lam)\ \text{in}\ \CQ\text{-}\Dist\\
\iff&\forall\lam\in\PB,\phi^*(\lam)\leq\psi^*(\lam)\ \text{in}\ (\PA)_0\\
\iff&\phi^*\leq\psi^*\ \text{in}\ \CQ\text{-}\CCat(\PB,\PA)
\end{align*}
and
\begin{align*}
&F\leq G\ \text{in}\ \CQ\text{-}\CCat(\PB,\PA)\\
\iff&\forall\lam\in\PB,F(\lam)\leq G(\lam)\ \text{in}\ (\PA)_0\\
\iff&\forall\lam\in\PB,F(\lam)\leq G(\lam)\ \text{in}\ \CQ\text{-}\Dist(\bbA,\bbB)\\
\Lra{}&\forall y\in\bbB_0,F\circ\sY_{\bbB}y\leq G\circ\sY_{\bbB}y\ \text{in}\ \CQ\text{-}\Dist(\bbA,\bbB)\\
\iff&\forall y\in\bbB_0,\ulc F\circ\sY_{\bbB}\urc(-,y)\leq\ulc G\circ\sY_{\bbB}\urc(-,y)\ \text{in}\ \CQ\text{-}\Dist(\bbA,\bbB)\\
\iff&\ulc F\circ\sY_{\bbB}\urc\leq\ulc G\circ\sY_{\bbB}\urc\ \text{in}\ \CQ\text{-}\Dist(\bbA,\bbB).
\end{align*}

Secondly,  $F=(\ulc F\circ\sY_{\bbB}\urc)^*$. For all $\lam\in\PB$, since $F$ is a left adjoint in $\CQ$-$\Cat$, by Example \ref{PA_tensor} and Proposition \ref{F_la_ra_condition} we have
\begin{align*}
F(\lam)&=F(\lam\circ\bbB)\\
&=F\Big(\bv_{y\in\bbB_0}\lam(y)\circ\sY_{\bbB}y\Big)\\
&=\bv_{y\in\bbB_0}\lam(y)\circ(F\circ\sY_{\bbB}y)\\
&=\lam\circ\ulc F\circ\sY_{\bbB}\urc\\
&=(\ulc F\circ\sY_{\bbB}\urc)^*(\lam).
\end{align*}

Finally, $\phi=\ulc\phi^*\circ\sY_{\bbB}\urc$. This is obtained in Proposition \ref{phi_star_Yoneda}.

(2) Let $F:\PdA\to\PdB$ be a right adjoint $\CQ$-functor. Similar to (1), one can show that the correspondence $F\mapsto\ulc F\circ\sYd_{\bbA}\urc$ is an inverse of the correspondence $\phi\mapsto\phi^{\dag}$, and thus they are both isomorphisms of posets between $\CQ\text{-}\Dist(\bbA,\bbB)$ and $(\CQ\text{-}\CCat^{\dag})^{\co}(\PdA,\PdB)$.
\end{proof}

Theorem \ref{Kan_distributor_bijection} adds two more isomorphisms of posets to (\ref{Isbell_isomorphism}):
\begin{align*}
\CQ\text{-}\Dist(\bbA,\bbB)&\cong\CQ\text{-}\Cat^{\co}(\bbA,\PdB)\\
&\cong\CQ\text{-}\Cat(\bbB,\PA)\\
&\cong\CQ\text{-}\CCat^{\co}(\PA,\PdB)\\
&\cong\CQ\text{-}\CCat(\PB,\PA)\\
&\cong(\CQ\text{-}\CCat^{\dag})^{\co}(\PdA,\PdB).
\end{align*}

The following corollary reveals the connection between Kan adjunctions and pointwise Kan extensions.

\begin{cor}
Let $F:\bbA\to\bbB$ be a $\CQ$-functor.
\begin{itemize}
\item[\rm (1)] $(F^{\nat})^*=\Lan_{\sY_{\bbA}}(\sY_{\bbB}\circ F)$, $(F_{\nat})^{\dag}=\Ran_{\sYd_{\bbA}}(\sYd_{\bbB}\circ F)$.
\item[\rm (2)] If the pointwise left Kan extension of $F$ along $K:\bbA\to\bbC$ exists, then
               $$\Lan_K F(c)=\bbB\lda(F^{\nat})^*(K_{\nat}(-,c)).$$
\item[\rm (3)] If the pointwise right Kan extension of $F$ along $K:\bbA\to\bbC$ exists, then
               $$\Ran_K F(c)=(F_{\nat})^{\dag}(K^{\nat}(c,-))\rda\bbB.$$
\end{itemize}
\end{cor}

\begin{proof}
(1) Since $\PB$ is complete, the pointwise left Kan extension of $\sY_{\bbB}\circ F:\bbA\to\PB$ along $\sY_{\bbA}:\bbA\to\PA$ exists. For all $\mu\in\PA$,
\begin{align*}
(F^{\nat})^*(\mu)&=\mu\circ F^{\nat}\\
&=\mu\circ F^{\nat}\circ(\sY_{\bbB})^{\nat}\circ(\sY_{\bbB})_{\nat}&\text{(Proposition \ref{fully_faithful_graph_cograph})}\\
&=\mu\circ(\sY_{\bbB}\circ F)^{\nat}\circ(\sY_{\bbB})_{\nat}&\text{(Proposition \ref{graph_cograph_functor})}\\
&={\sup}_{\PB}(\sY_{\bbB}\circ F)^{\ra}(\mu)&\text{(Example \ref{PA_complete})}\\
&={\sup}_{\PB}(\sY_{\bbB}\circ F)^{\ra}((\sY_{\bbA})_{\nat}(-,\mu))&\text{(Yoneda lemma)}\\
&=\Lan_{\sY_{\bbA}}(\sY_{\bbB}\circ F)(\mu).&\text{(Proposition \ref{Kan_limit})}
\end{align*}
Thus $(F^{\nat})^*=\Lan_{\sY_{\bbA}}(\sY_{\bbB}\circ F)$. That $(F_{\nat})^{\dag}=\Ran_{\sYd_{\bbA}}(\sYd_{\bbB}\circ F)$ can be proved similarly.

(2) and (3) are immediate consequences of Proposition \ref{Kan_limit}.
\end{proof}

\begin{rem}
Consider the two-element Boolean algebra ${\bf 2}$ as a quantaloid. Then every set can be regarded as a discrete ${\bf 2}$-category. Given sets $A$ and $B$, a ${\bf 2}$-distributor $F:A\oto B$ is essentially a relation from $A$ to $B$, or a set-valued map $A\to {\bf 2}^B$. If we write $F^{\op}$ for the dual relation of $F$, then both $F_*$ and  $(F^{\op})^*$ are maps from ${\bf 2}^B$ to ${\bf 2}^A$. Explicitly, for each $V\subseteq B$,
$$F_*(V)=\{x\in A\mid F(x)\subseteq V\}\quad\text{and}\quad(F^{\op})^*(V)=\{x\in A\mid F(x)\cap V\not=\emptyset\}. $$

If both $A$ and $B$ are topological spaces, then the upper and lower semi-continuity of $F$  (as a set-valued map) \cite{Berge1963} can be phrased as follows:  $F$ is upper (resp. lower) semi-continuous  if $F_*(V)$ (resp. $(F^{\op})^*(V)$) is open in $A$ whenever $V$ is open in $B$. In particular, if $F$ is the graph of some map $f:A\to B$, then $(F^{\op})^*(V)=F_*(V)=f^{-1}(V)$ for all $V\subseteq B$, hence $f$ is continuous iff $F$ is lower semi-continuous iff $F$ is upper semi-continuous \cite{Berge1963}.
\end{rem}

The following corollary shows that for a fully faithful $\CQ$-functor $F:\bbA\to\bbB$, both $(F^{\nat})^*$ and $(F_{\nat})_*$ can be regarded as extensions of $F$ \cite{Lawvere1973}.

\begin{cor}
\label{left_right_Kan_equivalent} If $F:\bbA\to\bbB$ is a fully faithful $\CQ$-functor, then for all $\mu\in\PA$, it holds that $(F^{\nat})^*(\mu)\circ F_{\nat}=\mu$ and $(F_{\nat})_*(\mu)\circ F_{\nat}=\mu$.
\end{cor}

\begin{proof}
The first equality is a reformulation of Proposition \ref{fully_faithful_graph_cograph}(1). For the second equality,
\begin{align*}
(F_{\nat})_*(\mu)\circ F_{\nat}&=(\mu\lda F_{\nat})\circ F_{\nat}\\
&=\mu\lda(F^{\nat}\circ F_{\nat})&(\text{Proposition \ref{adjoint_arrow_calculation}(4)})\\
&=\mu\lda\bbA&(\text{Proposition \ref{fully_faithful_graph_cograph}(1)})\\
&=\mu.
\end{align*}
This completes the proof.
\end{proof}

\section{Functoriality of the Kan adjunction}

In this section, we show that the construction of Kan adjunctions $\phi^*\dv\phi_*$ is contravariant functorial from the category $\CQ$-$\Info$ to $\CQ$-$\Cls$, and thus to $\CQ$-$\CCat$.

Since $\phi_*\circ\phi^*:\PB\to\PB$ is a $\CQ$-closure operator for each $\CQ$-distributor $\phi:\bbA\oto\bbB$, it follows that $(\bbB,\phi_*\circ\phi^*)$ is a $\CQ$-closure space.

\begin{prop} \label{F_G_infomorhpism_F_phistar_continuous}
Let $(F,G):(\phi:\bbA\oto\bbB)\to(\psi:\bbA'\oto\bbB')$ be an infomorphism. Then $G:(\bbB',\psi_*\circ\psi^*)\to(\bbB,\phi_*\circ\phi^*)$ is a continuous $\CQ$-functor.
\end{prop}

\begin{proof}
Consider the following diagram:
$$\bfig
\square|alrb|[\PB'`\PA'`\PB`\PA;\psi^*`G^{\ra}`F^{\la}`\phi^*]
\square(500,0)/>``>`>/[\PA'`\PB'`\PA`\PB;\psi_*``G^{\ra}`\phi_*]
\efig $$
We must prove $G^{\ra}\circ\psi_*\circ\psi^*\leq\phi_*\circ\phi^*\circ G^{\ra}$. To this end, it suffices to check that
\begin{itemize}
\item[\rm (a)] the left square commutes if and only if $(F,G):\phi\to\psi$ is an infomorphism; and
\item[\rm (b)] if $G^{\nat}\circ\phi\leq\psi\circ F_{\nat}$, then $G^{\ra}\circ\psi_*\leq\phi_*\circ F^{\la}$.
\end{itemize}

For (a), suppose $F^{\la}\circ\psi^*=\phi^*\circ G^{\ra}$, then for all $y'\in\bbB'_0$,
\begin{align*}
G^{\nat}(-,y')\circ\phi&=\bbB'(-,y')\circ G^{\nat}\circ\phi\\
&=\phi^*(\sY_{\bbB'}y'\circ G^{\nat})&(\text{Equation (\ref{phidag_def})})\\
&=\phi^*(G^{\ra}\circ\sY_{\bbB'}y')&(\text{Definition \ref{direct_inverse_image_def}})\\
&=F^{\la}(\psi^*\circ\sY_{\bbB'}y')\\
&=F^{\la}(\psi(-,y'))&\text{(Proposition \ref{phi_star_Yoneda})}\\
&=\psi(-,y')\circ F_{\nat}.&(\text{Definition \ref{direct_inverse_image_def}})
\end{align*}

Conversely, if $(F,G):\phi\to\psi$ is an infomorphism, then for all $\lam'\in\PB'$,
\begin{align*}
F^{\la}\circ\psi^*(\lam')&=\lam'\circ\psi\circ F_{\nat}&(\text{Equation (\ref{phidag_def})})\\
&=\lam'\circ G^{\nat}\circ\phi&(\text{Definition \ref{infomorphism_def}})\\
&=\phi^*\circ G^{\ra}(\lam').&(\text{Equation (\ref{phidag_def})})\\
\end{align*}

For (b), suppose $G^{\nat}\circ\phi\leq\psi\circ F_{\nat}$, then for all $\mu'\in\PA'$,
\begin{align*}
G^{\ra}\circ\psi_*(\mu')&=(\mu'\lda\psi)\circ G^{\nat}&(\text{Equation (\ref{phistar_def})})\\
&\leq(\mu'\lda\psi)\circ((\psi\circ F_{\nat})\lda\phi)\\
&\leq((\mu'\lda\psi)\circ\psi\circ F_{\nat})\lda\phi\\
&\leq(\mu'\circ F_{\nat})\lda\phi\\
&=\phi_*\circ F^{\la}(\mu').&(\text{Equation (\ref{phistar_def})})
\end{align*}
This completes the proof.
\end{proof}

By virtue of Proposition \ref{F_G_infomorhpism_F_phistar_continuous} we obtain a functor $\CV:(\CQ\text{-}\Info)^{\op}\to\CQ\text{-}\Cls$ that sends an infomorphism
$$(F,G):(\phi:\bbA\oto\bbB)\to(\psi:\bbA'\oto\bbB')$$
to a continuous $\CQ$-functor
$$G:(\bbB',\psi_*\circ\psi^*)\to(\bbB,\phi_*\circ\phi^*).$$

The composition of
$$\CV:(\CQ\text{-}\Info)^{\op}\to\CQ\text{-}\Cls$$
and
$$\CT:\CQ\text{-}\Cls\to\CQ\text{-}\CCat$$
gives a  functor
\begin{equation} \label{K_def}
\CK=\CT\circ\CV:(\CQ\text{-}\Info)^{\op}\to\CQ\text{-}\CCat
\end{equation}
$$\bfig
\Vtriangle/->`->`@{<-}@<2.5pt>/[(\CQ\text{-}\Info)^{\op}`\CQ\text{-}\CCat`\CQ\text{-}\Cls;\CK`\CV`\CT]
\morphism(500,0)/@{<-}@<2.5pt>/<500,500>[\CQ\text{-}\Cls`\CQ\text{-}\CCat;\CD]
\efig$$
that sends each $\CQ$-distributor $\phi:\bbA\oto\bbB$ to the complete $\CQ$-category $\CK(\phi)=\phi_*\circ\phi^*(\PB)$.

The following conclusion asserts that the free cocompletion functor of $\CQ$-categories  factors through  $\CK$.

\begin{prop} \label{K_Yoneda_PA}
If $F:\bbA\to\bbB$ is a fully faithful $\CQ$-functor, then $\CK(F^{\nat})=\PA$. In particular, the diagram
$$\bfig
\qtriangle<700,500>[\CQ\text{-}\Cat`(\CQ\text{-}\Info)^{\op}`\CQ\text{-}\CCat;\BY^{\dag}`\CP`\CK]
\efig$$
commutes.
\end{prop}

\begin{proof}
In order to see that $\CK(F^{\nat})=(F^{\nat})_*\circ(F^{\nat})^*(\PA)=\PA$, it suffices to check that $(F^{\nat})_*\circ(F^{\nat})^*(\mu)=\mu$ for all $\mu\in\PA$. Indeed,
\begin{align*}
(F^{\nat})_*\circ(F^{\nat})^*(\mu)&=(F_{\nat})^*\circ(F^{\nat})^*(\mu)&(\text{Theorem \ref{why_kan}})\\
&=(F^{\nat}\circ F_{\nat})^*(\mu)\\
&=\bbA^*(\mu)&(\text{Proposition \ref{fully_faithful_graph_cograph}(1)})\\
&=\mu.
\end{align*}

Furthermore, it is easy to verify that $\CK\circ\BY^{\dag}(G)=G^\ra=\CP(G)$ for each $\CQ$-functor $G:\bbA\to\bbB$. Thus, the conclusion follows.
\end{proof}

\section{When $\CQ$ is a Girard quantaloid} \label{When_Q_Girard}

Theorem \ref{complete_category_fca} shows that every skeletal complete $\CQ$-category is of the form $\CM(\phi)$.  It is natural to ask whether every skeletal complete $\CQ$-category can be written of the form $\CK(\phi)$ for some $\CQ$-distributor $\phi$. A little surprisingly, this is not true in general. This fact was pointed out in \cite{Lai2009} in the case that $\CQ$ is a unital commutative quantale. However, the answer is positive when $\CQ$ is a Girard quantaloid.

Let $\CQ$  be a Girard quantaloid with a cyclic dualizing family
$$\FD=\{d_X:X\to X\mid X\in\CQ_0\}.$$
For all $f\in\CQ(X,Y)$, let
$$\neg f=d_X\lda f=f\rda d_Y:Y\to X.$$
Then $\neg\neg f=f$ since $\FD$ is a dualizing family.
For each $\CQ$-category $\bbA$, set
$$(\neg\bbA)(y,x)=\neg\bbA(x,y)$$
for all $x,y\in\bbA_0$. It is easy to verify that $\neg\bbA:\bbA\oto\bbA$ is a $\CQ$-distributor.

\begin{prop} {\rm\cite{Rosenthal1996}} \label{DistQ_Girard}
If $\CQ$ is a Girard quantaloid, then $\CQ$-$\Dist$ is a Girard quantaloid.
\end{prop}

\begin{proof}
If $\FD=\{d_X:X\to X\mid X\in\CQ_0\}$ is a cyclic dualizing family of $\CQ$, we show that
$$\FD'=\{\neg\bbA:\bbA\oto\bbA\mid\bbA\text{ is a }\CQ\text{-category}\}$$
is a cyclic dualizing family of $\CQ\text{-}\Dist$. Indeed, for each $\CQ$-distributor $\phi:\bbA\oto\bbB$,
\begin{align*}
\neg\bbA\lda\phi&=\bw_{x\in\bbA_0}(\bbA(-,x)\rda d_{tx})\lda\phi(x,-)\\
&=\bw_{x\in\bbA_0}\bbA(-,x)\rda(d_{tx}\lda\phi(x,-))\\
&=\bbA\rda\neg\phi\\
&=\neg\phi\lda\bbB\\
&=\bw_{y\in\bbB_0}(\phi(-,y)\rda d_{ty})\lda\bbB(y,-)\\
&=\bw_{y\in\bbB_0}\phi(-,y)\rda(d_{ty}\lda\bbB(y,-))\\
&=\phi\rda\neg\bbB,
\end{align*}
and it follows that
\begin{align*}
(\neg\bbA\lda\phi)\rda\neg\bbA&=(\bbA\rda\neg\phi)\rda\neg\bbA\\
&=\neg\phi\rda\neg\bbA\\
&=\neg\neg\phi\lda\bbA\\
&=\phi,
\end{align*}
as desired.
\end{proof}

Therefore, by assigning $\neg\phi=\neg\bbA\lda\phi=\phi\rda\neg\bbB$ for each $\CQ$-distributor $\phi:\bbA\oto\bbB$, we obtain a  functor $\neg:\CQ\text{-}\Info\to(\CQ\text{-}\Info)^{\op}$ that sends an infomorphism
$$(F,G):(\phi:\bbA\oto\bbB)\to(\psi:\bbA'\oto\bbB')$$
to
$$(G,F):(\neg\psi:\bbB'\oto\bbA')\to(\neg\phi:\bbB\oto\bbA).$$

It is clear that $\neg\circ\neg={\bf id}_{\CQ\text{-}\Info}$. We leave it to the reader to check that $(\neg\phi)(y,x)=\neg\phi(x,y)$ for any $\CQ$-distributor $\phi:\bbA\oto\bbB$ and $x\in\bbA_0,y\in\bbB_0$.

\begin{lem} \label{V_G_id}
Suppose $\CQ$ is a Girard quantaloid. Then for any $\CQ$-distributor $\phi:\bbA\oto\bbB$, it holds that $\phi^*=\neg\circ(\neg\phi)_{\ua}$ and $\phi_*=(\neg\phi)^{\da}\circ\neg$.
\end{lem}

\begin{proof}
For all $\lam\in\PB$ and  $\mu\in\PA$, we have
\begin{align*}
\phi^*(\lam)&=\lam\circ\phi\\
&=\lam\circ(\neg\phi\rda\neg\bbA)&(\text{Proposition \ref{DistQ_Girard}})\\
&=(\neg\phi\lda\lam)\rda\neg\bbA&(\text{Proposition \ref{Girard_quantaloid_properties}(3)})\\
&=\neg\circ(\neg\phi)_{\ua}(\lam)
\end{align*}
and
\begin{align*}
\phi_*(\mu)&=\mu\lda\phi\\
&=\mu\lda(\neg\phi\rda\neg\bbA)&(\text{Proposition \ref{DistQ_Girard}})\\
&=(\neg\bbA\lda\mu)\rda\neg\phi&(\text{Proposition \ref{Girard_quantaloid_properties}(5)})\\
&=\neg\mu\rda\neg\phi&(\text{Proposition \ref{DistQ_Girard}})\\
&=(\neg\phi)^{\da}\circ\neg\mu.
\end{align*}
The conclusion thus follows.
\end{proof}

\begin{prop} \label{G_V_adjunction}
Suppose $\CQ$ is a Girard quantaloid. Then  $\CV=\CU\circ\neg$ and it has a left adjoint right inverse given by
$$\CG=\neg\circ\CF:\CQ\text{-}\Cls\to(\CQ\text{-}\Info)^{\op}.$$
Therefore, every skeletal complete $\CQ$-category is isomorphic to $\CK(\phi)$ for some $\CQ$-distributor $\phi$.
\end{prop}

\begin{proof}
This is an immediate consequence of Theorem \ref{F_U_adjunction} and Lemma \ref{V_G_id}.
\end{proof}

\section{Towards a characterization of $\CK(\phi)$}

In the case that $\CQ$ is a Girard quantaloid, the results in Section \ref{When_Q_Girard} asserts that
$$\CK(\phi)=\CM(\neg\phi)$$
for each $\CQ$-distributor $\phi:\bbA\oto\bbB$. Thus, the characterizations of $\CK(\phi)$ can be obtained directly from that of $\CM(\phi)$ presented in Section \ref{Characterizations_of_Mphi}.

For a general quantaloid $\CQ$, characterizing $\CK(\phi)$ is more complicated. In this section, we provide a characterization of $\CK(\phi)$ parallel to Proposition \ref{closure_system_fca}. However, we fail to find a counterpart of Theorem \ref{complte_category_sup_dense} for $\CK(\phi)$.

Given a $\CQ$-distributor $\phi:\bbA\oto\bbB$, let  $\CK_{\phi}(\bbA,\bbB)$ denote the set of pairs $(\mu,\lam)\in\PA\times\PB$ such that $\lam=\phi_*(\mu)$ and $\mu=\phi^*(\lam)$. $\CK_{\phi}(\bbA,\bbB)$ becomes a $\CQ$-typed set if we assign $t(\mu,\lam)=t\mu=t\lam$. For $(\mu_1,\lam_1),(\mu_2,\lam_2)\in\CK_{\phi}(\bbA,\bbB)$, let
\begin{equation} \label{P_A_B_phi_order}
\CK_{\phi}(\bbA,\bbB)((\mu_1,\lam_1),(\mu_2,\lam_2))=\PA(\mu_1,\mu_2)=\PB(\lam_1,\lam_2),
\end{equation}
Then $\CK_{\phi}(\bbA,\bbB)$ becomes a $\CQ$-category.

The projection
$$\pi_2:\CK_{\phi}(\bbA,\bbB)\to\PB,\quad (\mu,\lam)\mapsto\lam$$
is clearly a fully faithful $\CQ$-functor. Since the image of $\pi_2$ is exactly the set of fixed points of the $\CQ$-closure operator $\phi_*\circ\phi^*:\PB\to\PB$, we obtain that $\CK_{\phi}(\bbA,\bbB)$ is isomorphic to the complete $\CQ$-category $\CM(\phi)=\phi_*\circ\phi^*(\PB)$.

Similarly, the projection
$$\pi_1:\CK_{\phi}(\bbA,\bbB)\to\PA,\quad (\mu,\lam)\mapsto\mu$$
is also a fully faithful $\CQ$-functor and the image of $\pi_1$ is exactly the set of fixed points of the $\CQ$-interior operator $\phi^*\circ\phi_*:\PA\to\PA$. Hence $\CK_{\phi}(\bbA,\bbB)$ is also isomorphic to the complete $\CQ$-category $\phi^*\circ\phi_*(\PA)$, which is a $\CQ$-interior system of the skeletal complete $\CQ$-category $\PA$.

Equation (\ref{P_A_B_phi_order}) shows that
$$\phi^*:\phi_*\circ\phi^*(\PB)\to\phi^*\circ\phi_*(\PA)$$
and
$$\phi_*:\phi^*\circ\phi_*(\PA)\to\phi_*\circ\phi^*(\PB)$$
are inverse to each other. Therefore, $\CK(\phi)(=\phi_*\circ\phi^*(\PB))$,  $\phi^*\circ\phi_*(\PA)$  and $\CK_{\phi}(\bbA,\bbB)$ are isomorphic to each other.

We present the following characterization of the complete $\CQ$-category $\CK(\phi)$, and we remind the readers to compare it with Proposition \ref{closure_system_fca}.

\begin{prop} \label{interior_system_rst}
Let $\bbX$ be a skeletal complete $\CQ$-category. The following conditions are equivalent:
\begin{itemize}
\item[\rm (1)] $\bbX$ is isomorphic to $\CK(\phi)$ for some $\CQ$-distributor $\phi:\bbA\oto\bbB$.
\item[\rm (2)] $\bbX$ is isomorphic to a $\CQ$-interior system of $\PA$ for some $\CQ$-category $\bbA$.
\item[\rm (3)] $\bbX$ is isomorphic to a subobject of $\PA$ for some $\CQ$-category $\bbA$ in the category $\CQ$-$\CCat$.
\end{itemize}
\end{prop}

\begin{proof}
(1)${}\Lra{}$(2) and (2)$\iff$(3): Trivial.

(2)${}\Lra{}$(1): Let $F:\PA\to\PA$ be a $\CQ$-interior operator and $\bbB=F(\PA)$. Define a $\CQ$-distributor $\zeta_F:\bbA\oto\bbB$ by
$$\zeta_F(x,\mu)=\mu(x),$$
which is also obtained by restricting the domain and the codomain of the $\CQ$-distributor (\ref{lam_circ_mu}). We prove $F=(\zeta_F)^*\circ(\zeta_F)_*$, and consequently $\CK(\zeta_F)$ is isomorphic to $F(\PA)$.

For all $\mu\in\PA$ and $x\in\bbA_0$, let $\lam\in\bbB_0$, then
$$\PA(\lam,\mu)\leq\PA(F(\lam),F(\mu))=\PA(\lam,F(\mu))\leq F(\mu)(x)\lda\lam(x),$$
and consequently
$$F(\mu)(x)\geq\PA(\lam,\mu)\circ\lam(x)=\Big(\bw_{a\in\bbA_0}\mu(a)\lda\lam(a)\Big)\circ\lam(x).$$
If $f:tx\to t\mu$ satisfies
$$f\geq\Big(\bw_{a\in\bbA_0}\mu(a)\lda\lam(a)\Big)\circ\lam(x)$$
for all $\lam\in\bbB_0$, then
$$f\geq\Big(\bw_{a\in\bbA_0}\mu(a)\lda F(\mu)(a)\Big)\circ F(\mu)(x)\geq 1_{t\mu}\circ F(\mu)(x)=F(\mu)(x),$$
hence
\begin{align*}
F(\mu)(x)&=\bv_{\lam\in\bbB_0}\Big(\bw_{a\in\bbA_0}\mu(a)\lda\lam(a)\Big)\circ\lam(x)\\
&=\bv_{\lam\in\bbB_0}\Big(\bw_{a\in\bbA_0}\mu(a)\lda\zeta_F(a,\lam)\Big)\circ\zeta_F(x,\lam)\\
&=(\zeta_F)^*\circ(\zeta_F)_*(\mu)(x),
\end{align*}
as required.
\end{proof}

We wish to give a direct characterization of $\CK_{\phi}(\bbA,\bbB)$ like Theorem \ref{complte_category_sup_dense} for $\CM_{\phi}(\bbA,\bbB)$. However, although there exists a $\sup$-dense $\CQ$-functor $G:\bbB\to\CK_{\phi}(\bbA,\bbB)$, the following Proposition \ref{Kan_adjucntion_inf_dense} states that there is not an $\inf$-dense $\CQ$-functor $F:\bbA\to\CK_{\phi}(\bbA,\bbB)$ in general.

\begin{prop}
There exists a $\sup$-dense $\CQ$-functor $G:\bbB\to\CK_{\phi}(\bbA,\bbB)$.
\end{prop}

\begin{proof}
Let $\bbX=\CK_{\phi}(\bbA,\bbB)$. Define a $\CQ$-functor $G:\bbB\to\bbX$ by
$$Gb=(\olphi b,\phi_*\circ\olphi b),$$
then $G$ is well defined by Equation (\ref{phiF_star_fixed}). We show that $G$ is $\sup$-dense. For all $(\mu,\lam),(\mu',\lam')\in\bbX_0$,
\begin{align*}
\bbX((\mu,\lam),(\mu',\lam'))&=\mu'\lda\mu\\
&=\mu'\lda\phi^*(\lam)\\ &=\mu'\lda(\lam\circ\phi)\\
&=(\mu'\lda\phi)\lda\lam\\
&=\PA(\olphi-,\mu')\lda\lam\\
&=\bbX(G-,(\mu',\lam'))\lda\lam&\text{(Equation (\ref{P_A_B_phi_order}))}\\
&=G_{\nat}(-,(\mu',\lam'))\lda\lam,
\end{align*}
thus $(\mu,\lam)=\colim_{\lam}G={\sup}_{\bbX}\circ G^{\ra}(\lam)$, as desired.
\end{proof}

\begin{prop} \label{Kan_adjucntion_inf_dense}
Let $\CQ$ be a quantaloid with the identity arrow $1_X:X\to X$ being the top element of $\CQ(X,X)$ for all $X\in\CQ_0$ and $\FD=\{\bot_{X,X}\mid X\in\CQ_0\}$ a cyclic family. Then the following conditions are equivalent:
\begin{itemize}
\item[\rm (1)] $\FD=\{\bot_{X,X}\mid X\in\CQ_0\}$ is a dualizing family, hence $\CQ$ is a Girard quantaloid.
\item[\rm (2)] There is an $\inf$-dense $\CQ$-functor $F:\bbA\to\CK_{\phi}(\bbA,\bbB)$ for each $\CQ$-distributor $\phi:\bbA\oto\bbB$.
\end{itemize}
\end{prop}

\begin{proof}
(1)${}\Lra{}$(2): For each $\CQ$-distributor $\phi:\bbA\oto\bbB$, by Proposition \ref{G_V_adjunction} we have $\CK_{\phi}(\bbA,\bbB)=\CM_{\neg\phi}(\bbB,\bbA)$, thus the conclusion follows from Theorem \ref{complte_category_sup_dense}.

(2)${}\Lra{}$(1): For each $X\in\CQ_0$, consider the $\CQ$-distributor $\phi:*_X\oto *_X$ given by $\phi(*,*)=1_X$. It is easy to see that $\CK(\phi)=\CP X$, thus there is an $\inf$-dense $\CQ$-functor $F:*_X\to\CP X$. This means that for all $f\in\CP X$, there is some $g\in\CPd(*_X)=\CPd X$ such that $f=\inf_{\PX}F^{\nra}(g)$, i.e.,
\begin{align*}
\CP X(-,f)&=F^{\nra}(g)\rda\CP X\\
&=(\CP X(F*,-)\circ g)\rda\CP X\\
&=g\rda\CP X(-,F*)\\
&=\PX(-,g\rda F*),
\end{align*}
and consequently $f=g\rda F*$. Hence $(-)\rda F*:\CPd X\to\CP X$ is a surjective $\CQ$-functor. Since $(-)\rda F*:(\CPd X)_X\to(\CP X)_X$ is order-preserving, we obtain that $\bot_{X,X}=1_X\rda F*=F*$. This means that the $\CQ$-functor $(-)\rda \bot_{X,X}:\CPd X\to\CP X$ is surjective, hence for all $f\in\CQ(X,Y)$, there is some $g\in\CQ(Y,X)$ such that $f=g\rda \bot_{X,X}$. Note that for all $f\in\CP X$ and $g\in\CPd X$,
$$\CPd X(\bot_{X,X}\lda f,g)=g\rda(\bot_{X,X}\lda f)=(g\rda \bot_{X,X})\lda f=\PX(f,g\rda \bot_{X,X}),$$
thus $\bot_{X,X}\lda(-)\dv(-)\rda \bot_{X,X}:\CP X\rhu\CPd X$, and therefore
$$(\bot_{X,X}\lda f)\rda \bot_{X,X}=(\bot_{X,X}\lda(g\rda \bot_{X,X}))\rda \bot_{X,X}=g\rda \bot_{X,X}=f,$$
as desired.
\end{proof}

\chapter{Applications in fuzzy set theory} \label{Applications}

Isbell adjunctions and Kan adjunctions in $\CQ$-categories correspond closely to the theories of formal concept analysis and rough set theory in computer science. Our results obtained in Chapter \ref{Isbell_adjunction} and \ref{Kan_adjunction} can be applied to two special quantaloids: the one-object quantaloid (i.e., a unital quantale), and the quantaloid generated by a divisible unital quantale. Through this way we develop the theories of formal concept analysis and rough set theory on fuzzy relations between fuzzy sets.

\section{Preordered sets valued in a unital quantale} \label{Preordered_sets}

Recall that a \emph{unital quantale} is a one-object quantaloid. Explicitly, a unital quantale is a complete lattice $Q$ equipped with a binary operation $\&:Q\times Q\to Q$ such that
\begin{itemize}
\item[\rm (i)] $(Q,\&)$ is a monoid with a unit $I$;
\item[\rm (ii)] $a\&(\bv b_i)=\bv(a\& b_i)$ and $(\bv b_i)\& a=\bv(b_i\& a)$ for all $a,b_i\in Q$.
\end{itemize}

In a unital quantale $Q$, the left implication $/$ and the right implication $\rd$ are two binary operations on $Q$ determined by the adjoint property
$$c\leq a\rd b\iff a\& c\leq b\iff a\leq b/c$$
for all $a,b,c\in Q$.

A unital quantale $Q$ is \emph{commutative} if $(Q,\&)$ is a commutative monoid. A commutative unital quantale $Q$ is called a \emph{complete residuated lattice} if the unit $I$ is the largest element of $Q$.

\begin{exmp} \label{quantale_example}
We list some basic examples of unital quantales.
\begin{itemize}
\item[\rm (1)] Each frame is a commutative unital quantale.
\item[\rm (2)] Each complete BL-algebra is a commutative unital quantale. In particular, each continuous t-norm on the unit interval $[0,1]$ is a commutative unital quantale.
\item[\rm (3)] $([0,\infty]^{\op},+)$ is a commutative unital quantale in which $b/a=a\rd b=\max\{0,b-a\}$.
\end{itemize}
\end{exmp}

Consider a unital quantale $Q$ as a one-object quantaloid, then each (crisp) set $A$ can be viewed as a discrete $Q$-category (see Example \ref{Q_category_example}(2)), in which
$$A(x,y)=\begin{cases}
1, & x=y;\\
0, & x\neq y.
\end{cases}$$

\begin{defn}
A \emph{$Q$-relation} (or \emph{fuzzy relation}) $\phi:A\oto B$ between (crisp) sets $A$ and $B$ is a $Q$-distributor between discrete $Q$-categories $A$ and $B$ .
\end{defn}

A $Q$-relation $\phi:A\oto B$ is exactly a function
$$\phi:A\times B\to Q.$$
Each $Q$-relation $\phi:A\oto B$ has a dual $Q$-relation $\phi^{\op}:B\oto A$ given by $\phi^{\op}(y,x)=\phi(x,y)$ for all $x\in A$ and $y\in B$.

\begin{defn}
Let $\phi:A\oto A$ be a $Q$-relation on a (crisp) set $A$.
\begin{itemize}
\item[\rm (1)] $\phi$ is \emph{reflexive} if $A\leq \phi$.
\item[\rm (2)] $\phi$ is \emph{transitive} if $\phi\&\phi\leq \phi$.
\item[\rm (3)] $\phi$ is \emph{symmetric} if $\phi=\phi^{\op}$.
\item[\rm (4)] $\phi$ is \emph{separated} if $\phi(x,x)\leq \phi(x,y)$ and $\phi(y,y)\leq \phi(y,x)$ implies $x=y$.
\end{itemize}
\end{defn}

\begin{defn}
A \emph{$Q$-preorder} (or \emph{fuzzy preorder}) on a (crisp) set $A$ is a reflexive and transitive $Q$-relation $\bbA:A\oto A$. The pair $(A,\bbA)$ is called a $Q$-preordered set.
\end{defn}

Explicitly, a $Q$-preorder on a (crisp) set $A$ is a function $\bbA:A\times A\to Q$ such that
\begin{itemize}
\item[{\rm (1)}] $I\leq \bbA(x,x)$ for each $x\in A$ (reflexivity);
\item[{\rm (2)}] $\bbA(y,z)\& \bbA(x,y)\leq \bbA(x,z)$ for all $x,y,z\in A$ (transitivity).
\end{itemize}

A $Q$-preordered set $(A,\bbA)$ is \emph{separated} if the $Q$-relation $\bbA$ is separated.

\begin{prop}
A $Q$-preordered set $(A,\bbA)$ is separated if and only if it satisfies
\begin{itemize}
\item[{\rm (3)}] $I\leq \bbA(x,y)$ and $I\leq \bbA(y,x)$ implies $x=y$ (anti-symmetry).
\end{itemize}
\end{prop}

\begin{proof}
For the non-trivial direction, suppose that $I\leq \bbA(x,y)$ and $I\leq \bbA(y,x)$, we have
$$\bbA(x,x)\leq\bbA(x,y)\rd\bbA(x,y)\leq I\rd\bbA(x,y)=\bbA(x,y).$$
and
$$\bbA(y,y)\leq\bbA(y,x)\rd\bbA(y,x)\leq I\rd\bbA(y,x)=\bbA(y,x).$$
Thus $x=y$.
\end{proof}

In a $Q$-preordered set $(A,\bbA)$, the value $\bbA(x,y)$ can be interpreted as the \emph{degree} to which $x$ is less than or equal to $y$. Compared with Definition \ref{Q_category}, it is easy to see that $Q$-preordered sets are exactly categories enriched over the unital quantale $Q$, and separated $Q$-preordered sets are skeletal $Q$-categories. We abbreviate the pair $(A,\bbA)$ to $\bbA$ and write $\bbA_0=A$ if there is no confusion.

\begin{defn}
An order-preserving map $F:\bbA\to\bbB$ between $Q$-preordered sets is a $Q$-functor between $Q$-categories $\bbA$ and $\bbB$.
\end{defn}

Given a $Q$-preordered set $\bbA$, a \emph{lower set} of $\bbA$ is a contravariant presheaf $\mu\in\PA$, and an \emph{upper set} of $\bbA$ is a covariant presheaf $\lam\in\PdA$. The $Q$-preordered sets $\PA$ and $\PdA$ are called respectively the \emph{$Q$-powerset} and the \emph{dual $Q$-powerset} of $\bbA$.

In particular, consider a (crisp) set $A$ as a discrete $Q$-category, the $Q$-category $\CP A$ of contravariant presheaves on $A$ is the \emph{$Q$-powerset} of $A$. Explicitly, the objects of $\CP A$ are the maps $A\to Q$, i.e., $(\CP A)_0=Q^A$, and the $Q$-preorder on $\CP A$ is given by
$$\CP A(\mu,\lam)=\bw_{x\in A}\lam(x)/\mu(x).$$
Dually, the $Q$-category $\CPd A$ of covariant presheaves on $A$ is the \emph{dual $Q$-powerset} of $A$. The objects of $\CPd A$ are also the maps $A\to Q$, but the $Q$-preorder on $\CPd A$ is given by
$$\CPd A(\mu,\lam)=\bw_{x\in A}\lam(x)\rd\mu(x).$$

A separated $Q$-preordered set $\bbA$ is a \emph{complete $Q$-lattice} if it is a skeletal complete $Q$-category. For each $Q$-preordered set $\bbA$, $\PA$ and $\PdA$ are both complete $Q$-lattices.

\section{Preordered fuzzy sets valued in a divisible unital quantale} \label{Preordered_fuzzy_sets}

In order to derive the theory of preordered fuzzy sets (not only preordered crisp sets), we add one more requirement to the unital quantale $Q$, i.e., being \emph{divisible}. This assumption makes sense because all the unital quantales given in Example \ref{quantale_example} are divisible, which cover most of the important truth tables in fuzzy set theory.

A unital quantale $Q$ is \emph{divisible} if it satisfies one of the equivalent conditions in the following Proposition \ref{divisible_condition}.

\begin{prop} \label{divisible_condition} {\rm\cite{Pu2012,Tao2012}}
For a unital quantale $Q$, the following conditions are equivalent:
\begin{itemize}
\item[\rm (1)] $\forall a,b\in Q$, $a\leq b$ implies $x\& b=a=b\& y$ for some $x,y\in Q$.
\item[\rm (2)] $\forall a,b\in Q$, $a\leq b$ implies $a=b\& (b\rd a)=(a/b)\& b $.
\item[\rm (3)] $\forall a,b,c\in Q$, $a,b\leq c$ implies $a\& (c\rd b)=(a/c)\& b$.
\item[\rm (4)] $\forall a,b\in Q$, $(b/a)\& a=a\wedge b=a\& (a\rd b)$.
\end{itemize}
In this case, the unit $I$ must be the top element $1$ in $Q$.
\end{prop}

\begin{prop} \label{divisible_quantale_quantaloid} {\rm\cite{Hohle2011,Pu2012}}
Each divisible unital quantale $Q$ gives rise to a quantaloid $\CQ$ that consists of the following data:
\begin{itemize}
\item[\rm (1)] objects: elements $X,Y,Z,\cdots$ in $Q$;
\item[\rm (2)] morphisms: $\CQ(X,Y)=\{\al\in Q:\al\leq X\wedge Y\}$;
\item[\rm (3)] composition: $\be\circ\al=\be\&(Y\backslash\al)=(\be/Y)\&\al$ for all $\al\in\CQ(X,Y)$ and $\be\in\CQ(Y,Z)$;
\item[\rm (4)] implication: for all $\al\in\CQ(X,Y)$, $\be\in\CQ(Y,Z)$ and $\ga\in\CQ(X,Z)$,
    $$\ga\lda\al=Y\wedge Z\wedge(\ga/(Y\rd\al))\quad\text{and}\quad\be\rda\ga=X\wedge Y\wedge((\be/Y)\rd\ga);$$
\item[\rm (5)] the unit $1_X$ of $\CQ(X,X)$ is $X$;
\item[\rm (6)] the partial order on $\CQ(X,Y)$ is inherited from $Q$.
\end{itemize}
\end{prop}

In Example \ref{Girard_quantaloid_example}(3) we have seen that if $Q$ is a Boolean algebra, then the quantaloid $\CQ$ induced by Proposition \ref{divisible_quantale_quantaloid} is a Girard quantaloid. However, we point out that $\CQ$ might not be a Girard quantaloid if $Q$ is a general Girard quantale. For example, the quantaloid induced by a {\L}ukasiewicz $t$-norm in this way is not a Girard quantaloid.

In this section, $Q$ is always assumed to be a divisible unital quantale and $\CQ$ the associated quantaloid given by Proposition \ref{divisible_quantale_quantaloid} if not otherwise specified.

\begin{defn}
A \emph{$Q$-subset} (or \emph{fuzzy set}) is a $\CQ$-typed set.
\end{defn}

Explicitly, a $Q$-subset is a pair $(\sA_0,\sA)$, where $\sA_0$ is the underlying (crisp) set and $\sA:\sA_0\to Q$ is the type map. For each $x\in\sA_0$, the type $\sA x$ is interpreted as the membership degree of $x$ in $\sA_0$. We abbreviate the pair $(\sA_0,\sA)$ to $\sA$ if no confusion would arise. $Q$-subsets and type-preserving functions constitute the slice category $\Set\da Q$.

An element $x$ in a $Q$-subset $\sA$ is \emph{global} if $\sA x=1$. A $Q$-subset is \emph{global} if every element in it is global. A global $Q$-subset is exactly a crisp set.

By Example \ref{Q_category_example}(2), each $Q$-subset $\sA$ can be viewed as a discrete $\CQ$-category, in which
$$\sA(x,y)=\begin{cases}
\sA x, & x=y,\\
0, & x\neq y.
\end{cases}$$

\begin{defn}
A \emph{$Q$-relation} (or \emph{fuzzy relation}) $\phi:\sA\oto\sB$ between $Q$-subsets is a $\CQ$-distributor between discrete $\CQ$-categories $\sA$ and $\sB$.
\end{defn}

Explicitly, a $Q$-relation $\phi:\sA\oto\sB$ is a map $\phi:\sA_0\times\sB_0\to Q$ such that
$$\phi(x,y)\leq\sA x\wedge\sB y$$
for all $x\in\sA_0$ and $y\in\sB_0$. Each $Q$-relation $\phi:\sA\oto\sB$ has a dual $Q$-relation $\phi^{\op}:\sB\oto\sA$ given by $\phi^{\op}(y,x)=\phi(x,y)$ for all $x\in\sA_0$ and $y\in\sB_0$.

Given $Q$-relations $\phi:\sA\oto\sB$ and $\psi:\sB\oto\sC$, the composition of $\psi$ and $\phi$
$$\psi\circ\phi:\sA\oto\sC$$
is defined by the composition of $\CQ$-distributors, i.e.,
$$\psi\circ\phi(x,z)=\bv_{y\in\sB_0}\psi(y,z)\&(\sB y\rd\phi(x,y))=\bv_{y\in\sB_0}(\psi(y,z)/\sB y)\&\phi(x,y).$$
This formula can be regarded as a many-valued reformulation of the statement that $x,z$ are related if there exists some $y$ in $\sB$ such that $x,y$ are related and $y,z$ are related.

\begin{defn}
Let $\phi:\sA\oto\sA$ be a $Q$-relation on a $Q$-subset $\sA$.
\begin{itemize}
\item[\rm (1)] $\phi$ is \emph{reflexive} if $\sA\leq\phi$.
\item[\rm (2)] $\phi$ is \emph{transitive} if $\phi\circ\phi\leq\phi$.
\item[\rm (3)] $\phi$ is \emph{symmetric} if $\phi=\phi^{\op}$.
\item[\rm (4)] $\phi$ is \emph{separated} if $\phi(x,x)\leq\phi(x,y)$ and $\phi(y,y)\leq\phi(y,x)$ implies $x=y$.
\end{itemize}
\end{defn}

\begin{defn}
A \emph{$Q$-preorder} (or \emph{fuzzy preorder}) on a $Q$-subset $\sA$ is a reflexive and transitive $Q$-relation $\bbA:\sA\oto\sA$.
\end{defn}

Explicitly, a $Q$-preorder on a $Q$-subset $\sA$ is a map $\bbA:\sA_0\times\sA_0\to Q$ such that for all $x,y,z\in\sA_0$,
\begin{itemize}
\item[\rm (i)] $\bbA(x,y)\leq\sA x\wedge\sA y$,
\item[\rm (ii)] $\sA x\leq\bbA(x,x)$,
\item[\rm (iii)] $\bbA(y,z)\&(\sA y\rd\bbA(x,y))=(\bbA(y,z)/\sA y)\&\bbA(x,y)\leq\bbA(x,z)$.
\end{itemize}

It can be inferred from (i) and (ii) that $\bbA(x,x)=\sA x$ for all $x\in\sA_0$. Therefore, a $Q$-preordered $Q$-subset can be described as a pair $(A,\bbA)$, where $A$ is a (crisp) set and $\bbA:A\times A\to Q$ is a map, such that for all $x,y,z\in A$,
\begin{itemize}
\item[\rm (1)] $\bbA(x,y)\leq\bbA(x,x)\wedge\bbA(y,y)$,
\item[\rm (2)] $\bbA(y,z)\&(\bbA(y,y)\rd\bbA(x,y))=(\bbA(y,z)/\bbA(y,y))\&\bbA(x,y)\leq\bbA(x,z)$.
\end{itemize}

A $Q$-preordered $Q$-subset $(A,\bbA)$ is \emph{separated} if the $Q$-relation $\bbA$ is separated. It is easy to see that $(A,\bbA)$ is separated if and only if $\bbA(x,x)=\bbA(y,y)=\bbA(x,y)=\bbA(y,x)$ implies $x=y$.

Compared with Definition \ref{Q_category}, we obtain that $Q$-preordered $Q$-subsets are exactly categories enriched over the quantaloid $\CQ$, and separated $Q$-preordered $Q$-subsets are skeletal $\CQ$-categories. We abbreviate the pair $(A,\bbA)$ to $\bbA$ and write $\bbA_0=A$ if there is no confusion. For each $Q$-preordered $Q$-subset $\bbA$, we denote by $|\bbA|$ the underlying $Q$-subset of $\bbA$, with $|\bbA|_0=\bbA_0$ and $|\bbA|x=\bbA(x,x)$ for all $x\in\bbA_0$.

A $Q$-preordered $Q$-subset $\bbA$ is \emph{global} if the underlying $Q$-subset $|\bbA|$ is global, i.e., $\bbA(x,x)=1$ for all $x\in\bbA_0$. A global $Q$-preordered $Q$-subset is exactly a $Q$-preordered set. Thus, our discussion in Section \ref{Preordered_sets} is a special case of this section when $Q$ is a divisible unital quantale.

\begin{defn}
An order-preserving map $F:\bbA\to\bbB$ between $Q$-preordered $Q$-subsets is a $\CQ$-functor between $\CQ$-categories $\bbA$ and $\bbB$.
\end{defn}

Given a $Q$-preordered $Q$-subset $\bbA$, a \emph{potential lower set} of $\bbA$ is a contravariant presheaf $\mu\in\PA$, and a \emph{potential upper set} of $\bbA$ is a covariant presheaf $\lam\in\PdA$. The $Q$-preordered $Q$-subsets $\PA$ and $\PdA$ are called respectively the \emph{$Q$-powerset} and the \emph{dual $Q$-powerset} of $\bbA$. We point out that the underlying $Q$-subset $|\PA|$ of $\PA$ is given by the following data:
\begin{itemize}
\item[\rm (1)] $|\PA|_0$ is the set of potential lower sets $\mu\in\PA$.
\item[\rm (2)] The type map $|\PA|:|\PA|_0\to Q$ sends each $\mu\in\PA$ to $t\mu$.
\end{itemize}
Dually, the underlying $Q$-subset $|\PdA|$ of $\PdA$ is given by the following data:
\begin{itemize}
\item[\rm (1)] $|\PdA|_0$ is the set of potential upper sets $\lam\in\PdA$.
\item[\rm (2)] The type map $|\PdA|:|\PdA|_0\to Q$ sends each $\lam\in\PdA$ to $t\lam$.
\end{itemize}

In particular, Let $\sA$ be a $Q$-subset. Consider $\sA$ as a discrete $\CQ$-category, then the $\CQ$-category $\CP\sA$ is called the \emph{$Q$-powerset} of $\sA$. In elementary words, $(\CP\sA)_0$ consists of pairs $(\mu,a)\in Q^{\sA_0}\times\sA_0$ satisfying
$$\forall x\in\sA_0,\mu(x)\leq\sA x\wedge a,$$
and
$$\CP\sA((\mu,a),(\lam,b))=a\wedge b\wedge\bw_{x\in\sA_0}\lam(x)/(a\rd\mu(x))$$
for all $(\mu,a),(\lam,b)\in(\CP\sA)_0$.
Dually, the \emph{dual $Q$-powerset} $\CPd\sA$ of $\sA$ has the same underlying $Q$-subset $|\CPd\sA|=|\CP\sA|$, but
$$\CPd\sA((\mu,a),(\lam,b))=a\wedge b\wedge\bw_{x\in\sA_0}(\lam(x)/b)\rd\mu(x)$$
for all $(\mu,a),(\lam,b)\in(\CPd\sA)_0$.

A $Q$-preordered $Q$-subset $\bbA$ is \emph{complete} if it is a complete $\CQ$-category. For each $Q$-preordered $Q$-subset $\bbA$, $\PA$ and $\PdA$ are both separated and complete.

\section{Formal concept analysis on fuzzy sets}

Formal concept analysis \cite{Carpineto2004,Davey2002,Ganter2005,Ganter1999} is a useful tool for qualitative data analysis. A formal context is a triple $(A,B,R)$, where $A,B$ are sets and $R\subseteq A\times B$ is a relation between $A$ and $B$. Elements in $A$ are interpreted as \emph{objects} and those in $B$ as \emph{attributes}. Each relation $R\subseteq A\times B$ induces a pair of operators $\uR:{\bf 2}^A\to{\bf 2}^B$ and $\dR:{\bf 2}^B\to{\bf 2}^A$ as follows:
$$\uR(U)=\{y\in B:\forall x\in U,(x,y)\in R\},$$
$$\dR(V)=\{x\in A: \forall y\in V,(x,y)\in R\}.$$
This pair of operators is a contravariant Galois connection in the sense that
$$U\subseteq\dR(V)\iff V\subseteq\uR(U).$$
This Galois connection plays a fundamental role in formal concept analysis.

A formal concept of a formal context $(A,B,R)$ is a pair $(U,V)\in {\bf 2}^A\times {\bf 2}^B$, where $U$ is the \emph{extent} and $V$ the \emph{intent}, such that $U=\dR(V)$ and $V=\uR(U)$. The fundamental theorem of formal concept analysis states that the formal concepts of a formal context form a complete lattice (called the \emph{formal concept lattice}) and every complete lattice is the formal concept lattice of some formal context.

From the viewpoint of category theory, a formal context $(A,B,R)$ is just a ${\bf 2}$-distributor between discrete ${\bf 2}$-categories $A$ and $B$, while the contravariant Galois connection $(\uR,\dR)$ is exactly an Isbell adjunction between ${\bf 2}^A$ and $({\bf 2}^B)^{\op}$. That is to say, formal concept analysis is essentially a theory on distributors between categories.

As demonstrated in Section \ref{Preordered_sets} and \ref{Preordered_fuzzy_sets}, fuzzy relations between (crisp) sets can be viewed as distributors between quantale-enriched categories, while fuzzy relations between fuzzy sets and can be viewed as distributors between quantaloid-enriched categories. Thus, the theory of formal concept analysis can be generalized to a theory on fuzzy relations between crisp sets \cite{Bvelohlavek2001,Bvelohlavek2002,Bvelohlavek2004,Georgescu2004,Lai2009,Popescu2004,Shen2013}, and it is possible to extend the theory further to the realm of fuzzy relations between fuzzy sets.

From now on we assume that $Q$ is a divisible unital quantale. A \emph{fuzzy context} is a triple $(\sA,\sB,\phi)$, where $\sA$ and $\sB$ are $Q$-subsets, and $\phi:\sA\oto\sB$ is a $Q$-relation. Fuzzy contexts and infomorphisms constitute a category $\CQ$-$\Ctx$. It is easy to see that $\CQ$-$\Ctx$ is a subcategory of $\CQ$-$\Info$ containing only $\CQ$-distributors between discrete $\CQ$-categories. It is noteworthy to point out that $(\CQ\text{-}\Ctx)^{\op}=\CQ\text{-}\Ctx$.

We point out that if $\sA$ and $\sB$ are global, i.e., $\sA$ and $\sB$ are crisp sets, then the fuzzy context $(\sA,\sB,\phi)$ is a $Q$-context (also called fuzzy context in some literatures) in the sense of \cite{Bvelohlavek2004,Lai2009,Shen2013}. Therefore, formal concept analysis and rough set theory (introduced in the next section) on $Q$-contexts are special cases of our discussion here.

Each fuzzy context  $(\sA,\sB,\phi)$ gives rise to an Isbell adjunction
$$\uphi\dv\dphi:\CP\sA\rhu\CPd\sB.$$
Let $\CM_{\phi}(\sA,\sB)$ denote the $Q$-preordered $Q$-subset given by
\begin{itemize}
\item[\rm (1)] $|\CM_{\phi}(\sA,\sB)|_0$ consists of pairs $(\mu,\lam)\in(\CP\sA)_0\times(\CPd\sB)_0$ such that
               $$\lam=\uphi(\mu)\quad\text{and}\quad\mu=\dphi(\lam);$$
\item[\rm (2)] for each $(\mu,\lam)\in|\CM_{\phi}(\sA,\sB)|_0$,
               $$|\CM_{\phi}(\sA,\sB)|(\mu,\lam)=t\mu=t\lam;$$
\item[\rm (3)] for each $(\mu,\lam),(\mu',\lam')\in|\CM_{\phi}(\sA,\sB)|_0$,
               $$\CM_{\phi}(\sA,\sB)((\mu,\lam),(\mu',\lam'))=\CP\sA(\mu,\mu')=\CPd\sB(\lam,\lam').$$
\end{itemize}
Then $\CM_{\phi}(\sA,\sB)$ is called the formal concept lattice of the fuzzy context $(\sA,\sB,\phi)$, and the $Q$-subset $|\CM_{\phi}(\sA,\sB)|$ is called the fuzzy set of formal concepts of $(\sA,\sB,\phi)$. Elements $(\mu,\lam)$ in $|\CM_{\phi}(\sA,\sB)|_0$ are called \emph{potential formal concepts} of $(\sA,\sB,\phi)$, and $|\CM_{\phi}(\sA,\sB)|(\mu,\lam)$ is interpreted as the degree that $(\mu,\lam)$ is a formal concept.

The following conclusions are direct consequences of the results in Chapter \ref{Isbell_adjunction}. In particular, we extend the fundamental theorem of formal concept analysis to the fuzzy setting.

\begin{thm}
For any fuzzy context $(\sA,\sB,\phi)$, $\CM_{\phi}(\sA,\sB)$ is a separated and complete $Q$-preordered $Q$-subset. Conversely, each separated and complete $Q$-preordered $Q$-subset $\bbA$ is isomorphic to the formal concept lattice $\CM_{\phi}(\sA,\sB)$ of some fuzzy context $(\sA,\sB,\phi)$.
\end{thm}

The category $\CQ$-$\CCat$ consists of the separated and complete $Q$-preordered $Q$-subsets as objects, and the left adjoint order-preserving maps as morphisms.

\begin{prop}
Let $\bbX$ be a separated and complete $Q$-preordered $Q$-subset. The following conditions are equivalent:
\begin{itemize}
\item[\rm (1)] $\bbX$ is isomorphic to the formal concept lattice $\CM_{\phi}(\sA,\sB)$ for some fuzzy context $(\sA,\sB,\phi)$.
\item[\rm (2)] $\bbX$ is isomorphic to a $\CQ$-closure system of $\CP\sA$ for some $Q$-subset $\sA$.
\item[\rm (3)] $\bbX$ is isomorphic to a quotient object of $\CP\sA$ for some $Q$-subset $\sA$ in the category $\CQ$-$\CCat$.
\end{itemize}
\end{prop}

\begin{proof}
Note that all the equivalences follow from Proposition \ref{closure_system_fca} except (2)${}\Lra{}$(1). To this end, suppose that $\bbX$ is isomorphic to $C(\CP\sA)$ for some $Q$-subset $\sA$ and $\CQ$-closure operator $C:\CP\sA\to\CP\sA$. Then similar to Step 1 of Theorem \ref{F_U_adjunction} one can deduce that $(\sA,\sB,\zeta_C)$ given by $\sB_0=C(\CP\sA)_0$, $\sB(\mu)=\CP\sA(\mu,\mu)$ and $\zeta_C(x,\mu)=\mu(x)$ is the desired fuzzy context.
\end{proof}

\begin{thm}
A separated and complete $Q$-preordered $Q$-subset $\bbX$ is isomorphic to the formal concept lattice $\CM_{\phi}(\sA,\sB)$ of some fuzzy context $(\sA,\sB,\phi)$ if and only if there exists a $\sup$-dense order-preserving map $F:\bbA\to\bbX$ and an $\inf$-dense order-preserving map $G:\bbB\to\bbX$ such that $\phi(x,y)=\bbX(Fx,Gy)$ for all $x\in\bbA_0$, $y\in\bbB_0$.
\end{thm}

The functoriality of the formal concept lattice is established as follows.

\begin{prop}
$\CM:\CQ\text{-}\Ctx\to\CQ\text{-}\CCat$ is a functor that sends each fuzzy context $(\sA,\sB,\phi)$ to the formal concept lattice $\CM_{\phi}(\sA,\sB)$. This functor is obtained by restricting to domain of $\CM:\CQ\text{-}\Info\to\CQ\text{-}\CCat$ defined by Equation (\ref{M_def}) and identifying $\CM_{\phi}(\sA,\sB)$ with $\CM(\phi)$.
$$\bfig
\ptriangle/->`<-^{) }`<-/<800,500>[\CQ\text{-}\Info`\CQ\text{-}\CCat`\CQ\text{-}\Ctx;\CM``\CM]
\efig$$
\end{prop}

\section{Rough set theory on fuzzy sets}

Rough set theory \cite{Pawlak1982,Polkowski2002,Yao2004} is another important tool for qualitative data analysis. Given a formal context $(A,B,R)$, the relation $R\subseteq A\times B$ induces another pair of operators $\eR:{\bf 2}^A\to{\bf 2}^B$ and $\aR:{\bf 2}^B\to{\bf 2}^A$ as follows:
$$\eR(U)=\{y\in B:\exists x\in U,(x,y)\in R\},$$
$$\aR(V)=\{x\in A: \forall y\in B,(x,y)\in R\text{ implies }y\in V\}.$$
This pair of operators is a covariant Galois connection in the sense that
$$\eR(U)\subseteq V\iff U\subseteq\aR(V).$$
This Galois connection plays a fundamental role in rough set theory.

A property oriented concept of a formal context $(A,B,R)$ is a pair $(U,V)\in {\bf 2}^A\times {\bf 2}^B$, where $U$ is the \emph{object} and $V$ the \emph{property}, such that $U=\aR(V)$ and $V=\eR(U)$. From the viewpoint of category theory, the covariant Galois connection $(\eR,\aR)$ is exactly a Kan adjunction between ${\bf 2}^A$ and ${\bf 2}^B$. That is to say, rough set theory is also a theory on distributors between categories, thus can be generalized to a theory on fuzzy relations between crisp sets \cite{Duntsch2002,Lai2009,Yao2004}, and further to that on fuzzy relations between fuzzy sets.

Suppose that $Q$ is a divisible unital quantale. Each fuzzy context  $(\sA,\sB,\phi)$ gives rise to a Kan adjunction
$$\phi^*\dv\phi_*:\CP\sB\rhu\CP\sA.$$
Let $\CK_{\phi}(\sA,\sB)$ denote the $Q$-preordered $Q$-subset given by
\begin{itemize}
\item[\rm (1)] $|\CK_{\phi}(\sA,\sB)|_0$ consists of pairs $(\mu,\lam)\in(\CP\sA)_0\times(\CPd\sB)_0$ such that
               $$\lam=\phi_*(\mu)\quad\text{and}\quad\mu=\phi^*(\lam);$$
\item[\rm (2)] for each $(\mu,\lam)\in|\CK_{\phi}(\sA,\sB)|_0$,
               $$|\CK_{\phi}(\sA,\sB)|(\mu,\lam)=t\mu=t\lam;$$
\item[\rm (3)] for each $(\mu,\lam),(\mu',\lam')\in|\CK_{\phi}(\sA,\sB)|_0$,
               $$\CK_{\phi}(\sA,\sB)((\mu,\lam),(\mu',\lam'))=\CP\sA(\mu,\mu')=\CP\sB(\lam,\lam').$$
\end{itemize}
Then $\CK_{\phi}(\sA,\sB)$ is called the property oriented concept lattice of the fuzzy context $(\sA,\sB,\phi)$, and the $Q$-subset $|\CK_{\phi}(\sA,\sB)|$ is called the fuzzy set of property oriented concepts of $(\sA,\sB,\phi)$. Elements $(\mu,\lam)$ in $|\CK_{\phi}(\sA,\sB)|_0$ are called \emph{potential property oriented concepts} of $(\sA,\sB,\phi)$, and $|\CK_{\phi}(\sA,\sB)|(\mu,\lam)$ is interpreted as the degree that $(\mu,\lam)$ is a property oriented concept.

The following conclusions are direct consequences of the results in Chapter \ref{Kan_adjunction}.

\begin{prop}
For any fuzzy context $(\sA,\sB,\phi)$, $\CK_{\phi}(\sA,\sB)$ is a separated and complete $Q$-preordered $Q$-subset.
\end{prop}

\begin{prop}
If $Q$ is a Boolean algebra, then each separated and complete $Q$-preordered $Q$-subset $\bbA$ is isomorphic to the property oriented concept lattice $\CK_{\phi}(\sA,\sB)$ of some fuzzy context $(\sA,\sB,\phi)$.
\end{prop}

\begin{proof}
Follows immediately from Example \ref{Girard_quantaloid_example}(3) and Proposition \ref{G_V_adjunction}.
\end{proof}

\begin{prop}
Let $\bbX$ be a separated and complete $Q$-preordered $Q$-subset. The following conditions are equivalent:
\begin{itemize}
\item[\rm (1)] $\bbX$ is isomorphic to the property oriented concept lattice $\CK_{\phi}(\sA,\sB)$ for some fuzzy context $(\sA,\sB,\phi)$.
\item[\rm (2)] $\bbX$ is isomorphic to a $\CQ$-interior system of $\CP\sA$ for some $Q$-subset $\sA$.
\item[\rm (3)] $\bbX$ is isomorphic to a subobject of $\CP\sA$ for some $Q$-subset $\sA$ in the category $\CQ$-$\CCat$.
\end{itemize}
\end{prop}

\begin{proof}
For (2)${}\Lra{}$(1), suppose that $\bbX$ is isomorphic to $F(\CP\sA)$ for some $Q$-subset $\sA$ and $\CQ$-interior operator $F:\CP\sA\to\CP\sA$, then similar to Proposition \ref{interior_system_rst} one can deduce that $(\sA,\sB,\zeta_F)$ given by $\sB_0=F(\CP\sA)_0$, $\sB(\mu)=\CP\sA(\mu,\mu)$ and $\zeta_F(x,\mu)=\mu(x)$ is the desired fuzzy context.
\end{proof}

We present below the functoriality of the property oriented concept lattice.

\begin{prop}
$\CK:\CQ\text{-}\Ctx\to\CQ\text{-}\CCat$ is a functor that sends each fuzzy context $(\sA,\sB,\phi)$ to the property oriented concept lattice $\CK_{\phi}(\sA,\sB)$. This functor is obtained by restricting to domain of $\CK:(\CQ\text{-}\Info)^{\op}\to\CQ\text{-}\CCat$ defined by Equation (\ref{K_def}) and identifying $\CK_{\phi}(\sA,\sB)$ with $\CK(\phi)$.
$$\bfig
\ptriangle/->`<-^{) }`<-/<900,500>[(\CQ\text{-}\Info)^{\op}`\CQ\text{-}\CCat`\CQ\text{-}\Ctx=(\CQ\text{-}\Ctx)^{\op};\CK``\CK]
\efig$$
\end{prop}

\bibliographystyle{alphaabbr}
\cleardoublepage
\phantomsection
\addcontentsline{toc}{chapter}{Bibliography}
\bibliography{lili}

\cleardoublepage
\pagestyle{plain}
\phantomsection
\addcontentsline{toc}{chapter}{Acknowledgement}
\chapter*{Acknowledgement}

First of all, I would like to express my deepest gratitude to my supervisor, Professor Dexue Zhang. During the six years under his supervision, I received careful guidance and constant encouragement both in life and study from Professor Zhang. His illuminating suggestions have significantly improved the quality of this dissertation.

Second, I am sincerely grateful to Professor Hui Kou, who not only taught me knowledge, but also instructed and encouraged me with valuable discussions. I am also grateful to Professor Shuguo Zhang, who inspired me with his persistence in a research field.  Meanwhile, I would like to thank all the teachers who taught me lessons when I was studying in Sichuan University.

Moreover, I wish to thank Professor Liangang Peng and Professor Bing Hu for their guidance and support for my career. I am grateful to the seniors Hongliang Lai, Yingming Chai, Lingqiang Li, Piwei Chen, Yayan Yuan, Jiachao Wu, Chengling Fang, Hui Wang and Wei Zhang for their helpful discussions and communications with me both in life and study. I would like to thank my friends and schoolmates Qiang Pu, Haoran Zhao, Liping Xiong, Zhongwei Yang and Chengyong Du for sharing the graduate study. I would like to thank the juniors Yuanye Tao, Wei Li, Jialiang He, Zhiqiang Shen and Junsheng Qiao who are important members of our group. Furthermore, my thanks go out to faculty members, staff members and all the graduate students in the School of Mathematics, for their friendship and help in different ways.

Finally, I would like to express my gratitude for the support from my wife Xiaojuan Zhao, my parents and every member of my family.

\end{document}